\documentclass[reqno]{amsart}

\usepackage{tikz}
\usetikzlibrary{cd}
\usepackage{graphicx}
\graphicspath{{images/}}
\usepackage{amsmath}
\usepackage{amsfonts}
\usepackage{amssymb}
\usepackage{amsthm}
\usepackage{hyperref} %Use the option [backref=page] for back references to citations
\hypersetup{pdftoolbar=False}
\usepackage{mathrsfs}
\usepackage[boxsize=.3 em]{ytableau}

\makeatletter
\newcommand{\addresseshere}{%
  \enddoc@text\let\enddoc@text\relax
}

\numberwithin{equation}{section}
\numberwithin{figure}{section}
\numberwithin{table}{section}

\input{toclayout}			%adds indentation to subsection in the table of contents

\newcommand{\tprod}[1][]{\textstyle\prod#1\displaystyle}

\newcommand{\ali}[1]{\begin{align*} #1 \end{align*}}
		\newcounter{romi}
\newenvironment{romanize}{\begin{list}{\addtocounter{romi}{1}\hspace{-3em}\textbf{(\roman{romi})}}{}}{\setcounter{romi}{0}\end{list}}
		\newcounter{alphi}

\newcommand{\qand}{\quad\text{and}\quad}												% to write ``and'' in the middle of an equation
\theoremstyle{definition}
\newtheorem{thm}{Theorem}[section]
\newtheorem{prop}[thm]{Proposition}
\newtheorem{lem}[thm]{Lemma}
\newtheorem{cor}[thm]{Corollary}
\newtheorem{df}[thm]{Definition}
\newtheorem{rem}[thm]{Remark}
\newtheorem{exa}[thm]{Example}

\newtheorem{algo}[thm]{Algorithm}
\newtheorem{conj}[thm]{Conjecture}
\newenvironment{ex}{\begin{exa}}{\hfill$\lozenge$\end{exa}}	%adds a diamond at the end of the example environment

\newcommand{\Nzero}{\mathbb{N}_0}			%natural numbers WITH 0
\newcommand{\Z}{\mathbb{Z}}

\newcommand{\C}{\mathbb{C}}

\newcommand{\lan}{\langle}
\newcommand{\ran}{\rangle}
\newcommand{\llan}{\lan\!\lan}
\newcommand{\rran}{\ran\!\ran}

\newcommand{\Id}{\mathrm{Id}}

\newcommand{\pdif}{\partial}

\newcommand{\op}{\oplus}

\newcommand{\cA}{\mathcal{A}}

\newcommand{\cE}{\mathcal{E}}
\newcommand{\cF}{\mathcal{F}}
\newcommand{\fg}{\mathfrak{g}}

\newcommand{\fgch}{\mathfrak{g}^\vee}

\newcommand{\Gch}{G^\vee}

\newcommand{\cI}{\mathcal{I}}

\newcommand{\bm}{\mathbf{m}}
\newcommand{\bn}{\mathbf{n}}

\newcommand{\bP}{\mathbb{P}}
\newcommand{\Pch}{P^\vee}

\newcommand{\cR}{\mathcal{R}}

\newcommand{\ft}{\mathfrak{t}}

\newcommand{\bT}{\mathbb{T}}
\newcommand{\fu}{\mathfrak{u}}

\newcommand{\cU}{\mathcal{U}}
\newcommand{\cW}{\mathcal{W}}

\newcommand{\bx}{\mathbf{x}}

\newcommand{\bX}{\mathbb{X}}
\newcommand{\cX}{\mathcal{X}}
\newcommand{\cZ}{\mathcal{Z}}

\newcommand{\al}{\alpha}
\newcommand{\be}{\beta}

\newcommand{\Ga}{\Gamma}
\newcommand{\de}{\delta}

\newcommand{\la}{\lambda}
\newcommand{\si}{\sigma}
\newcommand{\Si}{\Sigma}

\newcommand{\vp}{\varpi}

\newcommand{\LGT}[1]{#1}												%Lie group types
\newcommand{\LGA}{\LGT{A}}											%Type A
\newcommand{\LGB}{\LGT{B}}											%Type B
\newcommand{\LGC}{\LGT{C}}											%Type C
\newcommand{\LGD}{\LGT{D}}										%Type D
\newcommand{\LGE}{\LGT{E}}											%Type E
\newcommand{\PSO}{\mathrm{PSO}}						%Projective special orthogonal group (B or D)
\newcommand{\Spin}{\mathrm{Spin}}					%Spin group (B or D)
\newcommand{\Sp}{\mathrm{Sp}}								%Symplectic group (C)
\newcommand{\Exc}{\LGE}													%Exceptional groups (E)
\newcommand{\SC}{\mathrm{sc}}								%simply-connected (group of type ~)
\newcommand{\AD}{\mathrm{ad}}								%adjoint (group of type ~)

\newcommand{\fsl}{\mathfrak{sl}}								%Special linear algebra (type A)					%need the f to prevent a clash with `slanted'
\newcommand{\PS}{\bP}													%Projective space (associated to a vector space)
\newcommand{\Gr}{\mathrm{Gr}}									%Grassmannian
\newcommand{\OG}{\mathrm{OG}}								%Orthogonal Grassmannian
\newcommand{\ECHS}[1][n]{\Exc_{#1}^{\SC}/\P_{#1}}		%Exceptional cominuscule homogeneous space

\newcommand{\cayley}{\mathbb{OP}^2}				%Cayley Plane
\newcommand{\freudenthal}{\Exc_7^\SC/\P_7}	%Freudenthal variety

\newcommand{\X}{X}																	%Original homogeneous space X
\newcommand{\G}{G}																	%Group G in X = G/P
\renewcommand{\P}{P}																%Parabolic subgroup P in X = G/P
\newcommand{\torus}{T}														%Torus of G
\newcommand{\unip}{U_+}													%unipotent subgroup (raising)
\newcommand{\unim}{U_-}													%unipotent subgroup (lowering)
\newcommand{\borelp}{B_+}											%Borel subgroup (raising)
\newcommand{\borelm}{B_-}											%Borel subgroup (lowering)
\newcommand{\g}{\fg}																%Lie algebra g of G
\newcommand{\up}{\fu_+}													%Lie subalgebra u_+ of g generated by e_i
\newcommand{\um}{\fu_-}													%Lie subalgebra u_- of g generated by f_i
\newcommand{\cartan}{\ft}													%Cartan subalgebra t of g
\newcommand{\parab}{\mathfrak{p}}					%Lie algebra p of P
\newcommand{\Che}{e}															%Chevalley basis element e_i
\newcommand{\Chf}{f}																%Chevalley basis element f_i
\newcommand{\Chh}{h}															%Element h_i=[e_i,f_i]
\newcommand{\ParabolicIndices}{I_P}					%Indices of the f_i that are generators of P
\newcommand{\IP}{I^P}															%Complement of these indices
\newcommand{\dG}{\Gch}													%Group G^ Langlands dual to G
\newcommand{\dP}{\Pch}														%Langlands dual to parabolic group
\newcommand{\dtorus}{T^\vee}									%Langlands dual maximal torus
\newcommand{\invdtorus}{T^P}									%Torus invariant under W_P
\newcommand{\dunip}{U^\vee_+}								%Langlands dual unipotent subgroup (raising)
\newcommand{\dunim}{U^\vee_-}								%Langlands dual unipotent subgroup (lowering)
\newcommand{\dborelp}{B^\vee_+}						%Langlands dual Borel subgroup (raising)
\newcommand{\dborelm}{B^\vee_-}						%Langlands dual Borel subgroup (lowering)
\newcommand{\dg}{\fgch}														%Lie algebra associated to Langlands dual of G
\newcommand{\dup}{\fu^\vee_+}									%Lie subalgebra u^_+ of g^ generated by e^_i
\newcommand{\dum}{\fu^\vee_-}									%Lie subalgebra u^_- of g^ generated by f^_i
\newcommand{\dcartan}{\ft^\vee}								%Cartan subalgebra t^ of g^
\newcommand{\dparab}{\mathfrak{p}^\vee}	%Lie subalgebra p^ associated to P^
\newcommand{\dChe}{e^\vee}											%Chevalley generator e^_i
\newcommand{\dChf}{f^\vee}												%Chevalley generator f^_i
\newcommand{\dChh}{h^\vee}											%Chevalley generator h^_i
\newcommand{\dx}{x^\vee}													%One-parameter subgroup generated by e^_i
\newcommand{\dy}{y^\vee}													%One-parameter subgroup generated by f^_i
\newcommand{\dudrt}[1][\drt]{\fu^\vee_{#1}}			%(Co-)root space in g^ spanned by a coroot
\newcommand{\wPdChe}{\dChe_{\wP}}							%(Upper) Chevalley generator associated to positive coroot mapped to negative by wP
\newcommand{\udG}{\widetilde{G}^\vee}						%Universal cover of G^
\newcommand{\udP}{\widetilde{P}^\vee}							%Universal cover of Langlands dual parabolic group P
\newcommand{\udtorus}{\widetilde{T}^\vee}				%Universal cover of Langlands dual maximal torus
\newcommand{\udborelp}{\widetilde{B}^\vee_+}	%Universal cover of Langlands dual Borel subgroup (raising)
\newcommand{\udborelm}{\widetilde{B}^\vee_-}	%Universal cover of Langlands dual Borel subgroup (lowering)
\newcommand{\dgdual}{(\dg)^*}												%Dual of the Lie algebra g^
\newcommand{\dChedual}[1][_i]{(\dChe#1)^*}			%Dual to Chevalley basis element e^_i
\newcommand{\dChfdual}[1][_i]{(\dChf#1)^*}				%Dual to Chevalley basis element f^_i
\newcommand{\dUEAp}{\cU^\vee_+}									%Universal enveloping algebra of u^_+
\newcommand{\dCUEAp}{\widehat{\cU}^\vee_+}	%Complete universal enveloping algebra of u^_+
\newcommand{\dUEAm}{\cU^\vee_-}									%Universal enveloping algebra of u^_-
\newcommand{\dCUEAm}{\widehat{\cU}^\vee_-}	%Complete universal enveloping algebra of u^_-
\newcommand{\dCUEApDual}{\bigl(\dCUEAp\bigr)^*_\mathrm{gr}} %Graded dual of dCUEAp
\newcommand{\roots}{\Phi}												%Roots of g and G
\newcommand{\proots}{\Pi}									%Positive roots
\newcommand{\nroots}{\Pi_-}										%Negative roots
\newcommand{\sroots}{\Delta}										%Simple roots
\newcommand{\rt}{\al}																%One of the roots of G
\newcommand{\sr}{\al}																%One of the simple roots alpha_i
\newcommand{\fwt}[1][i]{{\vp_{#1}}}						%Fundamental weight omega
\newcommand{\droots}{\Phi^\vee}							%Coroots 
\newcommand{\dfwt}[1][i]{{\vp^\vee_{#1}}}	%Fundamental coweight
\newcommand{\pdroots}{\Pi^\vee}				%Positive coroots
\newcommand{\ndroots}{\Pi^\vee_-}					%Negative coroots
\newcommand{\sdroots}{\Delta\!^\vee}			%Simple coroots
\newcommand{\pdrootsP}{\pdroots_{\P}}		%Positive coroots mapped to negative coroots by w_P
\newcommand{\wPpdroots}{\pdroots_{\wP}}	%Positive coroots mapped to negative coroots by w^P
\newcommand{\sdr}{\al^\vee}										%One of the simple coroots alpha^_i
\newcommand{\drt}{\al^\vee}											%One of the coroots of G
\newcommand{\drtb}{\be^\vee}											%Another of the coroots of G
\newcommand{\drtzero}{\drt_0}									%The longest coroot of G
\newcommand{\wPdrt}{{\be^\vee_{P}}}				%One of the positive coroots mapped to negatives by w^P
\newcommand{\weyl}{W}																%Weyl group W of G^
\newcommand{\wo}{w_0}																%Longest element of W
\newcommand{\weylp}{W_P}														%Weyl group W_P of P^
\newcommand{\wop}{w_P}															%Longest element of W_P
\newcommand{\cosets}{W^P}													%Set of minimal coset representatives of W/W_P
\newcommand{\wP}{w^P}																%Minimal coset representative of w_0
\newcommand{\wPinv}{(\wP)^{-1}}										%Shorthand for the inverse of w^P
\newcommand{\wPrime}{w'}														%Minimal coset representative of w^P s_k w_P
\newcommand{\wPPrime}{w''}													%Minimal coset representative of w_P s_k
\newcommand{\ellwo}{{\ell_0}}										%Lenght of w_0
\newcommand{\ellwop}{{\ell_P}}									%Length of w_P
\newcommand{\ellwP}{\ell}															%Length of w^P
\newcommand{\ellwPrime}{{\ell'}}										%Length of w'
\newcommand{\ds}{\dot{s}}												%First element associated to a simple reflection in Langlands dual group
\newcommand{\bs}{\bar{s}}												%Second element
\newcommand{\dw}{\dot{w}}											%Generalization to arbitrary w
\newcommand{\bw}{\bar{w}}											%Generalization to arbitrary w
\newcommand{\bwo}{\bw_0}											%Applied to w_0
\newcommand{\bwop}{\bw_P}										%Applied to w_P
\newcommand{\dwP}{\dw^P}											%Applied to w^P
\newcommand{\dwPinv}{(\dwP)^{-1}}			%Inverse of dwP
\newcommand{\mX}{X^\vee}												%Mirror of X
\newcommand{\cmX}{\bX^\vee}										%Compactification of X
\newcommand{\dRichard}{\cR^\vee}							%Richardson variety (inside G^/B^_-)
\newcommand{\Jac}[1]{\lan\pdif\pot_{#1}\ran}					%The Jacobian
\newcommand{\decomps}{\cZ^\vee_P}					%Elements of G^ with a double decomposition U+TwPU- = B-w0
\newcommand{\opendecomps}{\cZ^\circ_P}	%Open Bruhat cell inside Z
\newcommand{\RisoZ}{\Psi}													%Isomorphism X×T->Z
\newcommand{\dunimP}{U_-^P}										%Elements in U_- such that there is a unique decomposition in Z
\newcommand{\dunipP}{U_+^P}										%Transpose of the above set
\newcommand{\opendunim}{U^\circ_-}					%Open Bruhat cell inside U_-^P
\newcommand{\opendunip}{U^\circ_+}					%Open Bruhat cell inside U_+^P
\newcommand{\pot}{\cW}														%Potential W
\newcommand{\Lie}{\mathrm{Lie}}									%Lie-theoretic (potential)
\newcommand{\can}{\mathrm{can}}								%Canonical (potential)
\newcommand{\SumOfEs}{\cE^*}
\newcommand{\SumOfFs}{\cF^*}
\newcommand{\potZ}{\pot_{\decomps}}						%Rietsch's potential on Z
\newcommand{\potZo}{\pot_{\opendecomps}}		%Restriction of the potential on Z to the Bruhat cell
\newcommand{\wPrimeSubExp}{\cI}									%Subexpressions of w' inside w^P
\newcommand{\dfwtrep}[1][i]{V(\dfwt[#1])}					%Fundamental weight representation
\newcommand{\dfwtrepdual}[1][i]{\dfwtrep[#1]^*}		%Dual of fundamental weight representation
\newcommand{\Pdfwtrepdual}[1][i]{\PS\bigl(\dfwtrepdual[#1]\bigr)}
\newcommand{\PluckerRep}[1][\DynkinSymmetry(k)]{V(\dfwt[#1])}
\newcommand{\PluckerRepAlt}[1][k]{V(-\wo\cdot\dfwt[#1])}
\newcommand{\PluckerDualRep}[1][k]{\PluckerRep^*}
\newcommand{\PPluckerDualRep}[1][k]{\PS\bigl(\PluckerDualRep\bigr)}
\newcommand{\hwt}[1][i]{v_{\dfwt[#1]}^+}					%Highest weight vector
\newcommand{\minwt}{\la^\vee}											%Arbitrary minuscule weight lambda
\newcommand{\minrep}{V(\minwt)}										%Minuscule representation with highest weight lambda
\newcommand{\minrepwts}{M(\minwt)}							%Weights of the representation V(la)
\newcommand{\wt}{\mu^\vee}													%Arbitrary weight mu
\newcommand{\wtu}[1][u]{\mu^\vee_{#1}}							%Weight obtained by acting with u in W on the highest weight
\newcommand{\wtv}[1]{v_{#1}}													%Basis vector of given weight
\newcommand{\wtvmu}{v_{\mu^\vee}}	              %Basis vector of weight mu
\newcommand{\ACS}{\cA}										%A-cluster structure
\newcommand{\XCS}{\cX}										%X-cluster structure
\newcommand{\seed}{\Si}										%Seed for cluster structure
\newcommand{\ACVars}{S_{\!\ACS}}					%A-cluster variables reachable from a given seed
\newcommand{\CTorusSeed}[1][\seed]{\bT_{\!#1}}	%Torus associated to a given seed
\newcommand{\CTorus}{\bT}									%Union of seed tori
\newcommand{\scs}{\sigma}										%Schubert class sigma
\newcommand{\loc}{\mathrm{loc}}							%(Quantum cohomology) localized at the quantum parameters
\newcommand{\rk}{\mathop{\mathrm{rk}}\nolimits}	%rank (over Z)
\newcommand{\height}{\mathrm{ht}}						%Height of a root (lattice element)
\newcommand{\ad}{\mathop{\mathrm{ad}}\nolimits}  %adjoint action of Lie algebra
\newcommand{\DynkinSymmetry}{\si}				%Symmetry of the Dynkin diagram
\newcommand{\p}{p}																			%(Generalized) Plücker coordinates
\newcommand{\pGLS}{p^\text{GLS}}						%Dual PBW basis defined by Geiss-Leclerc-Schroer
\newcommand{\phiGLS}{\phi}						%Cluster variable defined by Geiss-Leclerc-Schroer
\newcommand{\woExch}{e(\wo)}					%w_0-exchangeable indices
\newcommand{\ClusterQuiver}{\Ga_{\!\wo}}			%Quiver for GLS cluster structure
\newcommand{\minor}{\Delta}									%Generalized minor (defined by Fomin-Zelevinsky)
\newcommand{\T}{T}																	%Transpose
\calclayout

 \newcommand{\marginbox}[1]{}

\begin{document}
\title[Towards Landau-Ginzburg models for cominuscule $\G/\P$ via cominuscule $\ECHS$]{Towards Landau-Ginzburg models for cominuscule Spaces via the exceptional cominuscule family}
\author{Peter Spacek}
\address{Technische Universit\"at Chemnitz, Germany}
\email{peter.spacek@math.tu-chemnitz.de}

\author{Charles Wang}
\address{Harvard University, USA}
\email{cmwang@math.harvard.edu}
\date{\today}

\begin{abstract}
We present projective Landau-Ginzburg models for the exceptional cominuscule homogeneous spaces $\cayley=\ECHS[6]$ and $\ECHS[7]$, known respectively as the Cayley plane and Freudenthal variety. These models are defined on the complement $\mX_\can$ of an anti-canonical divisor of the ``Langlands dual homogeneous spaces'' $\cmX=\dP\backslash\dG$ in terms of generalized Pl\"ucker coordinates, analogous to the canonical models defined for Grassmannians, quadrics and Lagrangian Grassmannians in \cite{Marsh_Rietsch_Grassmannians,Pech_Rietsch_Williams_Quadrics,Pech_Rietsch_Lagrangian_Grassmannians}. We prove that these models for the exceptional family are isomorphic to the Lie-theoretic mirror models defined in \cite{Rietsch_Mirror_Construction} using a restriction to an algebraic torus, also known as the \emph{Lusztig torus}, as proven in \cite{Spacek_LP_LG_models}. We also give a cluster structure on $\C[\mX]$, prove that the Pl\"ucker coordinates form a Khovankii basis for a valuation defined using the Lusztig torus, and compute the Newton-Okounkov body associated to this valuation. Although we present our methods for the exceptional types, they generalize immediately to the members of other cominuscule families.
%
%In this article we give Pl\"ucker coordinate Landau-Ginzburg models $(\mX_\can,\pot_\can)$ for the cominuscule homogeneous spaces $\mathbb{X}$ of types $E_6$ and $E_7$, respectively the Cayley plane $\cayley=\ECHS[6]$ and the Freudenthal variety $\freudenthal$. We show that our models are isomorphic to the corresponding ``Lie-theoretic'' Landau-Ginzburg models defined by Rietsch, and conjecture a general construction for any cominuscule homogeneous space. We furthermore describe Pl\"ucker coordinate cluster structures on these $\cmX$, and show that the Pl\"ucker coordinates give a Khovanskii basis and hence Newton-Okounkov body.
\end{abstract}

\maketitle
\vspace{-1cm}
\setcounter{tocdepth}{2}
\tableofcontents
\section{Introduction}\label{sec:Introduction}
Let $\X$ be a projective homogeneous space $\X=\G/\P$, for a simple and simply-connected $\G$ and parabolic subgroup $\P$. The (small) quantum cohomology of such $\X$ has a long history of research, starting initially with the consideration of Grassmannians by Witten in \cite{Witten_The_Verlinde_algebra_and_the_cohomology_of_Gr}---see \cite{Fulton_on_qH_of_hom_spaces} for a survey. In this article, we will be considering mirror symmetry statements regarding the (small) quantum cohomology ring of $\X$ via a \emph{Landau-Ginzburg model}. In this context, a Landau-Ginzburg model consists of a mirror space $\mX$ and a superpotential $\pot_q$ which is a regular function on $\mX$ such that the \emph{Jacobi ring} of $\pot_q$ is isomorphic to the (small) quantum cohomology ring of $\X$ with localized quantum parameters, i.e.\vspace{-.25em}
\[
qH^*(\X)_{\loc} \cong \C[\mX]/\Jac{q}.
\vspace{-.25em}
\]
In \cite{Rietsch_Mirror_Construction}, a Lie-theoretic construction of a Landau-Ginzburg model is given and shown to be isomorphic to Peterson's presentation of the small quantum cohomology of general projective homogeneous spaces as given in \cite{Peterson}. The superpotential of this model is defined completely intrinsically, using natural maps on the Langlands dual Lie group $\dG$, namely the dual maps associated to a set of Chevalley generators on the Lie algebra $\dg$ of $\dG$.

In practice it is more convenient to have a Landau-Ginzburg model whose superpotential is expressed in terms of coordinates on the mirror space. For example, such LG-models are used in \cite{Pech_Rietsch_Williams_Quadrics} and \cite{Castronovo_Fukaya_category_of_Grassmannians} to prove mirror conjectures for certain homogeneous spaces on the level of $D$-modules and homological algebra, respectively, and in \cite{Rietsch_Williams_NO_bodies_cluster_duality_and_mirror_symmetry_for_Grassmannians} to construct a family of polytopes characterizing toric degenerations of Grassmannians; we will remark more on this towards the end of the section. This is especially the case for homogeneous spaces that have natural coordinates in the form of (generalized) Pl\"ucker coordinates. Hence, the results in \cite{Rietsch_Mirror_Construction} were considered for a number of families of homogeneous spaces that were particularly accessible, namely Grassmannians \cite{Marsh_Rietsch_Grassmannians}, quadrics \cite{Pech_Rietsch_Williams_Quadrics} and Lagrangian Grassmannians \cite{Pech_Rietsch_Lagrangian_Grassmannians}. These families are all \emph{cominuscule} homogeneous spaces, which implies that the Pl\"ucker coordinates on the Langlands dual homogeneous space $\cmX=\dP\backslash\dG$ can be defined in terms of projections on a \emph{minuscule} representation. The structure of these representations is particularly well-understood---see e.g.~\cite{Green_Reps_from_polytopes} or \cite{Geck_Minuscule_weights_and_Chevalley_groups}---and for the families mentioned above the representations are classical matrix representations.

In this paper, we consider the exceptional cominuscule family: the Cayley plane, $\ECHS[6]$ (often written as $\cayley$), and the Freudenthal variety, $\ECHS[7]$, where $\Exc_n^\SC$ denotes the simply-connected Lie group of type $\LGE$ and rank $n$. Although less studied than the classical homogeneous spaces, the geometry of the exceptional cominuscule family is still well-known---see e.g.~\cite{Iliev_Manivel_Chow_ring_of_the_Cayley_plane} for an overview and calculation of the Chow ring of the Cayley plane. Moreover, the ring structure of the quantum cohomologies of the exceptional family has been described in detail in \cite{CMP_Quantum_cohomology_of_minuscule_homogeneous_spaces}; there the quantum cohomologies of all cominuscule homogeneous spaces are presented in a uniform way. However, to our knowledge, the only known Landau-Ginzburg models available until now are Rietsch's Lie-theoretic model and the first author's Laurent polynomial expression for the model on an algebraic torus, sometimes called the \emph{Lusztig torus}, as given in \cite{Spacek_LP_LG_models}, neither of which give expression in terms of coordinates on the mirror space. 

We present Pl\"ucker coordinate Landau-Ginzburg models $(\mX_\can,\pot_\can)$ for the exceptional cominuscule family. The superpotentials are defined as\vspace{-.2em}
\ali{
\pot_\can &= \frac{\p_1}{\p_0} + \frac{\p_9'}{\p_8} + \frac{q_{13}}{q_{12}} + \frac{q_{17}}{q_{16}}+ \frac{q_{21}}{q_{20}}+ \frac{q_{25}}{q_{24}}+ q\frac{\p_5''}{\p_{16}}, 
\\
\pot_\can &= \frac{\p_1}{\p_0} + \frac{q_{19}}{q_{18}} + \frac{q_{28}}{q_{27}} + \frac{q_{37}}{q_{36}} +\frac{q_{37}'}{q_{36}'} + \frac{q_{46}}{q_{45}} + \frac{q_{55}}{q_{54}} + q\frac{p_{10}}{p_{27}}\vspace{-.3em}
}
for $\ECHS[6]$ and $\ECHS[7]$ respectively, where the $\p_i$ are Pl\"ucker coordinates and the $q_i$ are certain polynomials in the Pl\"ucker coordinates given in Sections \ref{sec:cayley-potential} and \ref{sec:freudenthal-potential}. The superpotentials are regular functions on $\cmX$ outside of an anticanonical divisor $D$ which defines the mirror space as $\mX_\can=\cmX\setminus D$; we follow the convention of \cite{Pech_Rietsch_Williams_Quadrics} in calling such models \emph{canonical Landau-Ginzburg models}. Our formulas give an interesting quantum Schubert calculus interpretation of our superpotentials above (see Remark \ref{rem:quantum-pieri} for more details): for a given term of a superpotential, its numerator is equal to the quantum product of $\sigma_1$ with its denominator, multiplied by the degree of its denominator as a polynomial in the Pl\"ucker coordinates. Because the sum of the degrees of the denominators is equal to the index $I$ in each case---$12$ for the Cayley plane and $18$ for the Freudenthal variety---we see that $\pot_\can$ is equal in some sense to $I$ copies of the hyperplane class, and thus represents the anticanonical class, e.g.~in the Jacobi ring. Similar observations were made in type $\LGA$ in \cite{Marsh_Rietsch_Grassmannians} and in type $\LGC$ in \cite{Pech_Rietsch_Lagrangian_Grassmannians}. We are working on a generalization of this observation that we believe holds for all cominuscule spaces, and this result will appear in a future article.

Our main tool for producing these formulas is the expression of generalized minors in terms of the coordinates of the Lusztig torus, which we summarize in Algorithms \ref{alg:torus-expansion} and \ref{alg:Plucker_torus_expansion}. One of our main results can be summarized as follows:

\begin{thm}[Theorems {\ref{thm:ExcFam_CanAndLie_Vars}} and {\ref{thm:ExcFam_CanAndLie_LGmodels}}]
For cominuscule $\ECHS$,  the canonical Landau-Ginzburg model is isomorphic to Rietsch's Lie-theoretic Landau-Ginzburg model.
\end{thm}
\noindent We will introduce the Lie-theoretic Landau-Ginzburg model in Section \ref{sec:Lie_theoretic_Mirror}.
%\begin{thm}[Theorem {\ref{thm:ExcFam_CanAndLie_Vars}}]
%For cominuscule $\ECHS$, the variety $\mX_\can$ is isomorphic to the open Richardson variety $\mX_\Lie=\dRichard_{\wop,\wo}$.
%\end{thm}
%
%\begin{thm}[Theorem {\ref{thm:ExcFam_CanAndLie_LGmodels}}]
%For cominuscule $\ECHS$, the pull-back of $\pot_\can$ under the isomorphism above equals $\pot_\Lie$. In particular, the Landau-Ginzburg model $(\mX_\can,\pot_\can)$ is isomorphic to Rietsch's Lie-theoretic Landau-Ginzburg model $(\mX_\Lie,\pot_\Lie)$.
%\end{thm}

Including our results, canonical models are now available for four out of the five cominuscule families: Grassmannians, Lagrangian Grassmannians, quadrics, and the exceptional family. The remaining family, consisting of the orthogonal Grassmannians $\OG(n,2n)$ of maximal isotropic subspaces of $\C^{2n}$ with respect to the standard inner product, is expected to have a model that is mostly similar to that of the Lagrangian Grassmannians. All canonical models, including for the exceptional case, have superpotentials consisting of the sum of quotients of polynomials in the Pl\"ucker coordinates. For Grassmannians these polynomials were linear, whereas for quadrics, Lagrangian Grassmannians, and (conjecturally) orthogonal Grassmannians these polynomials are quadratic. In contrast, the canonical model for the Cayley plane has a term with cubic polynomials, and the canonical model for the Freudenthal variety has a term with quartic polynomials as well as two terms with cubic polynomials.
%(Arguably, the quadrics can be split into odd and even cases, as they are homogeneous over Lie groups of different types, but the canonical models are nonetheless very similar.)

In fact, our methods apply to all cominuscule homogeneous spaces, with completely analogous proofs. The models obtained in this way will coincide with the canonical models for the Grassmannians, Lagrangian Grassmannians, and quadrics, and our methods will provide similar models for the orthogonal family. In the present form, our approach only handles one member of a family at the time; however, with general representation theoretic knowledge and the associated combinatorial tools, it should be possible to perform our methods either inductively or in generality. This is in contrast with some of the previous approaches using properties of the particular representations. 

Specifically, the methods in \cite{Pech_Rietsch_Williams_Quadrics, Pech_Rietsch_Lagrangian_Grassmannians,Marsh_Rietsch_Grassmannians} use specific details about the fundamental representations of the corresponding Lie groups to express their generalized minors as literal minors of matrix representations obtained from the associated standard representations. This dependence makes generalizing these approaches impractical, especially since the representations of $\LGE_6$ and $\LGE_7$ cannot be given as wedge products of natural representations. On the other hand, our approach is similar in spirit to that of \cite{Pech_Rietsch_Odd_Quadrics}, which compares both the Lie-theoretic and canonical LG-models with an intermediate obtained from the Lusztig torus, although the details in our situation are significantly more complicated. For example, the torus expansions of Pl\"ucker coordinates obtained in Lemma 7.3 of \cite{Pech_Rietsch_Odd_Quadrics} contain up to two terms, while the example computed in Appendix \ref{sec:expansion-appendix} here has $45$ terms. Furthermore, the numerators and denominators of some of the terms of our potentials are of degree $3$ or $4$ in Pl\"ucker coordinates, complicating matters further.

In \cite{Pech_Rietsch_Williams_Quadrics}, it is noted that the results of \cite{GLS_partial_flag_varieties_and_preprojective_algebras} imply that the coordinate rings of $\mX_\can$ and $\cmX$ carry a cluster structure. In particular, we view $\mX_\can$ as an open cluster variety and $\cmX$ as a compactification of $\mX_\can$ given by frozen variables, whose divisors form an anticanonical divisor. On the level of coordinate rings, this means that we obtain $\mathbb{C}[\mX_\can]$ from $\mathbb{C}[\cmX]$ by inverting frozen variables. In Section \ref{sec:cluster}, we give a Pl\"ucker coordinate description for the exceptional cominuscule family of the cluster structures given in \cite{GLS_partial_flag_varieties_and_preprojective_algebras}. The cluster variables are given in Lemmas \ref{lem:plucker-expressions-E6} and \ref{lem:plucker-expressions-E7}, and the quivers are given in Figures \ref{fig:E6-quiver} and \ref{fig:E7-quiver}. In order to compute the Pl\"ucker coordinate expressions for the cluster variables, we formulated Algorithm \ref{alg:subduction-for-gen-minors} for expressing generalized minors in terms of Pl\"ucker coordinates. This algorithm is reminiscent of the \emph{subduction algorithm} for \emph{Khovanskii bases}, and we were able to use it to define a \emph{valuation} $\nu$ on the coordinate ring $\mathbb{C}[\cmX]$. With this, we proved the following theorems:

\begin{thm}[Lemma {\ref{lem:one-dim-leaves}} and Proposition {\ref{prop:nobody}}]
  For minuscule $\cmX=(\ECHS)^\vee$, the valuation $\nu$ defined on $\C[\cmX]$ above has \emph{one-dimensional leaves}, and the associated \emph{Newton-Okounkov body} $\Delta(\mathbb{C}[\cmX],\nu)$ is equal to the convex hull of the valuations of the Pl\"ucker coordinates on $\cmX$.
\end{thm}

\begin{thm}[Proposition {\ref{prop:khovanskii}}]
  For minuscule $\cmX=(\ECHS)^\vee$, the Pl\"ucker coordinates form a Khovanskii basis for $\mathbb{C}[\cmX]$ with respect to the valuation $\nu$ defined on $\C[\cmX]$ above. 
\end{thm}

Although we only gave the statements above for $\cmX=\ECHS$, our proofs do not depend on the exceptional types in a meaningful way, and we propose in Conjecture \ref{conj:nobody} that the Pl\"ucker coordinates form a Khovanskii basis for a family of valuations depending on a choice of reduced expression for a particular Weyl group element. In particular, we obtain from this a family of Newton-Okounkov bodies as the convex hulls of the valuations of Pl\"ucker coordinates. While our conjecture makes sense for any homogeneous space, %Pl\"ucker coordinates are far more difficult to understand in non-minuscule situations, and 
we are unsure what additional conditions---if any---are required to obtain Khovanskii bases and Newton-Okounkov bodies.

%%%%%We also give cluster structures for $\cmX$ in terms of Pl\"ucker coordinates, which to the best of our knowledge were not known before. %cluster introduction here, e.g. parts of the introductory material in section 4. Move other parts to section 2
%%%%As a consequence of our computations, we were able to show that the Pl\"ucker coordinates form a Khovanskii basis with respect to a valuation coming from the lusztig torus expansions. This allows us to construct a Newton-Okounkov body for $\cmX$, as well as superpotential polytopes. %I think this is an important point we could try to sell more -- the lusztig torus computations are actually some kind of Khovanskii basis phenomenon, and this seems like a novel point that hasn't been mentioned before. 

We expect that our main results---these Pl\"ucker coordinate canonical models, the Pl\"ucker coordinate cluster structures, and the Khovanskii bases and Newton-Okounkov bodies---can be formulated type-independently for cominuscule and possibly more general homogeneous spaces. This stems from the fact that many of the intermediate results are type-independent themselves: the Lie-theoretic model of \cite{Rietsch_Mirror_Construction} holds for general homogeneous spaces; the Laurent polynomial expression for this model in \cite{Spacek_LP_LG_models} holds for all cominuscule spaces and is currently being extended to adjoint spaces by the first author; the Lie-theoretic and canonical models are compared using results by \cite{GLS_Kac_Moody_groups_and_cluster_algebras} which hold on the coordinate rings of general unipotent cells; and finally the cluster structures given by \cite{GLS_partial_flag_varieties_and_preprojective_algebras} are also valid on general unipotent cells. In fact, in Proposition \ref{prop:Criterion_Plucker_coordinates_coincide_with_GLS} we give a criterion expressing the generators of the coordinate rings given in \cite{GLS_Kac_Moody_groups_and_cluster_algebras} in terms of Pl\"ucker coordinates that holds for general cominuscule homogeneous spaces. Thus, the only obstruction to presenting canonical models and cluster structures type-independently is expressing the generalized minors appearing in \cite{GLS_partial_flag_varieties_and_preprojective_algebras} and \cite{GLS_Kac_Moody_groups_and_cluster_algebras} in terms of Pl\"ucker coordinates. 

 %with an explicit expression for the superpotential in Pl\"ucker coordinates
%I think we mentioned this above once, and we can speculate a bit more in a "future work" section if we decide to include it
%Having such a unified expression for the cominuscule cases would also give indications on how to generalize these models to other cases of homogeneous spaces. \marginbox{Put this here? We might be able to say some more specific things about the combinatorics of a Khovanskii basis...}

In \cite{Rietsch_Williams_NO_bodies_cluster_duality_and_mirror_symmetry_for_Grassmannians}, the authors study a combinatorial version of mirror symmetry in the form of an equality of polytopes defined using the auxiliary data of \emph{cluster structures}. There, cluster structures are used to construct a family of Newton-Okounkov associated with $\X$ and a family of polytopes associated with the superpotential on $\mX_\can$. These families are indexed by \emph{seeds} in the cluster structures, and it is shown that these two polytopes coincide for any choice of seed. While we do not pursue the story here, we provide many of the necessary ingredients, and it would be interesting to explore the combinatorics and see whether there is any relation between our Newton-Okounkov bodies and those coming from the cluster structure.

Apart from being of interest on their own, the canonical models for homogeneous spaces have appeared in a number of other contexts as well. Firstly, it was conjectured in \cite{Rietsch_Mirror_Construction} that solutions to the quantum differential equations of homogeneous spaces can be obtained from oscillatory integrals associated to the Landau-Ginzburg models. In \cite{Lam_Templier_The_mirror_conjecture_for_minuscule_flag_varieties} this conjecture was proven for \emph{minuscule} homogeneous spaces using the language of $D$-modules and the Langlands correspondence. However, as noted there, the isomorphism given in this proof is not constructed explicitly. In \cite{Pech_Rietsch_Williams_Quadrics}, an explicit injective morphism between the $D$-modules in question is given using the canonical mirror model for quadrics; for odd-dimensional quadrics (that are cominuscule but \emph{not} minuscule) an implicit isomorphism was known before so that their morphism gave an isomorphism; for even-dimensional quadrics (that are both cominuscule \emph{and} minuscule) the results of \cite{Lam_Templier_The_mirror_conjecture_for_minuscule_flag_varieties} implied that their injective morphism is in fact an isomorphism. We believe that the canonical model presented for the exceptional cominuscule family will lead to an analogous explicit morphism of $D$-modules, which will then be an isomorphism by \cite{Lam_Templier_The_mirror_conjecture_for_minuscule_flag_varieties}, but we do not verify this here.

Finally, in \cite{Castronovo_Fukaya_category_of_Grassmannians}, the canonical model for Grassmannians was modified to prove \emph{homological mirror symmetry} for Grassmannians $\Gr(k,n)$ with $n$ prime. This is an isomorphism of triangulated categories between the Fukaya category of a variety and the (bounded) derived category of coherent sheaves on its mirror. With canonical models available for four out of the five cominuscule families, these results form promising first steps in homological mirror symmetry for homogeneous spaces.

The outline of this article is as follows. In Section $2$, we set up our conventions and notation, and recall background on Rietsch's Lie-theoretic Langdau-Ginzburg models, the Cayley plane, and the Freudenthal variety. In Section $3$, we define our canonical models, give their Pl\"ucker coordinate expressions, and prove the isomorphism with Rietsch's Lie-theoretic model. In Section $4$, we review background on cluster structures and give Pl\"ucker coordinate cluster structures for the Cayley plane and Freudenthal variety, and in Section $5$, we define our Newton-Okounkov body and show that the Pl\"ucker coordinates give Khovanskii bases in both cases. In the appendices, we provide extra data and examples from our computations in Sections $3$ and $5$. 

\subsubsection*{Acknowledgements}
We implemented Algorithms \ref{alg:torus-expansion} and \ref{alg:subduction-for-gen-minors} in Sage \cite{sagemath} and Macaulay2 \cite{M2}, and used Polymake \cite{polymake-2000,polymake-2017} for Newton-Okounkov body computations. Our code is available at \url{https://github.com/foxflo/ExceptionalCominusculeSpaces}. The authors are grateful to Cl\'elia Pech and Lauren Williams, as well as to the anonymous referees, for many helpful comments and suggestions.
\newpage

\section{Preliminaries}

In the first two subsections (Subsections \ref{sec:notation} and \ref{sec:Lie_theoretic_Mirror}), we will consider general complete homogeneous spaces $\X$ for a simple and simply-connected complex algebraic group $G$ of rank $n$. These spaces are also often called \emph{generalized flag varieties}. We will work with these spaces in their (Pl\"ucker) embeddings as the closed orbits of the highest weight vectors in the projectivizations of the corresponding irreducible representations. Afterwards, beginning from Subsection \ref{sec:local-laurent}, we will restrict to the cases when $X$ is \emph{cominuscule}. 
\subsection{Conventions and notation}\label{sec:notation}
%We consider a complete homogeneous space $\X$, sometimes called a generalized flag variety, for a simple and simply-connected complex algebraic group $\G$ of rank $n$. %In this section and in section \ref{sec:Lie_theoretic_Mirror} we do not make any further assumptions on $\X$, but in the remaining sections will specialize to the case in which $\X$ is cominuscule.
%
We write $\g$ for the Lie algebra of $G$ and choose a set $(\Che_1,\Chf_1,\Chh_1,\ldots,\Che_n,\Chf_n,\Chh_n)$ of Chevalley generators, i.e.~triples $(\Che_i,\Chf_i,\Chh_i=[\Che_i,\Chf_i])$ that generate Lie subalgebras isomorphic to $\fsl_2$. These generators give a Cartan (or root space) decomposition of the Lie algebra as $\g=\up\op\cartan\op\um$ with $\up$ and $\um$ generated by $\{\Che_i\}$ and $\{\Chf_i\}$ respectively, and with $\cartan$ spanned by $\{\Chh_i\}$. %We denote by $\UEAp$ and $\UEAm$ the universal enveloping algebras of $\up$ and $\um$ respectively, and we write their completions as $\CUEAp$ and $\CUEAm$.

The Cartan decomposition of $\g$ gives rise to subgroups of $\G$: let $\unip$ and $\unim$ be the (upper and lower) unipotent subgroups of $\G$ with $\up$ and $\um$ as Lie algebras, and let $\torus$ be the maximal torus of $\G$ with $\cartan$ as Lie algebra. %Note that we can consider $\unip$ and $\unim$ as lying inside $\CUEAp$ and $\CUEAm$ respectively and that they are generated by the one-parameter subgroups
%\[
%\x_i(a) = \exp(a\,\Che_i) \qand \y_i(a) = \exp(a\,\Chf_i),
%\]
%for $i\in\{1,\ldots,n\}$ and $a\in\C$. Here $\exp(a\,\Che_i) = 1 + a\,\Che_i + \tfrac12a^2\,\Che_i^2 + \ldots \in \CUEAp$ and $\exp(a\,\Chf_i)\in\CUEAm$ is given analogously. 
We obtain a Borel subgroup $\borelp=\torus\unip$ and its opposite $\borelm=\torus\unim$; they have $\cartan\op\up$ and $\cartan\op\um$ as Lie algebras. The Borel subgroup defines aa unique parabolic subgroup $P$ containing $\borelp$ such that $\X=\G/\P$. Its Lie algebra $\parab$ satisfies $\up\op\cartan \subset\parab\subset\up\op\cartan\op\um=\g$; we will denote by $\ParabolicIndices\subset\{1,\ldots,n\}$ the set of indices such that
\begin{equation}
\parab = \lan \Che_i,\Chh_i,\Chf_j~|~i\in\{1,\ldots,n\},~j\in\ParabolicIndices\ran.
\label{eq:Parabolic_Lie_algebra}
\end{equation}
In this case we say that $\P$ is associated to the set $\IP=\{1,\ldots,n\}\setminus\ParabolicIndices$.

Denote by $\roots$ the set of roots and let $\sroots=\{\sr_1,\ldots,\sr_n\}$ be the base of simple roots determined by the Chevalley generators. The sets of positive and negative roots determined by the base are denoted by $\proots$ and $\nroots$ respectively. Writing $\droots$ for the coroots of $\G$, there is a unique pair $(\dG,\dtorus)$ of a Lie group $\dG$ and maximal torus $\dtorus\subset\dG$ that have the coroots $\droots$ of $\G$ as roots and the cocharacter lattice of $\G$ as character lattice; this is called the \emph{Langlands dual pair} and we will refer to $\dG$ as the \emph{Langlands dual group}.
%We write $\chars$ for the lattice of characters $\chi:\torus\to\C^*$ of the maximal torus (written additively). Within $\chars$, we denote the set of roots by $\roots\subset\chars$ and a base of simple roots $\sroots=\{\sr_1,\ldots,\sr_n\}$ is determined by the Chevalley generators. The associated sets of positive and negative roots are denoted by $\proots$ and $\nroots$ respectively. We denote the cocharacter lattice by $\dchars$ and the simple coroots by $\sdr_i:\chars\to\C$.
%
%With a given root system $\roots$ and character lattice $\chars$, there exists a unique group $\dG$ determined by having as root system the coroots $\droots$ and as character lattice the cocharacter lattice $\dchars$ of $\G$. The character lattice $\dchars$ of $\dG$ also determines a maximal torus $\dtorus$ in $\dG$. The pair $(\dG,\dtorus)$ is called the \emph{Langlands dual pair} associated to $(\G,\torus)$; we call $\dG$ the \emph{Langlands dual group}.
\begin{rem}\label{rem:dG-is-adjoint}
As $\G$ is assumed to be simply-connected, $\dG$ will be adjoint.
\end{rem}
$\sdroots=\{\sdr_1,\ldots,\sdr_n\}$ forms a base of simple roots for $\dG$ and determines both a decomposition $\dg=\dum\op\dcartan\op\dup$ of the Lie algebra of $\dG$ and positive $\pdroots$ and negative roots $\ndroots$ for $\dG$. We obtain the Langlands dual groups $\dunip$, $\dunim$, $\dborelp$, $\dborelm$ and $\dP$ analogously to above. The base $\sdroots$ also determines Chevalley generators $(\dChe_1,\dChf_1,\dChh_1,\ldots,\dChe_n,\dChf_n,\dChh_n)$ for $\dg$ and with these we can write the Lie algebra $\dparab$ of $\dP$ as
\begin{equation}
\dparab = \lan \dChe_i,\dChh_i,\dChf_j~|~i\in\{1,\ldots,n\},~j\in\ParabolicIndices\ran.
\label{eq:Lie_algebra_of_dual_parabolic_subgroup}
\end{equation}
In other words, the set $\ParabolicIndices\subset\{1,\ldots,n\}$ determines $\dP$ as well as $\P$; hence $\dP$ is associated to the complement $\IP$ as well.
\begin{rem}\label{rem:LanglandsDuality}
Note that Langlands duality reverses arrows in the Dynkin diagram of the type of $\G$; hence, for $\G=\LGE_n^\SC$ we have $\dG=\LGE_n^\AD$ the adjoint Lie group of type $\LGE_n$.
%When the Dynkin diagram of $\G$ is simply-laced, $\dG$ has the same Dynkin diagram (with the same numbering of the vertices). When the Dynkin diagram of $\G$ has a double or triple edge, the Dynkin diagram of $\dG$ is obtained by reversing the arrows at these edges. Explicitly, if $\G$ is of type $\LGA_n$, $\LGD_n$ or $\LGE_n$, then $\dG$ is of the same type; if $\G$ is of type $\LGB_n$, then $\dG$ is of type $\LGC_n$ and vice versa; finally, for $\G$ of type $\LGF_4$ and $\LGG_2$, $\dG$ is of the same type but has the reverse numbering of vertices.
\end{rem}

We will also need the universal enveloping algebras of $\dup$ and $\dum$, which we will denote by $\dUEAp$ and $\dUEAm$, and their completions $\dCUEAp$ and $\dCUEAm$. Note that we can consider $\dunip\subset\dCUEAp$ and $\dunim\subset\dCUEAm$ by identifying the one-parameter subgroups
\begin{equation}
\dx_i(a) = \exp(a\, \dChe_i) \qand \dy_i(a) = \exp(a\,\dChf_i)
\label{eq:dual_one_parameter_subgroups}
\end{equation}
for $i\in\{1,\ldots,n\}$ and $a\in\C$. Here %$(\dChe_1,\dChf_1,\dChh_1,\ldots,\dChe_n,\dChf_n,\dChh_n)$ are the Chevalley generators for $\dg$ and 
$\exp(a\,\dChe_i)=1+a\,\dChe_i+\frac12a^2(\dChe_i)^2+\ldots\in\dCUEAp$. %Note that the parabolic subgroup $\dP$ is associated to the same set $\ParabolicIndices$ as $\P$: that is, its Lie algebra is given by

%The complement of the set of indices is the same, and is denoted by $\IP=\{1,\ldots,n\}\setminus\ParabolicIndices$ as well.

We will write $\dChedual,\dChfdual\in\dgdual$ for the dual maps corresponding to the Chevalley generators $(\dChe_1,\dChf_1,\dChh_1,\ldots,\dChe_n,\dChf_n,\dChh_n)$ for $\dg$; they are given by
\begin{equation}
\dChedual(\dChe_j) = \de_{ij} = \dChfdual(\dChf_j) \qand \dChedual(\dChf_j)=0=\dChfdual(\dChe_j)
\label{eq:dualdualChevalleyGenerators}
\end{equation}
and are identically zero on the remaining root spaces and the Cartan algebra. We extend these maps to $\dunip$ and $\dunim$ using the the identification of the one-parameter subgroups $\dx_i(a)\in\dunip$ and $\dy_i(a)\in\dunim$ with $\exp(a\,\dChe_i)\in\dCUEAp$ and $\exp(a\,\dChf_i)\in\dCUEAm$ in combination with the inclusions of $\dup$ and $\dum$ into their completed universal algebras $\dCUEAp$ and $\dCUEAm$. Equivalently, $\dChedual$ and $\dChfdual$ are defined as the unique group homomorphisms $\dunip\to\C$ and $\dunim\to\C$ such that
\begin{equation}
%\begin{gathered}
\dChedual(\dx_j(a)) = a\de_{ij} = \dChfdual(\dy_j(a)). %\\
%\dChedual(\dy_j(a))=0=\dChfdual(\dx_j(a)).
%\end{gathered}
\label{eq:e*_and_f*}
\end{equation}

As $\dG$ is in general not simply-connected, we will need to consider the universal cover $\udG$ of $\dG$. As before, we define the universal covers $\udP$, $\udtorus$, 
%$\udunip$, $\udunim$, 
$\udborelp$ and $\udborelm$. 
%Write $\cen$ for the center of $\udG$, so that $\udG/\cen\cong\dG$, and similarly $\udP/\cen\cong\dP$, $\udborelp/\cen\cong\dborelp$, $\udborelm/\cen\cong\dborelm$ and $\udtorus/\cen\cong\dtorus$. 
%Note that $\udunip\cong\dunip$ and $\udunim\cong\dunim$ 
Note that the cover of $\dunip$ is in fact isomorphic to $\dunip$ and the same holds for $\dunim$, so we simply identify them.
\begin{rem}
Note that for Lie groups with simply-laced Dynkin diagrams Remark \ref{rem:LanglandsDuality} tells us that $\udG\cong\G$ as they are both simply-connected and of the same type.
\end{rem}

Denote by $\weyl$ the Weyl group associated to (the Coxeter diagram of) $\dG$; $\G$ and $\dG$ share the same Coxeter diagram and hence give rise to the same Weyl group. We write $s_i=s_{\sdr_i}\in\weyl$ for the simple reflections, $\sdr_i\in\sdroots$, which form a set of generators. A decomposition for $w\in\weyl$ as $w=s_{i_1}\cdots s_{i_j}$ is called a \emph{reduced expression} if $j$ is minimal and $\ell(w)=j$ is called the length of $w$. There is a unique element of greatest length denoted by $\wo\in\weyl$ and we abbreviate its length by $\ellwo=\ell(\wo)$. There is also a Weyl group associated to $\dP$, given by the subgroup $\weylp\subset\weyl$ generated by the simple reflections $\{s_i~|~i\in\ParabolicIndices\}$, compare equation \eqref{eq:Lie_algebra_of_dual_parabolic_subgroup}. Note that $\weylp$ is a Weyl group in its own right, associated to the Dynkin diagram of $\dG$ with the vertices marked by $\IP$ removed. We write $\wop$ for the longest element of $\weylp$ and $\ellwop=\ell(\wop)$ for its length. 

We are particularly interested in the quotient $\weyl/\weylp$. We fix a set $\cosets\subset\weyl$ of minimal coset representatives of this quotient (i.e.~for every coset the representative of minimal length), and we write the minimal representative of $\wo\weylp$ as $\wP$. Thus $\wo=\wP\wop$, and we fix reduced expressions 
\begin{equation}
\wP=s_{r_1}\cdots s_{r_\ellwP} \qand \wop=s_{r_{\ellwP+1}}\cdots s_{r_{\ellwP+\ellwop}},
\label{eq:Fixed_reduced_expression_for_wP_and_wop}
\end{equation}
so we obtain a reduced expression $\wo=s_{r_1}\cdots s_{r_\ellwP}s_{r_{\ellwP+1}}\cdots s_{r_{\ellwo}}$ (observe that $\ellwo=\ellwP+\ellwop$). However, it will be more convenient to use the reversed reduced expression for $\wo$
\begin{equation}
\wo=s_{r_{\ellwo}}\cdots s_{r_{\ellwP+1}}s_{r_{\ellwP}}\cdots s_{r_{1}},
\label{eq:Fixed_reduced_expression_for_wo}
\end{equation}
obtained from $\wo=\wo^{-1}$.

Finally, we associate two elements in $\dG$ to each $s_i\in\weyl$:
\begin{equation}
\ds_i = \dx_i(1)\dy_i(-1)\dx_i(1) \qand \bs_i = \dx_i(-1)\dy_i(1)\dx_i(-1)=\ds_i^{-1}.
\label{eq:def_ds_and_bs}
\end{equation}
For an arbitrary $w\in\weyl$ with reduced expression $w=s_{i_1}\cdots s_{i_d}$ we find an associated element in $\dG$ by setting $\dw=\ds_{i_1}\cdots\ds_{i_d}$ and $\bw =\bs_{i_1}\cdots\bs_{i_d}$. Note that $\bar w$ and $\dot w^{-1}$ have the reverse product of simple reflections, i.e.~$\dot w^{-1} = \bs_{i_d}\cdots\bs_{i_1}$. We write
\begin{equation}
\invdtorus=(\dtorus)^{\weylp}\subset\dtorus
\label{eq:df_invdtorus}
\end{equation} 
for the part of $\dtorus$ that is invariant under the action $\weylp\times\dtorus\to\dtorus$ given by $(w,t)\mapsto \dw t\dw^{-1}$. Clearly, $\invdtorus$ has dimension $\#\IP$.

\subsection{Rietsch's Lie-theoretic mirror model and a local Laurent polynomial expression}\label{sec:Lie_theoretic_Mirror}
%TODO add mention of geometric Satake

The mirror models we construct will be shown to be isomorphic to the Landau-Ginzburg models that have been constructed for general homogeneous spaces in \cite{Rietsch_Mirror_Construction}. In this subsection we will briefly discuss the construction given in \cite{Rietsch_Mirror_Construction} and the local Laurent polynomial expression for its superpotential given by the first author in \cite{Spacek_LP_LG_models}.

\subsubsection{The Lie-theoretic mirror model}
In this subsection $\X=\G/\P$ denotes a general homogeneous spaces with $\P$ an arbitrary parabolic subgroup. The mirror model consists of a subvariety of the open Richardson variety associated to $(\wop,\wo)$. This open Richardson variety is given by:
\begin{equation}
\mX_\Lie = \dRichard_{\wop,\wo} = \bigl(\dborelp\wop\dborelm\cap\dborelm\wo\dborelm\bigr)/\dborelm~\subset~\dG/\dborelm.
\label{eq:Richardson_var}
\end{equation}
There exists an isomorphism
\begin{equation}
\RisoZ:\mX_\Lie\times\invdtorus\overset\sim\longrightarrow\decomps
\label{eq:Decomposition_Isomorphism}
\end{equation}
to the variety
\begin{equation}
\decomps = \dborelm\bwo^{-1}\cap\dunip\invdtorus\bwop\dunim~\subset~\dG.
\label{eq:decomps}
\end{equation}
(See \cite{Rietsch_Mirror_Construction}, Section 4.)
\begin{rem}\label{rem:Decomposition_Isomorphism}
Compared to \cite{Rietsch_Mirror_Construction}, we have translated the definition of $\decomps$ by right multiplication by $\bwo^{-1}$. For a detailed description of $\RisoZ$, see Remark 3.1 of \cite{Spacek_LP_LG_models}.
\end{rem}
In \cite{Peterson}, the quantum cohomology of a generalized flag variety $\G/\P$ is presented as the coordinate ring of what is now known as the \emph{Peterson variety} (see e.g.~\cite{Rietsch_Mirror_Construction}, paragraph 3.2). Inside this (non-reduced) variety there is an open stratum whose coordinate ring is isomorphic to the quantum cohomology localised at the quantum parameters (see e.g.~\cite{Rietsch_Mirror_Construction}, equation (3.2)). In \cite{Rietsch_Mirror_Construction}, a function on $\decomps$ is defined whose critical locus is isomorphic to this open stratum. The isomorphism $\RisoZ$ of equation \eqref{eq:Decomposition_Isomorphism} then yields a subvariety of $\mX_\Lie\times\invdtorus$ whose coordinate ring is isomorphic to the localized quantum cohomology of $\X$.

These results have been reformulated in \cite{Marsh_Rietsch_Grassmannians}, Theorem 6.5; see also \cite{Pech_Rietsch_Lagrangian_Grassmannians}, Section 4. There, a Jacobi ring of $\mX_\Lie\times\invdtorus$ with respect to a superpotential replaces the critical locus in the original formulation. That is, the localised quantum cohomology of $\X=\G/\P$ is isomorphic to the coordinate ring of $\mX_\Lie\times\invdtorus$ modulo the ideal generated by derivatives of the superpotential:

\begin{thm}[{\cite{Rietsch_Mirror_Construction}, Theorem 4.1}]\label{thm:Lie-theoretic_LG_model}
Let $\X=\G/\P$ be a complete homogeneous space with $\G$ a simple, simply-connected algebraic group over $\C$ and with $\P$ a parabolic subgroup. There exists a superpotential $\pot_\Lie:\mX_\Lie\times\invdtorus\to\C$ (given in Definition \ref{df:potential}) such that
%\marginbox{$\pot$~might~clash\\w/~other~chapters!}
\[
qH^*(\X)_\loc \cong \C[\mX_\Lie\times\invdtorus]/\Jac{\Lie},
\]
where $qH^*(\X)_\loc$ is the (small) quantum cohomology of $\X$ with all quantum parameters inverted and where $\Jac{\Lie}$ is the ideal generated by the derivatives of $\pot_\Lie$ along $\mX_\Lie$.
\end{thm}
The superpotential $\pot_\Lie$ is presented in \cite{Pech_Rietsch_Lagrangian_Grassmannians} as the pull-back of a superpotential defined on $\decomps$ along the isomorphism $\RisoZ:\mX_\Lie\times\invdtorus\overset\sim\longrightarrow\decomps$ from equation \eqref{eq:Decomposition_Isomorphism}. To state this superpotential, we introduce to the following subset of $\dunim$:
\begin{equation}
\dunimP = \dunim\cap \dborelp\bwop\bwo\dborelp = \dunim\cap \dborelp\dwPinv\dborelp ~\subset~\dunim.
\label{eq:unique_decomps_unipotents}
\end{equation}
(Note that $\wop\wo=\wPinv$, so that $\bwop\bwo$ and $\dwPinv$ differ by a torus element and the two definitions of $\dunimP$ coincide.) This set has the following property:
\begin{lem}[{\cite{Pech_Rietsch_Lagrangian_Grassmannians}, Section 4}]
  \label{lem:z_has_unique_decomposition}
Every $z\in\decomps$ has a \emph{unique} decomposition $z=u_+t\bwop u_-$ with $u_+\in\dunip$, $t\in\invdtorus$ and $u_-\in\dunimP$.
\end{lem}
\begin{rem}
The proof of this result %Lemma 4.8(i)
in \cite{Pech_Rietsch_Lagrangian_Grassmannians} can be carried over to the general case without any modification, so it will be omitted here.
\end{rem}
The superpotential on $\decomps$ is now defined as follows:
\begin{df}\label{df:potential}
Define the superpotential $\potZ:\decomps\to\C$ as the map:
\[
\potZ:\quad z=u_+t\bwop u_- \longmapsto \SumOfEs(u_+^{-1}) + \SumOfFs(u_-),
\]
where $\SumOfEs = \sum_{i=1}^n\dChedual$ and $\SumOfFs=\sum_{i=1}^n\dChfdual$, and where the decomposition of $z=u_+t\bwop u_-$ is the unique decomposition with $u_-\in\dunimP$ as stated in Lemma \ref{lem:z_has_unique_decomposition}. Moreover, the superpotential $\pot_\Lie:\mX_\Lie\times\invdtorus\to\C$ mentioned in Theorem \ref{thm:Lie-theoretic_LG_model} is given by $\pot_\Lie=\potZ\circ\RisoZ$ with $\RisoZ$ given in equation \eqref{eq:Decomposition_Isomorphism}.
\end{df}

\subsubsection{A local Laurent polynomial expression for the superpotential}\label{sec:local-laurent}
From this point forward, we will restrict our attention to \emph{cominuscule} homogeneous spaces. Recall that parabolic subgroups containing a given Borel subgroup can be associated to a subsets of vertices of the Dynkin diagram, with the Borel subgroup corresponding to the full diagram and a maximal parabolic corresponding to a single vertex. We will write $\P=\P_k$ if $\P$ is maximal parabolic and associated to the $k$th vertex. The homogeneous space $\X=\G/\P_k$ is called cominuscule if the fundamental coweight $\dfwt[k]$ is \emph{minuscule}, i.e. it satisfies
\begin{equation}
\llan\dfwt[k],\sr\rran\in\{-1,0,+1\},\quad\forall\sr\in\roots,
\label{eq:minuscule_coweight}
\end{equation}
where $\llan\cdot,\cdot\rran$ denotes the dual pairing between the character and the cocharacter lattice.

In \cite{Spacek_LP_LG_models} a Laurent polynomial expression is given for $\potZ$ restricted to the following open, dense subvariety:
\[
\opendecomps = \dborelm\bwo^{-1}\cap\dunip\invdtorus\bwop\opendunim ~\subset~ \decomps,
\]
where $\opendunim\subset\dunimP$ is the open, dense algebraic torus defined by
\begin{equation}
\opendunim = \{\dy_{r_\ell}(a_\ell)\cdots\dy_{r_1}(a_1)~|~a_i\in\C^*\}.
\label{eq:df_opendunim}
\end{equation}
Recall that $(r_1,\ldots,r_\ell)$ is the sequence of indices of our fixed reduced expression for $\wP$ (see equation \ref{eq:Fixed_reduced_expression_for_wP_and_wop}). The fact that $\opendunim\subset\dunimP$ is open and dense is well-known, see for example Lemma 4.8(ii) of \cite{Pech_Rietsch_Lagrangian_Grassmannians} or Lemma 5.2 of \cite{Spacek_LP_LG_models}. The superpotential is parametrized by a certain Weyl group element: let $\wPPrime$ be the minimal representative of the coset $\wop s_k\weylp$ and set $\wPrime=\wP(\wPPrime)^{-1}$, writing its length as $\ell(\wPrime)=\ellwPrime$, then define $\wPrimeSubExp$ as the set of reduced subexpressions of $\wPrime$ in the fixed reduced expression of $\wP$.

The Laurent polynomial superpotential is now given as follows:
\begin{prop}[{\cite{Spacek_LP_LG_models}, Theorem 5.7}]\label{prop:LP_LG_model}
Let $\X=\G/\P$ be a cominuscule complete homogeneous space with $\G$ a simply-connected, simple, complex algebraic group and $\P=\P_k$ a (maximal) parabolic subgroup. The restriction $\potZo$ of $\potZ$ to $\opendecomps$ has the following Laurent polynomial expression:
\[
\potZo(z) = \sum_{i=1}^\ell a_i + q \frac{\sum_{(i_j)\in\wPrimeSubExp}\prod_{j=1}^{\ellwPrime} a_{i_j}}{\prod_{i=1}^\ell a_i}.
\]
Here $z\in\opendecomps$ is uniquely decomposed as $z=u_+t\bwop u_-$ with $u_-=\dy_{r_\ell}(a_\ell)\cdots\dy_{r_1}(a_1)\in\opendunim$. Also, $q\in\C^*$ is given by $q=\sdr_k(t)$ with $t\in\invdtorus$.
\end{prop}

In Corollary 8.12 of \cite{Spacek_LP_LG_models}, it is shown that the numerator of the quantum term can be enumerated using subsets of a quiver defined in \cite{CMP_Quantum_cohomology_of_minuscule_homogeneous_spaces} and that these subsets can be obtained from a given one using two elementary moves, which can be represented graphically in the same way: by replacing a vertex in the subset with the one below (resp.~above) it in the same column under the condition that all outgoing (resp.~incoming) arrows at the original vertex do not point to elements in the subset. (See Lemma 8.9 and Proposition 8.11 of \cite{Spacek_LP_LG_models}.) This quiver will appear again later, see Remark \ref{rem:GLS_and_CMP_quivers}. We would like to note that these combinatorial operations coincide with the ``excited moves'' defined on $d$-complete posets in \cite{Naruse_Okada_posets_equivariant_K_theory}, where equivariant $K$-theory is used to derive a skew hook formula.

\subsection{Generalized Pl\"ucker coordinates and the coordinate ring of \texorpdfstring{$(\dunimP)^\T$}{(U\_-\textasciicircum{}P)\textasciicircum{}T}}\label{sec:Plucker_coords_and_ring_of_dunimP}
For cominuscule homogeneous spaces $\X=\G/\P_k$, a consequence of the geometric Satake correspondence is that the Schubert classes that form a basis of the cohomology are in one-to-one correspondence with a set of projective coordinates called \emph{(generalized) Pl\"ucker coordinates} on the \emph{Langlands dual homogeneous space}
\begin{equation}
\cmX = \dP_k\backslash\dG=\udP_k\backslash\udG.
\label{eq:Langlands_dual_hom_space_ECHS}
\end{equation}
(Both expressions give the same homogeneous space, and we mostly use the second.) See \cite{Marsh_Rietsch_Grassmannians}, p.~3, for more references regarding the geometric Satake correspondence. Thus, if we reformulate the superpotential in Definition \ref{df:potential} in terms of Pl\"ucker coordinates, we obtain a description of the localized quantum cohomology in terms of Schubert classes and relations between them. Pl\"ucker coordinate expressions for Rietsch's superpotential have been given for Grassmannians \cite{Marsh_Rietsch_Grassmannians}, quadrics \cite{Pech_Rietsch_Williams_Quadrics}, and Lagrangian Grassmannians \cite{Pech_Rietsch_Lagrangian_Grassmannians}. In \cite{Pech_Rietsch_Williams_Quadrics}, this was obtained using a description of the coordinate ring of $\dunimP$ by \cite{GLS_Kac_Moody_groups_and_cluster_algebras}. 

%\marginbox{Remove paragraph?}In section \ref{sec:Canonical_LG_models_for_ECHS}, we will present Pl\"ucker coordinate expressions for the superpotentials of the exceptional cominuscule homogeneous spaces, the Cayley plane and the Freudenthal variety. Our approach will be analogous to \cite{Pech_Rietsch_Williams_Quadrics} (with some deviations), so we will introduce Pl\"ucker coordinates and the description of the coordinate ring of the transpose of $\dunimP$ in this subsection. \marginbox{section?}We will discuss the Cayley plane and the Freudenthal variety in subsection \ref{sec:ECHS}.

\subsubsection{Minuscule representations and generalized Pl\"ucker coordinates.}
With the assumption that $\X=\G/\P$ is cominuscule, we will be able to define the (generalized) Pl\"ucker coordinates on $\cmX=\udP\backslash\udG$ using a minuscule representation. To do this, we first recall a few properties of such representations. Since we consider minuscule representations of $\udG$ and $\dg$ with corresponding roots $\droots$, we adapt our notation to this situation.

First, we will need a general statement concerning the structure of minuscule representations and the action of Chevalley generators; this result can be found in \cite{Green_Reps_from_polytopes} or \cite{Geck_Minuscule_weights_and_Chevalley_groups}. Fix a non-degenerate, symmetric, bilinear form $(\cdot,\cdot)$ on the weight lattice of $\dg$ (e.g.~the Killing form) and let $c_i(\wt)=2(\wt,\sdr_i)/(\sdr_i,\sdr_i)$ for $\wt$ an element of the weight lattice.
\begin{thm}\label{thm:Green_structure_minuscule_reps}
Let $\minrep$ be a minuscule representation, i.e.~a highest weight representation with as highest weight $\minwt$ satisfying \eqref{eq:minuscule_coweight}. Denote its set of weights by $\minrepwts$. Then the following holds:
\begin{romanize}
\item For all $\wt\in\minrepwts$ and $i\in\{1,\ldots,n\}$, $c_i(\wt)\in\{-1,0,+1\}$.
\item Each weight space is one-dimensional.
\item Given a highest weight vector $\wtv{\minwt}^+$, there exists a basis of weight vectors $\{\wtv{\wt}\}$ such that $\wtv{\minwt}=\wtv{\minwt}^+$ and such that
\ali{
\dChe_i\cdot\wtv{\wt} &= \left\{
\begin{array}{ll}
	\wtv{\wt+\sdr_i}, &\text{if $c_i(\wt)=-1$,} \\
	0, &\text{otherwise,}
\end{array}\right. \\
\dChf_i\cdot\wtv{\wt} &= \left\{
\begin{array}{ll}
	\wtv{\wt-\sdr_i}, &\text{if $c_i(\wt)=+1$,} \\
	0, &\text{otherwise,}
\end{array}
\right.\\
\dChh_i\cdot\wtv{\wt}&=c_i(\wt)\,\wtv{\wt}.
}
\item For any vector $v\in\minrep$ and $i,j\in\{1,\ldots,n\}$,
\ali{
(\dChe_i)^2\cdot v &= 0 = (\dChf_i)^2\cdot v, \\
\dChe_i\dChe_j\dChe_i\cdot v &= 0 = \dChf_i\dChf_j\dChf_i\cdot v \qquad\text{if $a_{ij}=c_i(\sdr_j)=-1$,} \\
\dChe_i\dChf_j\cdot v &= 0 = \dChf_j\dChe_i\cdot v ~\qquad\text{if $a_{ij}<0$.}
}
\end{romanize}
\end{thm}

Combining this result with the definition of $\ds_i$ and $\bs_i$ in \eqref{eq:def_ds_and_bs}, we find
\begin{cor}\label{cor:Action_of_s_e_f}
With the notation of Theorem \ref{thm:Green_structure_minuscule_reps}, $s_i(\wt) = \wt-c_i(\wt)\,\sdr_i$ and
\ali{
\ds_i\cdot\wtvmu &= \left\{\begin{array}{ll}
\wtv{\wt+\sdr_i} = \dChe_i\cdot\wtvmu, & \text{if $c_i(\wt)=-1$,}\\
\wtvmu, & \text{if $c_i(\wt)=\hphantom{+}0$},\\
-\wtv{\wt-\sdr_i} = -\dChf_i\cdot\wtvmu, & \text{if $c_i(\wt)=+1$,}
\end{array}\right. \\
\bs_i\cdot\wtvmu &= \left\{\begin{array}{ll}
-\wtv{\wt+\sdr_i} = -\dChe_i\cdot\wtvmu, & \text{if $c_i(\wt)=-1$,}\\
\wtvmu, & \text{if $c_i(\wt)=\hphantom{+}0$,}\\
\wtv{\wt-\sdr_i} = \dChf_i\cdot\wtvmu, & \text{if $c_i(\wt)=+1$.}
\end{array}\right.
\intertext{Conversely, we have}
\dChe_i\cdot\wtvmu &= \left\{\begin{array}{ll}
\wtv{\wt+\sdr_i} = \ds_i\cdot\wtvmu = -\bs_i\cdot\wtvmu,  & \text{if $c_i(\wt)=-1$,}\\
0, & \text{otherwise,}
\end{array}\right. \\
\dChf_i\cdot\wtvmu &= \left\{\begin{array}{ll}
\wtv{\wt-\sdr_i} = \bs_i\cdot\wtvmu = -\ds_i\cdot\wtvmu,  & \text{if $c_i(\wt)=+1$,}\\
0, & \text{otherwise.}
\end{array}\right.
}
\end{cor}

With these general facts about the structure of minuscule representations, we return to defining the Pl\"ucker coordinates on $\cmX$. Recall that by assumption $\dfwt[k]$ is minuscule; this implies that the fundamental weight $-\wo\cdot\dfwt[k]$ is minuscule as well:
\[
\llan -\wo\cdot\dfwt[k],\rt\rran=\llan -\wo\cdot\dfwt[k],-\wo\cdot\rt'\rran = \llan\dfwt[k],\rt'\rran\in\{-1,0,+1\} %\qquad \rt\in\roots,~\rt'=-\wo\cdot\rt,
\]
(where $\rt\in\roots$ and $\rt'=-\wo\cdot\rt$), compare \eqref{eq:minuscule_coweight}. Note that $-\wo\cdot\dfwt[k]$ equals $\dfwt[\DynkinSymmetry(k)]$ for $\DynkinSymmetry:\{1,\ldots,n\}\to\{1,\ldots,n\}$ the symmetry involution of the Dynkin diagram of $\udG$ (note that $\DynkinSymmetry$ is trivial in the cases of type $\LGB_n$, $\LGC_n$, $\LGD_{2n}$ and $\LGE_7$). We will use the minuscule representation $\PluckerRep=\PluckerRepAlt$ to define the Pl\"ucker coordinates as follows.
%We will use the minuscule representation $\PluckerRepAlt$ to define the coordinates, so to simplify notation we will use $\DynkinSymmetry:\{1,\ldots,n\}\to\{1,\ldots,n\}$ to denote \emph{either} the Dynkin diagram symmetry involution \emph{or} the identity map depending on the case as described above; with this notation, we can simply write $\PluckerRepAlt=\PluckerRep$.

By Theorem \ref{thm:Green_structure_minuscule_reps}, there is a natural weight basis for $\PluckerRep$, which we denote by $\{v_i\}$. Writing $v_0$ for the basis vector of \emph{lowest} weight (i.e.~of weight $-\dfwt[k]$), the dual representation $\PluckerDualRep$ has highest weight vector $v_0^*$ (of weight $\dfwt[k]$), where $\{v_i^*\}$ is the dual basis. Note that $\PluckerDualRep$ is a \emph{right}-representation with action $v_i^*\cdot g = \bigl(v\mapsto v_i^*(g\cdot v)\bigr)$, and that $[v_0^*]\cdot\udP=[v_0^*]\in\PPluckerDualRep$. Hence, we obtain an embedding $\cmX\hookrightarrow\PPluckerDualRep$ given by $\udP g\mapsto[v_0^*\cdot g]$, using the fact that $\udP$ and $\{\dy_k(a)~|~a\in\C\}$ generate $\udG$. This induces the coordinates:

%\marginbox{Do we need more results on minuscule reps?}The structure of minuscule representations is well-known---see \cite{Green_Reps_from_polytopes} and \cite{Geck_Minuscule_weights_and_Chevalley_groups}. In particular, these representations allow a natural choice of weight basis $\{v_i\}$ that is in bijection with the elements of $\cosets$; in fact, if $w\in\cosets$ has reduced expression $w=s_{i_1}\cdots s_{i_j}$, then the corresponding weight basis vector $v_i$ can be written as $v_i=\dw\cdot v_0 = \dChe_{i_1}\cdots\dChe_{i_j}\cdot v_0$, where $v_0$ denotes the weight basis vector of lowest weight (i.e.~of weight $-\dfwt[k]$). 
%
%Taking the dual basis $\{v_i^*\}$ for $\PluckerDualRep$, we note that $v_0^*$ is the highest weight vector of weight $\dfwt[k]$. Thus, we have $[v_0^*]\cdot\udP=[v_0^*]\in\PPluckerDualRep$ and we obtain a natural embedding $\cmX\hookrightarrow\PPluckerDualRep$ given by $\udP g\mapsto[v_0^*\cdot g]$, using the fact that $\udP$ and $\{\dy_k(a)~|~a\in\C\}$ generate $\udG$. Using this embedding, we can define the coordinates as follows:
\begin{df}\label{df:Plucker_coordinates}
For any weight basis vector $v_i$ as above and $g\in\udG$, define $\p_i:\udG\to\C$ by
\begin{equation}
\p_i(g) = (v_0^*\cdot g)(v_i) = v_{0}^*(g \cdot v_i).
\label{eq:Generalized-Plucker-coordinates}
\end{equation}
These maps induce projective coordinates on $\cmX$ called \emph{(generalized) Pl\"ucker coordinates}.
\end{df}
We will abuse notation and call the maps $\p_i:\udG\to\C$ Pl\"ucker coordinates as well.

\subsubsection{The coordinate ring of $\dunipP=(\dunimP)^\T\subset\udG$.}
Recall from equation \eqref{eq:unique_decomps_unipotents} that we defined $\dunimP=\dunim\cap\dborelm\dwPinv\dborelm$. In Proposition 8.5 of \cite{GLS_Kac_Moody_groups_and_cluster_algebras}, the coordinate rings of these \emph{unipotent cells} have been described using a dual PBW basis compatible with a reduced expression. However, the result is formulated using cells of $\dunip$ instead of $\dunim$: we will have to use the \emph{transposition anti-automorphism} $\cdot^\T$ to translate their results to cells of $\dunim$. Recall that this anti-automorphism is given by mapping $\dx_i(a)$ to $\dy_i(a)$ and vice versa, while acting trivially on $\udtorus$. Note that $\bs_i^\T=\ds_i=\bs_i^{-1}$. 

We would like to emphasize that the results in \cite{GLS_Kac_Moody_groups_and_cluster_algebras} hold for general unipotent cells, but we will only consider the case of $\dunimP$ for $\P=\P_k$ with $\fwt[k]$ cominuscule here.

Writing $\dunipP$ for the transpose of $\dunimP$, we clearly have
\begin{equation}
\dunipP=(\dunimP)^\T = \dunip\cap\dborelp\dwP\dborelp.
\label{eq:dunipP}
\end{equation}
The coordinates are defined on $\dunipP$ using the embedding of $\dunip$ in the completed universal enveloping algebra $\dCUEAp$ of $\dup$. Namely, they form a subset of a basis for the \emph{graded dual} of $\dCUEAp$. As the universal enveloping algebra is in particular a infinite-dimensional vector space, we need to be more careful in defining the dual. We will make use of the Chevalley generators $(\dChe_1,\dChf_1,\dChh_1,\ldots,\dChe_n,\dChf_n,\dChh_n)$. 

Consider the Cartan (or root space) decomposition
\[
\dg = \dcartan\oplus\bigoplus_{\drt\in\droots} \dudrt,
\]
where for $\drt\in\pdroots$ a positive root we have that $\dudrt[+\drt]\subset\dup$ and $\dudrt[-\drt]\subset\dum$ are one-dimensional subspaces (see for example \cite{Humphreys_Lie_Algebras}, Section 8.5). Clearly, $\dudrt[+\sdr_i]=\C\dChe_i$ is spanned by the corresponding Chevalley generator, $\dudrt[+s_i(\sdr_j)]=\C[\dChe_i,\dChe_j]$ is spanned by the commutator of the generators if $(s_is_j)^2\neq 1$, and so on. 

Now, consider the \emph{positive root lattice} $\Nzero\sdroots\cong\Nzero^n$, i.e.~all non-negative integral linear combinations of the simple roots $\sdr_i$. For $\drt=\sum_{i=1}^n d_i\sdr_i\in\Nzero\sdroots$, denote by $\dCUEAp(\drt)$ the subspace spanned by those pure tensors $\dChe_{i_1}\cdots\dChe_{i_j}$ such that exactly $d_i$ of the factors have index $i$. For example, $\dCUEAp(\sdr_i) = \C\dChe_i=\dudrt[+\sdr_i]$ for $\sdr_i\in\sdroots$, and $\dCUEAp(2\sdr_1+\sdr_2) = \C\dChe_1\dChe_1\dChe_2\op\C\dChe_1\dChe_2\dChe_1\op\C\dChe_2\dChe_1\dChe_1$. (If $2\sdr_1+\sdr_2\in\pdroots$, note that $\dudrt[2\sdr_1+\sdr_2]=\C[\dChe_1,[\dChe_1,\dChe_2]]\subset\dCUEAp(2\sdr_1+\sdr_2)$ .) This endows $\dCUEAp$ with an $\Nzero\sdroots$-grading that naturally extends the root space decomposition of $\dup$ (i.e.~$\dudrt[+\drt]\subset\dCUEAp(\drt)$ for $\drt\in\pdroots\subset\Nzero\sdroots$). The graded dual is defined as the direct sum of the duals of these finite-dimensional subspaces,
\[
\dCUEApDual = \bigoplus_{\drt\in\Nzero\sdroots} \bigl(\dCUEAp(\drt)\bigr)^*.
\]

To define a basis on this dual, we start by recalling that we fixed a reduced expression for the minimal representative $\wP\in\cosets$ of the coset $\wo\weylp$ and denoted its sequence of indices by $(r_i)_{i=1}^{\ellwP}$ where $\ellwP=\ell(\wP)$, see equation \eqref{eq:Fixed_reduced_expression_for_wP_and_wop}. As a Weyl group element, $\wP$ maps a set $\wPpdroots=\{\wPdrt(j)~|~j=1,\ldots,\ellwP\}\subset\pdroots$ of positive roots to negative roots. It is well known that these are given by
\begin{equation}
\wPdrt(1)=\sdr_{r_{\ellwP}} \qand \wPdrt(j)=s_{r_{\ellwP}}\cdots s_{r_{\ellwP-j+2}}(\sdr_{r_{\ellwP-j+1}}) \quad\text{for $j\in\{2,\ldots,\ellwP\}$,}
\label{eq:wP_positive_roots_mapped_to_negative_roots}
\end{equation}
(for example, this follows from \cite{Humphreys_Lie_Algebras}, Section 10.2). Note that $s_{r_{\ellwP}}\cdots s_{r_{\ellwP-j+2}}$ are the first $j-1$ factors of the reduced expression for $\wPinv$ obtained by reversing the simple reflections in \eqref{eq:Fixed_reduced_expression_for_wP_and_wop}. Since $\wo=\wP\wop$, it is easy to see that $\pdroots=\wPpdroots\sqcup\pdrootsP$, where $\pdrootsP$ is the set of positive (co-) roots mapped to negative ones by $\wop$. We also remark that both $\wPpdroots$ and $\pdrootsP$ are \emph{bracket closed} subsets: if $\drt,\drtb\in\wPpdroots$ and $\drt+\drtb\in\pdroots$ then $\drt+\drtb\in\wPpdroots$, and the same holds for roots in $\pdrootsP$. (This follows from the linearity of action of the Weyl group elements on roots.) See also \cite{GLS_Kac_Moody_groups_and_cluster_algebras}, Section 4.3, and the references therein.

For each of the positive roots $\wPdrt(m)$, $m\in\{1,\ldots,\ell\}$, we choose a basis element $\wPdChe(m)$ for $\dudrt[+\wPdrt(m)]$. In particular, we take $\wPdChe(1)=\dChe_k$, as $\wPdrt(1)=\sdr_k$.
\begin{rem}\label{rem:GLS_dum_basis}
In \cite{GLS_Kac_Moody_groups_and_cluster_algebras}, the basis elements for the one-dimensional subspaces $\dudrt[+\wPdrt(m)]$ are not defined more explicitly, nor will we need an explicit choice of basis vector in the following. However, it is easy to define a set of basis elements $\wPdChe(m)\in\dudrt[+\wPdrt(m)]$ for $m\in\{1,\ldots,\ell\}$ using the following algorithm: consider the defining expression of $\wPdrt(m)$ in \eqref{eq:wP_positive_roots_mapped_to_negative_roots}; remove all simple reflections acting trivially; replace each remaining simple reflection $s_i$ by the corresponding $\ad(\dChe_i)$; and finally replace $\sdr_i$ with the corresponding $\dChe_i$.
\end{rem}
Now, we complete the elements $\{\wPdChe(m)~|~m = 1,\ldots,\ellwP\}$ to a basis of $\dup$ by choosing generating elements in $\dudrt[+\drt]$ for $\drt\in\pdroots\setminus\wPpdroots=\pdrootsP$. To simplify notation, we will denote these additional elements by $\wPdChe(m)$ for $m\in\{\ellwP+1,\ldots,\ellwP+\ellwop\}$, where $\ellwP+\ellwop=\ellwo=\ell(\wo)$ is the number of positive roots. Here we will also assume that the Chevalley generators $\dChe_i$ are part of this basis. With a basis for $\dup$, we can in turn define a PBW basis for $\dCUEAp$, the elements of which we denote by
\begin{equation}
\wPdChe(\bm) = \bigr(\wPdChe(1)\bigr)^{m_1}\bigl(\wPdChe(2)\bigr)^{m_2}\cdots\bigl(\wPdChe(\ellwo)\bigr)^{m_\ellwo}, %\quad\text{for }\bm=(m_1,\ldots,m_{\ellwo})\in\Nzero^{\ellwo}.
\label{eq:PBW_basis_elements}
\end{equation}
for $\bm=(m_1,\ldots,m_{\ellwo})\in\Nzero^{\ellwo}$. As each $\wPdChe(m)$ is homogeneous with respect to the grading on $\dCUEAp$ (this follows from Remark \ref{rem:GLS_dum_basis}), the PBW basis elements are homogeneous as well. %We denote by $|\bm|\in\Nzero^n$ the \emph{degree} of $\wPdChe(\bm)$ with respect to this grading, i.e.~$\wPdChe(\bm)\in\dCUEAp(|\bm|)$.

We can now define the dual PBW basis for $\dCUEApDual$ by taking the finite-dimensional vector space dual bases and extending the functionals trivially on the complement. We denote these dual basis elements $\pGLS_\bm$, so that $\pGLS_\bm\bigl(\wPdChe(\bn)\bigr)$ equals $1$ if $\bm=\bn\in\Nzero^\ellwo$ and $0$ otherwise. Whenever $\bm=(0,\ldots,0,1,0,\ldots,0)$ with a $1$ at the $m$th position, we simply write $\pGLS_m=\pGLS_\bm$; indeed, in this case we have $\wPdChe(\bm)=\wPdChe(m)$. This basis is called the \emph{dual PBW basis compatible with the reduced expression for $\wP$}. Using the natural embedding of $\dunip$ into $\dCUEAp$, we can consider the maps $\pGLS_m$ on $\dunip$ as well.

From the results in \cite{GLS_Kac_Moody_groups_and_cluster_algebras}, it follows that the coordinate ring of $\dunipP$ is generated by the dual basis elements $\pGLS_m$ for $m\in\{1,\ldots,\ellwP\}$ where $\ellwP=\ell(\wP)$. However, to fully express the coordinate ring, we also need to introduce a number of \emph{generalized minors} in the sense of Fomin and Zelevinsky \cite{Fomin_Zelevinsky_Double_Bruhat_Cells_and_Total_Positivity}. Recall that the highest weight space of any highest weight representation is one-dimensional. Thus, taking a highest weight vector $\hwt$, we can define the projection $\dfwtrep\to\C\hwt$ (parallel to a basis of weight vectors). We denote the coefficient $c$ such that a vector $v\in\dfwtrep$ projects to $c\hwt$ by
\begin{equation}
\bigl\lan v,\hwt\bigr\ran=c.
\label{eq:coefficient_after_projection_to_highest_weight_space}
\end{equation}
Note that the highest weight representation $\dfwtrep$ is implicit in the notation. Using the action of $\udG$ on these representations, we obtain maps $\udG\to\C$ given by $g\mapsto\lan g\cdot\hwt,\hwt\ran$. For $w_1,w_2\in\weyl$, the generalized minors are given in Definition 1.4 of \cite{Fomin_Zelevinsky_Double_Bruhat_Cells_and_Total_Positivity} as
\begin{equation}
\minor_{w_1(\dfwt),w_2(\dfwt)}(g) = 
\Bigl\lan (\bw_1)^{-1}g\bw_2\cdot\hwt,\hwt\Bigr\ran.
\label{eq:FZ_gen_minor}
\end{equation}
(\cite{Fomin_Zelevinsky_Double_Bruhat_Cells_and_Total_Positivity} gives a different formula, but their formula can easily be seen to agree with this expression. See for example Proposition 7.2 of \cite{GLS_Kac_Moody_groups_and_cluster_algebras}.) In particular, we will be concerned with the following minors:
\begin{equation}
\minor_{\dfwt,\wPinv(\dfwt)}(g) = \Bigl\lan g\dwPinv\cdot\hwt,\hwt\Bigr\ran.
\label{eq:GLS_def_of_minor}
\end{equation}
Note that we wrote $\dwPinv = \overline{\wPinv}$ to simplify the expression; it has the reverse reduced expression compared to $\wP$. We can now state the result:
\begin{prop}[{\cite[Proposition 8.5]{GLS_Kac_Moody_groups_and_cluster_algebras}}]\label{prop:GLS_coord_ring_dunipP}
The cell $\dunipP=(\dunimP)^\T\subset\udG$, see equation \eqref{eq:dunipP}, is an affine variety with coordinate ring isomorphic to
\begin{equation}
\C[\pGLS_{1},\ldots,\pGLS_{\ellwP}][\minor_{\dfwt[1],\wPinv(\dfwt[1])}^{-1},\ldots,\minor_{\dfwt[n],\wPinv(\dfwt[n])}^{-1}].
\label{eq:GLS_expression_for_coordinate_ring_of_dunimP}
\end{equation}
\end{prop}
%\marginbox{Adapt to new structure, still needed?}In Proposition \ref{prop:Criterion_Plucker_coordinates_coincide_with_GLS}, we will give a (type-independent) criterion for a generator $\pGLS_m$ to correspond to a Pl\"ucker coordinate under the translation isomorphism $\cdot^\T:\dunimP\to\dunipP$ (and up to a constant). Unfortunately, we have not yet found a type-independent approach to express the localized generalized minors in terms of Pl\"ucker coordinates. However, using the open, dense torus $\opendunim\subset\dunimP$ (or more precisely: its transpose) we can directly verify Pl\"ucker coordinate expressions for the minors in the case of the Cayley plane $\ECHS[6]$ (see Section \ref{sec:coord_ring_dunimP_E6}) and in the case of the Freudenthal variety $\ECHS[7]$ (see Section \ref{sec:coord_ring_dunimP_E7}). The remainder of Chapter \ref{ch:MS_for_ECHS} will be devoted to show an isomorphism between Rietsch's Lie-theoretic Landau-Ginzburg model (as discussed in Section \ref{sec:Lie_theoretic_Mirror}) and the so-called \emph{canonical Landau-Ginzburg model}, which we will define in Section \ref{sec:Canonical_LG_models_for_ECHS} for these \emph{exceptional cominuscule homogeneous spaces} $\ECHS$. We will introduce these in the next section.

%\input{Orthogonal-Grassmannians}
\subsection{The exceptional cominuscule family}\label{sec:ECHS}
% As we can see from Table \ref{tab:CominusculeSpaces},
Of the exceptional types, only $\LGE_6$ and $\LGE_7$ allow cominuscule weights. There are two cominuscule weights in type $\LGE_6$, $\fwt[1]$ and $\fwt[6]$, although the homogeneous spaces $\Exc_6^\SC/\P_1$ and $\Exc_6^\SC/\P_6$ are isomorphic. There is only one cominuscule weight in type $\LGE_7$, namely $\fwt[7]$. Thus, in practice, we only need to consider two exceptional cominuscule homogeneous spaces: the Cayley plane $\cayley=\ECHS[6]$ and the Freudenthal variety $\freudenthal$. Whenever we consider both varieties at the same time, we will call them ``cominuscule $\ECHS$'' (we are adding the adjective ``cominuscule'' to exclude the \emph{quasi-cominuscule} $\ECHS[8]$).

We follow the labeling on the Dynkin diagrams used by \cite{Bourbaki}:
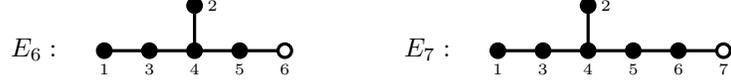
\[
\Exc_6:\quad
\begin{tikzpicture}[scale=.6,style=very thick,baseline=10em]
	\draw (0,6) -- (4,6); \draw (2,6) -- (2,7);
	\draw[fill=black]
	(0,6) circle (.15)
	(1,6) circle (.15)
	(2,6) circle (.15)
	(2,7) circle (.15)
	(3,6) circle (.15);
	\draw[fill=white]
	(4,6) circle (.15);
	\node at (0,6)[below=1pt]{\tiny$1$};
	\node at (2,7)[right=1pt]{\tiny$2$};
	\node at (1,6)[below=1pt]{\tiny$3$};
	\node at (2,6)[below=1pt]{\tiny$4$};
	\node at (3,6)[below=1pt]{\tiny$5$};
	\node at (4,6)[below=1pt]{\tiny$6$};
\end{tikzpicture}
\qquad\qquad
\Exc_7:\quad
\begin{tikzpicture}[scale=.6,style=very thick,baseline=10em]
	\draw (0,6) -- (5,6); \draw (2,6) -- (2,7);
	\draw[fill=black]
	(0,6) circle (.15)
	(1,6) circle (.15)
	(2,6) circle (.15)
	(2,7) circle (.15)
	(3,6) circle (.15)
	(4,6) circle (.15);
	\draw[fill=white]
	(5,6) circle (.15);
	\node at (0,6)[below=1pt]{\tiny$1$};
	\node at (2,7)[right=1pt]{\tiny$2$};
	\node at (1,6)[below=1pt]{\tiny$3$};
	\node at (2,6)[below=1pt]{\tiny$4$};
	\node at (3,6)[below=1pt]{\tiny$5$};
	\node at (4,6)[below=1pt]{\tiny$6$};
	\node at (5,6)[below=1pt]{\tiny$7$};
\end{tikzpicture}
\]
The Cayley plane and the Freudenthal variety are similar enough to allow for the same approach, but they do have some notable differences. For example, the symmetry involution $\DynkinSymmetry$ is trivial for type $\LGE_7$ while it is the obvious permutation for type $\Exc_6$. Note that this implies that $\dfwt[\DynkinSymmetry(6)]=\dfwt[1]$ for $\Exc_6$ and $\dfwt[\DynkinSymmetry(7)]=\dfwt[7]$ for $\Exc_7$.

In fact, we will need very few geometric properties of the exceptional cominuscule homogeneous spaces other than the structure of their cohomologies. The cohomology rings, and indeed the (small) quantum cohomology rings, of both varieties are already known; see \cite{CMP_Quantum_cohomology_of_minuscule_homogeneous_spaces} for presentations of both as well as for further references. %Rietsch's Lie-theoretic mirror model (discussed in Section \ref{sec:Lie_theoretic_Mirror}) holds for these exceptional cominuscule homogeneous spaces as well. However, as we remarked at the start of Section \ref{sec:Plucker_coords_and_ring_of_dunimP}, our goal is a mirror model in terms of Pl\"ucker coordinates. In Chapter \ref{ch:MS_for_ECHS}, we will construct such mirror models for the exceptional cominuscule homogeneous spaces.

\subsubsection{The Cayley plane.}
The exceptional cominuscule variety $\ECHS[6]$ is often denoted by $\cayley$ as it can be interpreted as the variety of (complex) octonion lines in three-dimensional (complex) octonion space. However, we will not use this description here and simply consider the Cayley plane as a homogeneous space of type $\LGE_6$. The geometry and the Chow ring (and thus the cohomology) of $\cayley$ is discussed in \cite{Iliev_Manivel_Chow_ring_of_the_Cayley_plane}\footnote{Note that a different labeling of the Dynkin diagram is used in \cite{Iliev_Manivel_Chow_ring_of_the_Cayley_plane}.}.

For the Cayley plane $\X=\ECHS[6]$, the embedding of the Langlands dual homogeneous space is $\cmX\hookrightarrow\Pdfwtrepdual[1]$. Since the fundamental weight representation $\dfwtrep[1]$ is minuscule, the Weyl group acts transitively on the weights. Thus, we can define a graph that has the weights of $\dfwtrep[1]$ as vertices and has an edge labeled $i$ between two weights if $s_i$ maps one weight to the other (and vice versa). The graph we obtain is the following:
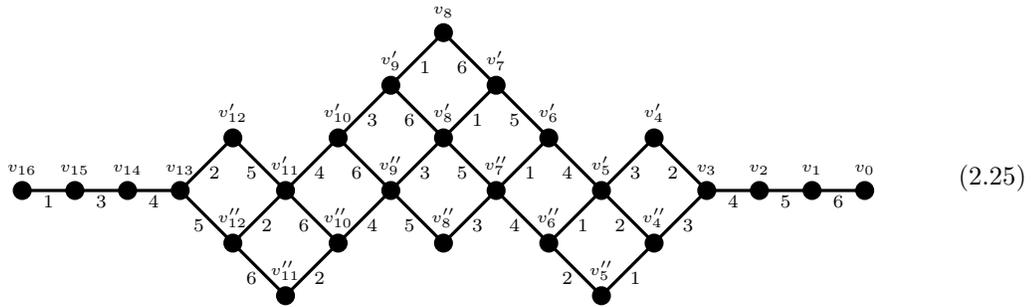
\vspace{-.5em}
\begin{equation}
\begin{tikzpicture}[rotate=90,scale=.7,style=very thick,baseline=0.25em]
	\draw 
	(0,-8)--(0,-5)--(-2,-3)--(0,-1)--(-1,0)--(0,1)--(-2,3)--(0,5)--(0,8)
	(0,-5)--(1,-4)--(0,-3)--(3,0)--(0,3)--(1,4)--(0,5)
	(-1,-4)--(0,-3)--(-1,-2) (-1,4)--(0,3)--(-1,2)
	(1,2)--(0,1)--(2,-1) (1,-2)--(0,-1)--(2,1);
	\draw[black, fill=black] 
	(0,-8) circle (.15) 
	(0,-7) circle (.15)
	(0,-6) circle (.15)
	(0,-5) circle (.15)
	(1,-4) circle (.15)
	(-1,-4) circle (.15)
	(0,-3) circle (.15)
	(-2,-3) circle (.15)
	(1,-2) circle (.15)
	(-1,-2) circle (.15)
	(2,-1) circle (.15)
	(0,-1) circle (.15)
	(3,0) circle (.15)
	(1,0) circle (.15)
	(-1,0) circle (.15)
	(2,1) circle (.15)
	(0,1) circle (.15)
	(1,2) circle (.15)
	(-1,2) circle (.15)
	(0,3) circle (.15)
	(-2,3) circle (.15)
	(1,4) circle (.15)
	(-1,4) circle (.15)
	(0,5) circle (.15)
	(0,6) circle (.15)
	(0,7) circle (.15)
	(0,8) circle (.15);
	\node at (0,-7.5)[below=-2pt]{\scriptsize 6};
	\node at (0,-6.5)[below=-2pt]{\scriptsize 5};
	\node at (0,-5.5)[below=-2pt]{\scriptsize 4};
	\node at (0.5,-4.5)[below left=-3pt]{\scriptsize 2};
	\node at (-0.5,-4.5)[below right=-3pt]{\scriptsize 3};
	\node at (0.5,-3.5)[below right=-3pt]{\scriptsize 3};
	\node at (-0.5,-3.5)[below left=-3pt]{\scriptsize 2};
	\node at (-1.5,-3.5)[below right=-3pt]{\scriptsize 1};
	\node at (0.5,-2.5)[below left=-3pt]{\scriptsize 4};
	\node at (-0.5,-2.5)[below right=-3pt]{\scriptsize 1};
	\node at (-1.5,-2.5)[below left=-3pt]{\scriptsize 2};
	\node at (1.5,-1.5)[below left=-3pt]{\scriptsize 5};
	\node at (0.5,-1.5)[below right=-3pt]{\scriptsize 1};
	\node at (-0.5,-1.5)[below left=-3pt]{\scriptsize 4};
	\node at (2.5,-0.5)[below left=-3pt]{\scriptsize 6};
	\node at (1.5,-0.5)[below right=-3pt]{\scriptsize 1};
	\node at (0.5,-0.5)[below left=-3pt]{\scriptsize 5};
	\node at (-0.5,-0.5)[below right=-3pt]{\scriptsize 3};
	\node at (2.5,0.5)[below right=-3pt]{\scriptsize 1};
	\node at (1.5,0.5)[below left=-3pt]{\scriptsize 6};
	\node at (0.5,0.5)[below right=-3pt]{\scriptsize 3};
	\node at (-0.5,0.5)[below left=-3pt]{\scriptsize 5};
	\node at (1.5,1.5)[below right=-3pt]{\scriptsize 3};
	\node at (0.5,1.5)[below left=-3pt]{\scriptsize 6};
	\node at (-0.5,1.5)[below right=-3pt]{\scriptsize 4};
	\node at (0.5,2.5)[below right=-3pt]{\scriptsize 4};
	\node at (-0.5,2.5)[below left=-3pt]{\scriptsize 6};
	\node at (-1.5,2.5)[below right=-3pt]{\scriptsize 2};
	\node at (0.5,3.5)[below left=-3pt]{\scriptsize 5};
	\node at (-0.5,3.5)[below right=-3pt]{\scriptsize 2};
	\node at (-1.5,3.5)[below left=-3pt]{\scriptsize 6};
	\node at (0.5,4.5)[below right=-3pt]{\scriptsize 2};
	\node at (-0.5,4.5)[below left=-3pt]{\scriptsize 5};
	\node at (0,5.5)[below=-2pt]{\scriptsize 4};
	\node at (0,6.5)[below=-2pt]{\scriptsize 3};
	\node at (0,7.5)[below=-2pt]{\scriptsize 1};
	\node at (0,-8)[above=2pt]{\tiny$v_0$};
	\node at (0,-7)[above=2pt]{\tiny$v_1$};
	\node at (0,-6)[above=2pt]{\tiny$v_2$};
	\node at (0,-5)[above=2pt]{\tiny$v_3$};
	\node at (1,-4)[above=2pt]{\tiny$v_4'$};
	\node at (-1,-4)[above=2pt]{\tiny$v_4''$};
	\node at (0,-3)[above=2pt]{\tiny$v_5'$};
	\node at (-2,-3)[above=2pt]{\tiny$v_5''$};
	\node at (1,-2)[above=2pt]{\tiny$v_6'$};
	\node at (-1,-2)[above=2pt]{\tiny$v_6''$};
	\node at (2,-1)[above=2pt]{\tiny$v_7'$};
	\node at (0,-1)[above=2pt]{\tiny$v_7''$};
	\node at (3,0)[above=2pt]{\tiny$v_8$};
	\node at (1,0)[above=2pt]{\tiny$v_8'$};
	\node at (-1,0)[above=2pt]{\tiny$v_8''$};
	\node at (2,1)[above=2pt]{\tiny$v_9'$};
	\node at (0,1)[above=2pt]{\tiny$v_9''$};
	\node at (1,2)[above=2pt]{\tiny$v_{10}'$};
	\node at (-1,2)[above=2pt]{\tiny$v_{10}''$};
	\node at (0,3)[above=2pt]{\tiny$v_{11}'$};
	\node at (-2,3)[above=2pt]{\tiny$v_{11}''$};
	\node at (1,4)[above=2pt]{\tiny$v_{12}'$};
	\node at (-1,4)[above=2pt]{\tiny$v_{12}''$};
	\node at (0,5)[above=2pt]{\tiny$v_{13}$};	
	\node at (0,6)[above=2pt]{\tiny$v_{14}$};
	\node at (0,7)[above=2pt]{\tiny$v_{15}$};
	\node at (0,8)[above=2pt]{\tiny$v_{16}$};
\end{tikzpicture}
\label{eq:Cayley-Hasse-diagram}
\end{equation}
Here, we have set the highest weight space at the left and denoted a choice of highest weight vector by $v_{16}$. Applying equation \eqref{eq:Generalized-Plucker-coordinates} to each of the basis vectors $v_i^{(j)}$ gives the full set of Pl\"ucker coordinates. 

Using Corollary \ref{cor:Action_of_s_e_f}, we see that the action of the representatives $\ds_i,\bs_i\in\udG$ can be directly related to the action of $\dChe_i,\dChf_i\in\dg$. Here, applied to $\dfwtrep[1]$, this translates as follows: an edge marked $i$ between two vectors $v$ (on the left) and $w$ (on the right) means that $\bs_i\cdot v=\dChf_i\cdot v = w$ and $\ds_i\cdot w=\dChe_i\cdot w = v$.

The diagram for $\dfwtrep[6]$ is in fact the same, however, now the highest weight space is on the right, and the actions of $\bs_i$ and $\ds_i$ go in the opposite direction. (In other words, exchange the position of $v$ and $w$ in the above.)

More importantly, this graph is identical to the \emph{Hasse diagram} for the Cayley plane: the Hasse diagram has as vertices the Schubert classes, and edges mark inclusions. The identification between the graph above and the Hasse diagram is obtained by replacing the weight vector $v_i^{(j)}$ with the Schubert class $\scs_i^{(j)}$, where we use the labeling of Schubert classes given in \cite{CMP_Quantum_cohomology_of_minuscule_homogeneous_spaces}, Section 2.3.\footnote{Note that \cite{CMP_Quantum_cohomology_of_minuscule_homogeneous_spaces} consider $\cayley$ as $\Exc_6^{\SC}/\P_1$, so the diagram would have the reverse labeling of edges.} This correspondence is implied by the geometric Satake correspondence, which we mentioned at the start of this section. Indeed, we have chosen the labeling for the basis of the representation in such a way that these match up with the labeling for the Schubert classes of \cite{CMP_Quantum_cohomology_of_minuscule_homogeneous_spaces}. Note that the fundamental class $[\X]=\scs_0=1\in H^0(\X)$ of the Cayley plane is on the right, and $\scs_1$ is the class of a hyperplane. Considering the diagram in this way, an edge $i$ from $\scs_{w_1}$ (on the left) to $\scs_{w_2}$ (on the right) denotes that $w_2=s_i w_1$. (Here $\scs_w$ is the class of the Schubert variety $\overline{\borelp w\P/\P}$.)

Now, the interpretation of \eqref{eq:Cayley-Hasse-diagram} as the Hasse diagram indicates a final interpretation: we can interpret the diagram as the partial ordering of the elements of $\cosets$ under the Bruhat order given by $v\le w$ if there is a $u\in\weyl$ such that $w=uv$ and $\ell(w)=\ell(u)+\ell(v)$. Thus, any path from the left to the right in the diagram in \eqref{eq:Cayley-Hasse-diagram} gives a reduced expression for the minimal coset representative $\wP$ of $\wo\weylp$; we will fix the following:
\begin{equation}
\wP = s_1s_3s_4s_2s_5s_4s_3s_1s_6s_5s_4s_3s_2s_4s_5s_6.
\label{eq:CayleyRedExp_wP}
\end{equation}
Whenever we consider the Cayley plane, the sequence $(r_i)_{i=1}^{16}$ will denote the sequence of indices of this reduced expression by, i.e.~$r_1=1$, $r_2=3$ and so on.

%%%For example, in the case $\cayley=\ECHS[6]$,
%%%\[
%%%\wPdrt(1) = \sdr_6, \quad \wPdrt(2) = s_6(\sdr_5)=\sdr_5+\sdr_6, \quad \wPdrt(3)=s_6s_5(\sdr_4)=\sdr_4+\sdr_5+\sdr_6.
%%%\]
%%%
%%%So, for example in the case of $\cayley=\ECHS[6]$\marginbox{Keep examples?}
%%%\[
%%%\wPdChf(1) = \dChf_6\in\dudrt[-\sdr_6], \quad \wPdChf(2) = [\dChf_6,\dChf_5]\in\dudrt[-\sdr_5-\sdr_6], \quad \wPdChf(3) = [\dChf_6,[\dChf_5,\dChf_4]]\in\dudrt[-\sdr_4-\sdr_5-\sdr_6].
%%%\]
%%%A less straightforward example is $\wPdChf(6)$: we have 
%%%\[
%%%\wPdrt(6)=s_6s_5s_4s_2s_3(\sdr_4) = \sdr_2+\sdr_3+\sdr_4+\sdr_5+\sdr_6 = s_6s_5s_2s_3(\sdr_4)
%%%\]
%%%so we find that $\wPdChf(6) = [\dChf_6,[\dChf_5,[\dChf_2,[\dChf_3,\dChf_4]]]]\in\dudrt[-\wPdrt(6)]$.
%%%
%-----------------------------------

\subsubsection{The Freudenthal variety.} The Freudenthal variety is also related to (complex) octonions, namely it is the closed orbit of $\Exc_7$ on the \emph{Zorn algebra}, see Section 2.3 of \cite{CMP_Quantum_cohomology_of_minuscule_homogeneous_spaces}. The most extensive analysis of the Freudenthal variety can be found in \cite{Freudenthal_Lie_groups_in_the_foundations_of_geometry}.

For the Freudenthal variety, the Langlands dual homogeneous space $\cmX$ is naturally embedded as $\cmX\hookrightarrow\Pdfwtrepdual[7]$. The minuscule fundamental weight representation $\dfwtrep[7]$ has the following weight structure:
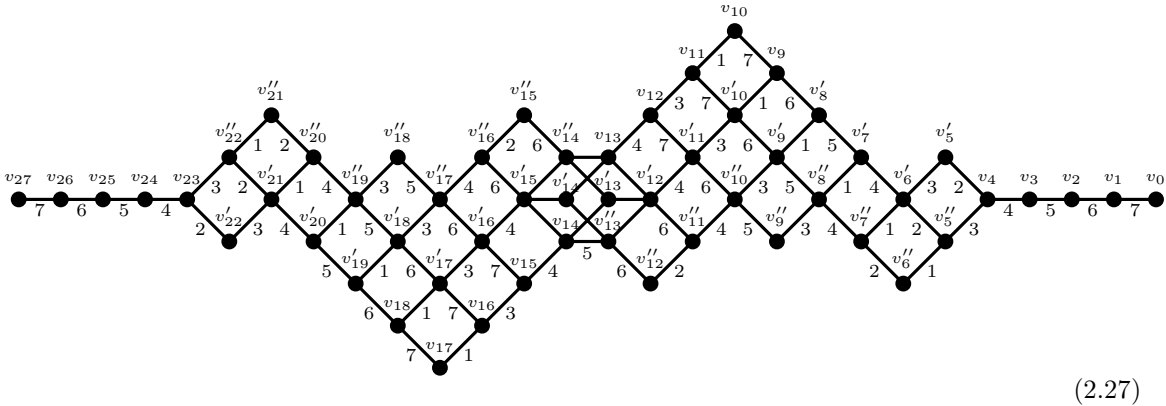
\vspace{-1em}
\begin{equation}\hspace{-1.65em}
\begin{tikzpicture}[rotate=90,scale=.56,style=very thick,baseline=0.25em]
	\draw 
	(0,-9)--(0,-5)--(-2,-3)--(0,-1)--(-1,0)--(0,1)--(-2,3)--(0,5)--(0,6)--(-3,9)--(0,12)--(-1,13)--(0,14)--(0,18)
	(0,-5)--(1,-4)--(0,-3)--(3,0)--(0,3)--(1,4)--(0,5)
	(-1,-4)--(0,-3)--(-1,-2) (-1,4)--(0,3)--(-1,2)
	(1,2)--(0,1)--(2,-1) (1,-2)--(0,-1)--(2,1)
	(0,14)--(2,12)--(0,10)--(1,9)--(0,8)--(2,6)--(0,4)--(0,3)
	(3,0)--(4,1)--(1,4)--(1,5)--(0,6)--(1,7)
	(-3,9)--(-4,8)--(-1,5)--(-1,4)
	(1,13)--(0,12)--(1,11) (-1,11)--(0,10)--(-1,9)--(0,8)--(-2,6)
	(-2,10)--(-1,9)--(-3,7) (0,6)--(-1,5)--(0,4)
	(2,1)--(3,2) (1,2)--(2,3);
	\draw[black, fill=black] 
	(0,-9) circle (.15)
	(0,-8) circle (.15) 
	(0,-7) circle (.15)
	(0,-6) circle (.15)
	(0,-5) circle (.15)
	(1,-4) circle (.15)
	(-1,-4) circle (.15)
	(0,-3) circle (.15)
	(-2,-3) circle (.15)
	(1,-2) circle (.15)
	(-1,-2) circle (.15)
	(2,-1) circle (.15)
	(0,-1) circle (.15)
	(3,0) circle (.15)
	(1,0) circle (.15)
	(-1,0) circle (.15)
	(4,1) circle (.15)
	(2,1) circle (.15)
	(0,1) circle (.15)
	(3,2) circle (.15)
	(1,2) circle (.15)
	(-1,2) circle (.15)
	(2,3) circle (.15)
	(0,3) circle (.15)
	(-2,3) circle (.15)
	(1,4) circle (.15)
	(0,4) circle (.15)
	(-1,4) circle (.15)
	(1,5) circle (.15)
	(0,5) circle (.15)
	(-1,5) circle (.15)
	(2,6) circle (.15)
	(0,6) circle (.15)
	(-2,6) circle (.15)
	(1,7) circle (.15)
	(-1,7) circle (.15)
	(-3,7) circle (.15)
	(0,8) circle (.15)
	(-2,8) circle (.15)
	(-4,8) circle (.15)
	(1,9) circle (.15)
	(-1,9) circle (.15)
	(-3,9) circle (.15)
	(0,10) circle (.15)
	(-2,10) circle (.15)
	(1,11) circle (.15)
	(-1,11) circle (.15)
	(2,12) circle (.15)
	(0,12) circle (.15)
	(1,13) circle (.15)
	(-1,13) circle (.15)
	(0,14) circle (.15)
	(0,15) circle (.15)
	(0,16) circle (.15)
	(0,17) circle (.15)
	(0,18) circle (.15);
	\node at (0,-8.5)[below=-2pt]{\scriptsize 7};
	\node at (0,-7.5)[below=-2pt]{\scriptsize 6};
	\node at (0,-6.5)[below=-2pt]{\scriptsize 5};
	\node at (0,-5.5)[below=-2pt]{\scriptsize 4};
	\node at (0.5,-4.5)[below left=-3pt]{\scriptsize 2};
	\node at (-0.5,-4.5)[below right=-3pt]{\scriptsize 3};
	\node at (0.5,-3.5)[below right=-3pt]{\scriptsize 3};
	\node at (-0.5,-3.5)[below left=-3pt]{\scriptsize 2};
	\node at (-1.5,-3.5)[below right=-3pt]{\scriptsize 1};
	\node at (0.5,-2.5)[below left=-3pt]{\scriptsize 4};
	\node at (-0.5,-2.5)[below right=-3pt]{\scriptsize 1};
	\node at (-1.5,-2.5)[below left=-3pt]{\scriptsize 2};
	\node at (1.5,-1.5)[below left=-3pt]{\scriptsize 5};
	\node at (0.5,-1.5)[below right=-3pt]{\scriptsize 1};
	\node at (-0.5,-1.5)[below left=-3pt]{\scriptsize 4};
	\node at (2.5,-0.5)[below left=-3pt]{\scriptsize 6};
	\node at (1.5,-0.5)[below right=-3pt]{\scriptsize 1};
	\node at (0.5,-0.5)[below left=-3pt]{\scriptsize 5};
	\node at (-0.5,-0.5)[below right=-3pt]{\scriptsize 3};
	\node at (3.5,0.5)[below left=-3pt]{\scriptsize 7};
	\node at (2.5,0.5)[below right=-3pt]{\scriptsize 1};
	\node at (1.5,0.5)[below left=-3pt]{\scriptsize 6};
	\node at (0.5,0.5)[below right=-3pt]{\scriptsize 3};
	\node at (-0.5,0.5)[below left=-3pt]{\scriptsize 5};
	\node at (3.5,1.5)[below right=-3pt]{\scriptsize 1};
	\node at (2.5,1.5)[below left=-3pt]{\scriptsize 7};
	\node at (1.5,1.5)[below right=-3pt]{\scriptsize 3};
	\node at (0.5,1.5)[below left=-3pt]{\scriptsize 6};
	\node at (-0.5,1.5)[below right=-3pt]{\scriptsize 4};
	\node at (2.5,2.5)[below right=-3pt]{\scriptsize 3};
	\node at (1.5,2.5)[below left=-3pt]{\scriptsize 7};
	\node at (0.5,2.5)[below right=-3pt]{\scriptsize 4};
	\node at (-0.5,2.5)[below left=-3pt]{\scriptsize 6};
	\node at (-1.5,2.5)[below right=-3pt]{\scriptsize 2};
	\node at (1.5,3.5)[below right=-3pt]{\scriptsize 4};
	%\node at (0.5,3.5)[below left=-3pt]{\scriptsize 7};
	%\node at (0,3.5)[below right=-2.5pt]{\scriptsize 5};
	%\node at (-0.5,3.5)[below right=-3pt]{\scriptsize 2};
	\node at (-1.5,3.5)[below left=-3pt]{\scriptsize 6};
	%\node at (1,4.5)[above=-2pt]{\scriptsize 5};
	%\node at (0.45,4.875){\scriptsize 2};
	%\node at (0.45,4.125){\scriptsize 7};
	%\node at (-0.45,4.875){\scriptsize 7};
	%\node at (-0.45,4.125){\scriptsize 2};
	\node at (-1,4.5)[below=-2pt]{\scriptsize 5};
	\node at (1.5,5.5)[below left=-3pt]{\scriptsize 6};
	%\node at (0.5,5.5)[below right=-3pt]{\scriptsize 2};
	%\node at (0,5.5)[below left=-2.5pt]{\scriptsize 5};
	%\node at (-0.5,5.5)[below left=-3pt]{\scriptsize 7};
	\node at (-1.5,5.5)[below right=-3pt]{\scriptsize 4};
	\node at (1.5,6.5)[below right=-3pt]{\scriptsize 2};
	\node at (0.5,6.5)[below left=-3pt]{\scriptsize 6};
	\node at (-0.5,6.5)[below right=-3pt]{\scriptsize 4};
	\node at (-1.5,6.5)[below left=-3pt]{\scriptsize 7};
	\node at (-2.5,6.5)[below right=-3pt]{\scriptsize 3};
	\node at (0.5,7.5)[below right=-3pt]{\scriptsize 4};
	\node at (-0.5,7.5)[below left=-3pt]{\scriptsize 6};
	\node at (-1.5,7.5)[below right=-3pt]{\scriptsize 3};
	\node at (-2.5,7.5)[below left=-3pt]{\scriptsize 7};
	\node at (-3.5,7.5)[below right=-3pt]{\scriptsize 1};
	\node at (0.5,8.5)[below left=-3pt]{\scriptsize 5};
	\node at (-0.5,8.5)[below right=-3pt]{\scriptsize 3};
	\node at (-1.5,8.5)[below left=-3pt]{\scriptsize 6};
	\node at (-2.5,8.5)[below right=-3pt]{\scriptsize 1};
	\node at (-3.5,8.5)[below left=-3pt]{\scriptsize 7};
	\node at (0.5,9.5)[below right=-3pt]{\scriptsize 3};
	\node at (-0.5,9.5)[below left=-3pt]{\scriptsize 5};
	\node at (-1.5,9.5)[below right=-3pt]{\scriptsize 1};
	\node at (-2.5,9.5)[below left=-3pt]{\scriptsize 6};
	\node at (0.5,10.5)[below left=-3pt]{\scriptsize 4};
	\node at (-0.5,10.5)[below right=-3pt]{\scriptsize 1};
	\node at (-1.5,10.5)[below left=-3pt]{\scriptsize 5};
	\node at (1.5,11.5)[below left=-3pt]{\scriptsize 2};
	\node at (0.5,11.5)[below right=-3pt]{\scriptsize 1};
	\node at (-0.5,11.5)[below left=-3pt]{\scriptsize 4};
	\node at (1.5,12.5)[below right=-3pt]{\scriptsize 1};
	\node at (0.5,12.5)[below left=-3pt]{\scriptsize 2};
	\node at (-0.5,12.5)[below right=-3pt]{\scriptsize 3};
	\node at (0.5,13.5)[below right=-3pt]{\scriptsize 3};
	\node at (-0.5,13.5)[below left=-3pt]{\scriptsize 2};
	\node at (0,14.5)[below=-2pt]{\scriptsize 4};
	\node at (0,15.5)[below=-2pt]{\scriptsize 5};
	\node at (0,16.5)[below=-2pt]{\scriptsize 6};
	\node at (0,17.5)[below=-2pt]{\scriptsize 7};
	\node at (0,-9)[above=2pt]{\tiny$v_0$};
	\node at (0,-8)[above=2pt]{\tiny$v_1$};
	\node at (0,-7)[above=2pt]{\tiny$v_2$};
	\node at (0,-6)[above=2pt]{\tiny$v_3$};
	\node at (0,-5)[above=2pt]{\tiny$v_4$};
	\node at (1,-4)[above=2pt]{\tiny$v_5'$};
	\node at (-1,-4)[above=2pt]{\tiny$v_5''$};
	\node at (0,-3)[above=2pt]{\tiny$v_6'$};
	\node at (-2,-3)[above=2pt]{\tiny$v_6''$};
	\node at (1,-2)[above=2pt]{\tiny$v_7'$};
	\node at (-1,-2)[above=2pt]{\tiny$v_7''$};
	\node at (2,-1)[above=2pt]{\tiny$v_8'$};
	\node at (0,-1)[above=2pt]{\tiny$v_8''$};
	\node at (3,0)[above=2pt]{\tiny$v_9$};
	\node at (1,0)[above=2pt]{\tiny$v_9'$};
	\node at (-1,0)[above=2pt]{\tiny$v_9''$};
	\node at (4,1)[above=2pt]{\tiny$v_{10}$};
	\node at (2,1)[above=2pt]{\tiny$v_{10}'$};
	\node at (0,1)[above=2pt]{\tiny$v_{10}''$};
	\node at (3,2)[above=2pt]{\tiny$v_{11}$};
	\node at (1,2)[above=2pt]{\tiny$v_{11}'$};
	\node at (-1,2)[above=2pt]{\tiny$v_{11}''$};
	\node at (2,3)[above=2pt]{\tiny$v_{12}$};
	\node at (0,3)[above=2pt]{\tiny$v_{12}'$};
	\node at (-2,3)[above=2pt]{\tiny$v_{12}''$};
	\node at (1,4)[above=2pt]{\tiny$v_{13}$};
	\node at (0,4)[above=-.5pt]{\tiny$v_{13}'$};
	\node at (-1,4)[above=0.5pt]{\tiny$v_{13}''$};
	\node at (1,5)[above=2pt]{\tiny$v_{14}''$};
	\node at (0,5)[above=-1pt]{\tiny$v_{14}'$};
	\node at (-1,5)[above=1pt]{\tiny$v_{14}$};
	\node at (2,6)[above=2pt]{\tiny$v_{15}''$};
	\node at (0,6)[above=2pt]{\tiny$v_{15}'$};
	\node at (-2,6)[above=2pt]{\tiny$v_{15}$};
	\node at (1,7)[above=2pt]{\tiny$v_{16}''$};
	\node at (-1,7)[above=2pt]{\tiny$v_{16}'$};
	\node at (-3,7)[above=2pt]{\tiny$v_{16}$};
	\node at (0,8)[above=2pt]{\tiny$v_{17}''$};
	\node at (-2,8)[above=2pt]{\tiny$v_{17}'$};
	\node at (-4,8)[above=2pt]{\tiny$v_{17}$};
	\node at (1,9)[above=2pt]{\tiny$v_{18}''$};
	\node at (-1,9)[above=2pt]{\tiny$v_{18}'$};
	\node at (-3,9)[above=2pt]{\tiny$v_{18}$};
	\node at (0,10)[above=2pt]{\tiny$v_{19}''$};
	\node at (-2,10)[above=2pt]{\tiny$v_{19}'$};
	\node at (1,11)[above=2pt]{\tiny$v_{20}''$};
	\node at (-1,11)[above=2pt]{\tiny$v_{20}'$};
	\node at (2,12)[above=2pt]{\tiny$v_{21}''$};
	\node at (0,12)[above=2pt]{\tiny$v_{21}'$};
	\node at (1,13)[above=2pt]{\tiny$v_{22}''$};
	\node at (-1,13)[above=2pt]{\tiny$v_{22}'$};
	\node at (0,14)[above=2pt]{\tiny$v_{23}$};
	\node at (0,15)[above=2pt]{\tiny$v_{24}$};
	\node at (0,16)[above=2pt]{\tiny$v_{25}$};
	\node at (0,17)[above=2pt]{\tiny$v_{26}$};
	\node at (0,18)[above=2pt]{\tiny$v_{27}$};
\end{tikzpicture}
\label{eq:Freudenthal-Hasse-diagram}
\end{equation}
Here we used the same conventions as for the diagram in \eqref{eq:Cayley-Hasse-diagram}. Note that the center consists of the following commutative cube:
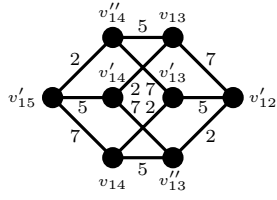
\vspace{-1em}
\[
\begin{tikzpicture}[rotate=90,scale=.8,style=very thick,baseline=0.25em]
	\draw 
	(0,3) -- (1,4) -- (1,5) -- (0,6) -- (-1,5) -- (-1,4) -- (0,3) -- (0,4) -- (1,5)
	(0,4) -- (-1,5) (0,6) -- (0,5) -- (1,4)  (0,5) -- (-1,4);
	\draw[black, fill=black] 
	(0,3) circle (.15)
	(1,4) circle (.15)
	(0,4) circle (.15)
	(-1,4) circle (.15)
	(1,5) circle (.15)
	(0,5) circle (.15)
	(-1,5) circle (.15)
	(0,6) circle (.15);
	\node at (0.5,3.5)[above right=-3pt]{\scriptsize 7};
	\node at (0,3.5)[below =-2.5pt]{\scriptsize 5};
	\node at (-0.5,3.5)[below right=-3pt]{\scriptsize 2};
	\node at (1,4.5)[above=-2pt]{\scriptsize 5};
	\node at (0.15,4.625){\scriptsize 2};
	\node at (0.15,4.375){\scriptsize 7};
	\node at (-0.15,4.625){\scriptsize 7};
	\node at (-0.15,4.375){\scriptsize 2};
	\node at (-1,4.5)[below=-2pt]{\scriptsize 5};
	\node at (0.5,5.5)[above left=-3pt]{\scriptsize 2};
	\node at (0,5.5)[below=-2.5pt]{\scriptsize 5};
	\node at (-0.5,5.5)[below left=-3pt]{\scriptsize 7};
	\node at (0,3)[right=2pt]{\tiny$v_{12}'$};
	\node at (1,4)[above=2pt]{\tiny$v_{13}$};
	\node at (0,4)[above=2pt]{\tiny$v_{13}'$};
	\node at (-1,4)[below=1pt]{\tiny$v_{13}''$};
	\node at (1,5)[above=2pt]{\tiny$v_{14}''$};
	\node at (0,5)[above=2pt]{\tiny$v_{14}'$};
	\node at (-1,5)[below=4pt]{\tiny$v_{14}$};
	\node at (0,6)[left=2pt]{\tiny$v_{15}'$};
\end{tikzpicture}
\]
with all southwest-northeast edges marked $2$, all west-east edges marked $5$ and all northwest-southeast edges marked $7$. We have again used the same labeling for the basis vectors as used for the Schubert classes of the cohomology of $\ECHS[7]$ in \cite{CMP_Quantum_cohomology_of_minuscule_homogeneous_spaces} in accordance with the geometric Satake correspondence.

We have the same interpretations of diagram \eqref{eq:Freudenthal-Hasse-diagram} as the actions of $\dChe_i$ and $\dChf_i$ on $\dfwtrep[1]$, as the Hasse diagram for the cohomology, and as the partial order on the elements of $\cosets$ under the Bruhat order. Although any path from left to right in \eqref{eq:Freudenthal-Hasse-diagram} gives a reduced expression for $\wP$, we fix the following:
\begin{equation}
\wP = s_7s_6s_5s_4s_3s_2s_4s_5s_6s_1s_3s_4s_2s_5s_7s_4s_3s_1s_6s_5s_4s_2s_3s_4s_5s_6s_7.
\label{eq:E7P7_wP_reduced_expression}
\end{equation}
%Whenever we consider the Freudenthal variety, the sequence $(r_i)_{i=1}^{27}$ denotes the sequence of indices of this reduced expression, i.e.~$r_1=7$, $r_2=6$ and so on. 
Note that we have $\wPinv=\wP$, which did not hold in the case of the Cayley plane.

%\input{Cluster-prelims}
%Recall that we defined the Pl\"ucker coordinates $\p_i$ in section \ref{sec:Plucker_coords_and_ring_of_dunimP}, and that we discussed the Cayley plane and the Freudenthal variety in section \ref{sec:ECHS}.

\section{The canonical models for cominuscule \texorpdfstring{$\ECHS$}{E\_n\textasciicircum{}sc/P\_n}}\label{sec:Canonical_LG_models_for_ECHS}
In this section we will define \emph{canonical mirror models} $(\mX_\can,\pot_\can)$ for the Cayley plane and the Freudenthal variety in terms of Pl\"ucker coordinates. (The terminology ``canonical'' follows that of \cite{Pech_Rietsch_Williams_Quadrics}.) In addition to the Pl\"ucker coordinate expressions, our main results are:

\begin{thm}\label{thm:ExcFam_CanAndLie_Vars}
For cominuscule $\ECHS$, the variety $\mX_\can$ is isomorphic to the open Richardson variety $\mX_\Lie=\dRichard_{\wop,\wo}$.
\end{thm}
\begin{thm}\label{thm:ExcFam_CanAndLie_LGmodels}
For cominuscule $\ECHS$, the pull-back of $\pot_\can$ under the isomorphism of Theorem \ref{thm:ExcFam_CanAndLie_Vars} equals $\pot_\Lie$. In other words, the Landau-Ginzburg model $(\mX_\can,\pot_\can)$ is isomorphic to Rietsch's Lie-theoretic Landau-Ginzburg model $(\mX_\Lie,\pot_\Lie)$.
\end{thm}

We follow the general approach as used in \cite{Pech_Rietsch_Lagrangian_Grassmannians,Pech_Rietsch_Williams_Quadrics}. In particular, we will use a presentation of the coordinate ring of $\dunimP$ derived from the results in \cite{GLS_Kac_Moody_groups_and_cluster_algebras} (which we introduced in Section \ref{sec:Plucker_coords_and_ring_of_dunimP}) and express it in terms of Pl\"ucker coordinates using representation theory. However, the varieties considered in these references are homogeneous under $\Sp_{2n}$ and $\Spin_{2n}$, and the representation theory of these groups is well-established (although by no means trivial). For the groups of type $\LGE_6$ and $\LGE_7$, the representations have a more complicated structure. Moreover, in contrast with \cite{Pech_Rietsch_Lagrangian_Grassmannians,Pech_Rietsch_Williams_Quadrics}, we will perform most calculations restricted to the open, dense torus $\opendunim$ (or more precisely: its transpose) to simplify the computations. This method is in fact similar to the approach in \cite{Pech_Rietsch_Odd_Quadrics}, where the analogous torus is considered for odd-dimensional quadrics; however, the representation-theoretic considerations there are significantly more straightforward compared to the methods that had to be developed for the exceptional cominuscule family.%This is reminiscent of our approach in \marginbox{missing ref}Chapter \ref{ch:LP_LG_for_CHS}. %This is remedied by the fact that we are only considering two groups, and not a whole family, so it is manageable to make direct computations.

We will denote the homogeneous polynomials appearing in the numerators and denominators of our superpotentials $\pot_\can$ by $q_i$; this should not cause any confusion with quantum parameters as cominuscule homogeneous spaces always have $\rk_\Z H_2(X,\Z)=1$ and thus also only a single quantum parameter, which we denote by $q$.

\subsection{The canonical model for the Cayley plane.}
\label{sec:cayley-potential}
The homogeneous polynomials in Pl\"ucker coordinates appearing in the denominators are the following:
\begin{equation}
\begin{aligned}
q_{12} &= \p_1\p_{11}''-\p_0\p_{12}'',\qquad q_{16} = \p_7'\p_9'-\p_8\p_8', \qquad q_{20} = \p_5''\p_{15}-\p_4''\p_{16}, \\
q_{24} &= \p_0(\p_{11}'\p_{13}-\p_{12}'\p_{12}'')-\p_1(\p_{10}''\p_{13}-\p_{11}''\p_{12}')+\p_2(\p_{10}''\p_{12}''-\p_{11}'\p_{11}'');
\end{aligned}\label{eq:CayleyPlane_denominators_of_potential}
\end{equation}
and the following will be numerators:
\begin{equation}
\begin{aligned}
q_{13} &= \p_2\p_{11}''-\p_0\p_{13},\qquad q_{17} = \p_7'\p_{10}'-\p_8\p_9'', \qquad q_{21} = \p_6''\p_{15}-\p_5'\p_{16}, \\
q_{25} &= \p_0\p_{11}'\p_{14}-\p_1\p_{10}''\p_{14}+\p_3(\p_{10}''\p_{12}''-\p_{11}'\p_{11}'').
\end{aligned}
\label{eq:CayleyPlane_numerators_of_potential}
\end{equation}

Using these coordinates, define the following rational map $\cmX=\udP\backslash\udG\to\C$, called the \emph{canonical superpotential}:
\begin{align}
\pot_\can ={}& \frac{\p_1}{\p_0} + \frac{\p_9'}{\p_8} 
+ \frac{q_{13}}{q_{12}} 
+ \frac{q_{17}}{q_{16}}
+ \frac{q_{21}}{q_{20}} 
+ \frac{q_{25}}{q_{24}}
+ q\frac{\p_5''}{\p_{16}}.
\label{eq:pot-for-cayley}
\end{align}

The denominators of the rational map of equation \eqref{eq:pot-for-cayley} define principal divisors on $\cmX$ which we will denote as follows:
\ali{
D_0^{(1)} &= \{\p_0=0\},\quad D_8^{(1)} = \{\p_8=0\},\quad D_{16}^{(1)} = \{p_{16}=0\}, \\
D_{12}^{(2)} &= \{q_{12} =0\},~\quad D_{16}^{(2)} = \{q_{16} = 0\},\quad D_{20}^{(2)} = \{q_{20}=0\}, \\
D_{24}^{(3)} &= \{q_{24} = 0\}.
}
Take $D=\sum D_{i}^{(j)}$ and let $\mX_\can=\cmX\setminus D$ be the complement of this anticanonical divisor. The restriction of $\pot_\can$ to $\mX_\can$ is regular and will also be denoted by $\pot_\can$. The pair $(\mX_\can,\pot_\can)$ forms the canonical LG model for the Cayley plane.

\subsection{The canonical model for the Freudenthal variety.}
\label{sec:freudenthal-potential}
The canonical superpotential for the Freudenthal variety will have the following denominators:
\begin{equation}
\begin{aligned}
q_{18} &= \p_1\p_{17}-\p_0\p_{18}, \qquad  q_{36} = \p_{10}\p_{26}-\p_{9}\p_{27}, \\
q_{27} &= \p_{0}\p_{27}-\p_{1}\p_{26}+\p_{2}\p_{25}-\p_{3}\p_{24}+\p_{4}\p_{23}-\p_{5}''\p_{22}''+\p_{6}''\p_{21}'', \\
q_{45} &= (\p_{9}\p_{11}-\p_{10}\p_{10}')\p_{25}-(\p_{8}'\p_{11}-\p_{9}'\p_{10})\p_{26}+(\p_{8}'\p_{10}'-\p_{9}\p_{9}')\p_{27}, \\
q_{36}' &= (\p_{0}\p_{14}'-\p_{1}\p_{13}''+\p_{2}\p_{12}'')\p_{22}'+(-\p_{0}\p_{13}+\p_{1}\p_{12}'-\p_{2}\p_{11}'')\p_{23} \\
&\quad+(\p_{0}\p_{12}-\p_{1}\p_{11}'+\p_{2}\p_{10}'')\p_{24}-\p_{2}\p_{9}''\p_{25}+\p_{1}\p_{9}''\p_{26}-\p_{0}\p_{9}''\p_{27}, \\
q_{54} &= (\p_{5}''\p_{7}''-\p_{6}'\p_{6}'')(\p_{20}''\p_{22}''-\p_{21}'\p_{21}'')-(\p_{4}\p_{7}''-\p_{5}'\p_{6}'')(\p_{20}''\p_{23}-\p_{21}''\p_{22}') \\
&\quad+(\p_{4}\p_{6}'-\p_{5}'\p_{5}'')(\p_{21}'\p_{23}-\p_{22}'\p_{22}'') + q_{27}'q_{27}'',
%&\quad+(-\p_{0}\p_{27}+\p_{1}\p_{26}-\p_{2}\p_{25}+\p_{3}\p_{24})(-\p_{6}''\p_{21}''-\p_{7}'\p_{20}'+\p_{7}''\p_{20}''+\p_{8}'\p_{19}'-\p_{9}\p_{18}+\p_{10}\p_{17});
\end{aligned}
\label{eq:Freudenthal_denominators_of_potential}
\end{equation}
where
\ali{
q_{27}'&= -\p_{0}\p_{27}+\p_{1}\p_{26}-\p_{2}\p_{25}+\p_{3}\p_{24}, \\
q_{27}'' &= -\p_{6}''\p_{21}''-\p_{7}'\p_{20}'+\p_{7}''\p_{20}''+\p_{8}'\p_{19}'-\p_{9}\p_{18}+\p_{10}\p_{17};
}
and the following numerators:
\begin{equation}
\begin{aligned}
q_{19} &= \p_2\p_{17}-\p_0\p_{19}', \qquad %\nonumber\\
q_{28} = \p_{5}'\p_{23}-\p_{6}'\p_{22}''+\p_{7}''\p_{21}'', \qquad %\nonumber\\
q_{37} = \p_{11}\p_{26}-\p_{10}'\p_{27}, \\
q_{46} &= (\p_{9}\p_{12}-\p_{10}\p_{11}')\p_{25}+(-\p_{8}'\p_{12}+\p_{10}\p_{10}'')\p_{26}+(\p_{8}'\p_{11}'-\p_{9}\p_{10}'')\p_{27}, \\
q_{37}' &=(\p_{0}\p_{15}'-\p_{1}\p_{14}+\p_{3}\p_{12}'')\p_{22}' + (-\p_{0}\p_{14}''+\p_{1}\p_{13}'-\p_{3}\p_{11}'')\p_{23} \\
&\quad+\p_{3}\p_{10}''\p_{24}+ (\p_{0}\p_{12}-\p_{1}\p_{11}'-\p_{3}\p_{9}'')\p_{25} +\p_{1}\p_{10}''\p_{26}-\p_{0}\p_{10}''\p_{27}, \\
q_{55} &=(\p_{5}''\p_{8}''-\p_{6}''\p_{7}')(\p_{20}''\p_{22}''-\p_{21}'\p_{21}'')-\p_{4}\p_{8}''(\p_{20}''\p_{23}-\p_{21}''\p_{22}')+\p_{4}\p_{7}'(\p_{21}'\p_{23}-\p_{22}'\p_{22}'') \\
&\quad+ \p_{5}'\p_{24}(-\p_{5}''\p_{21}'+\p_{6}''\p_{20}'')+\p_{4}\p_{24}(-\p_{4}\p_{23}+\p_{5}'\p_{22}'+\p_{5}''\p_{22}''-\p_{6}''\p_{21}'') \\
&\quad+ (\p_{4}\p_{24}-\p_{7}'\p_{21}'+\p_{8}''\p_{20}'')(-\p_{0}\p_{27}+\p_{1}\p_{26}-\p_{2}\p_{25}+\p_{3}\p_{24}).
\end{aligned}
\label{eq:Freudenthal_numerators_of_potential}
\end{equation}
The \emph{canonical superpotential} is the following map $\cmX\to\C$:
\begin{align}
\pot_\can = \frac{\p_1}{\p_0} + \frac{q_{19}}{q_{18}} + \frac{q_{28}}{q_{27}} + \frac{q_{37}}{q_{36}} +\frac{q_{37}'}{q_{36}'} + \frac{q_{46}}{q_{45}} + \frac{q_{55}}{q_{54}} + q\frac{p_{10}}{p_{27}}
\end{align}
The sum of the loci where the denominators of this superpotential are zero forms an anticanonical divisor $D$; we will write $\mX_\can = \cmX\setminus D$ for its complement. We will also write $\pot_\can$ for the regular restriction of the superpotential to $\mX_\can$. The pair $(\mX_\can,\pot_\can)$ is the canonical LG model for the Freudenthal variety.

\begin{rem}  \label{rem:quantum-pieri}
We observed that the numerators and denominators in our superpotentials are related in the quantum cohomology ring of $\X$ as follows. After identifying Pl\"ucker coordinates with Schubert classes via the geometric Satake correspondence, we may interpret the numerator and denominator of each term in our superpotentials as elements of the quantum cohomology ring by interpreting all products as quantum products. We observe after these identifications that for each term in our superpotentials, its numerator equals the quantum product of its denominator with the hyperplane class $\sigma_1$, scaled by the degree of the denominator (or numerator) as a polynomial in Pl\"ucker coordinates. For example, for our Cayley plane superpotential in equation \eqref{eq:pot-for-cayley}, we compute that $q_{25}=3\sigma_1*q_{24}$ where all products (including those in the Pl\"ucker coordinate expressions of $q_{24}$ and $q_{25}$) are interpreted as quantum products. 

In order to verify this relationship between the numerators and denominators of the superpotential, we used the description of the quantum cohomology ring for the Cayley plane presented in Proposition 3 and Theorem 31 of \cite{CMP_Quantum_cohomology_of_minuscule_homogeneous_spaces} as well as the description of the quantum cohomology ring of the Freudenthal variety using Theorems 6 and 34 of \cite{CMP_Quantum_cohomology_of_minuscule_homogeneous_spaces}. 

Similar statements have appeared in type $\LGA$ in e.g.~\cite{Marsh_Rietsch_Grassmannians,Kalashnikov_Plucker_coord_mirror_type_A_flag}, and in type $B$ in \cite[Section 9]{Pech_Rietsch_Odd_Quadrics}. This scaling by degree was not observed in the type-$\LGA$ Grassmannians $\Gr(k,n)$ because there each term was a ratio of degree-$1$ expressions in the Pl\"ucker coordinates, while the type $B$ superpotential had one term ($W_{m-1}$ in the notation of \cite{Pech_Rietsch_Odd_Quadrics}) which contributed only $p_1$, despite being a ratio of degree-$2$ polynomials. Furthermore, in type $C$, \cite{Pech_Rietsch_Lagrangian_Grassmannians} observed that each term of their superpotential contributes only one hyperplane class $p_1$. These observations suggest the existence of a general construction of canonical Pl\"ucker coordinate superpotentials for certain homogeneous spaces, and we are investigating this topic in an upcoming article.
\end{rem}

%Because of these observation, we make the following conjecture which we plan to investigate in a future article. While we state it only for cominuscule spaces for simply-laced Lie groups $G$, we believe that a suitable analogue will hold for more general cominuscule spaces, with possible extensions to certain other flag varieties as well.

%\begin{conj}\label{conj:LG-model-cominuscule}
%  Let $\X$ be a cominuscule projective homogeneous space for a simply-laced Lie group $G$. Then $\X$ has a Landau-Ginzburg model $(\X_\can,\pot_\can)$ expressed in Pl\"ucker coordinates as follows: Let $F$ be the set of frozen variables of the cluster structure on $\cmX$ given by \cite{GLS_partial_flag_varieties_and_preprojective_algebras}, expressed in Pl\"ucker coordinates. Then $F$ defines an anticanonical divisor $D$, and we set $\X_\can=\cmX\setminus D$, where $\cmX=\dP\backslash\dG$. Lastly, we set $\pot_\can=\sum_{f\in F} \frac{\deg(f)\,\scs_1*f}{f}$, where the degree is taken as a polynomial in Pl\"ucker coordinates, $\scs_1$ is the (quantum) Schubert class of a hyperplane, and $*$ denotes the quantum product.
%\end{conj}

%We remark that all but one of the frozen variables correspond in the construction of \cite{GLS_partial_flag_varieties_and_preprojective_algebras} to nodes of the Dynkin diagram. If we label each frozen variable by its degree when expressed as a polynomial in Pl\"ucker coordinates, we observed that this labelling corresponds to the coefficient of the corresponding simple root in the longest root. We plan to investigate this in an upcoming article.

\subsection{Overview of the isomorphisms}\label{sec:Mirror_model_isomorphisms}
The proof of Theorem \ref{thm:ExcFam_CanAndLie_Vars} requires an auxiliary variety, $\dunimP$, which we defined in Subsection \ref{sec:Lie_theoretic_Mirror}. Namely, we will construct the isomorphism as the following composition of maps:
\begin{equation}
\mX_\Lie \overset\varphi\longrightarrow\dunimP \overset\pi\longrightarrow \cmX\supset\mX_\can.
\label{eq:ExcFam_OverviewOfIsomorphism}
\end{equation}
Recall from equation \eqref{eq:unique_decomps_unipotents} that
\[
\dunimP = \dunim\cap \dborelp\bwop\bwo\dborelp.
\]
The map $\pi:\dunimP\to\cmX$ is the restriction of the quotient map $\udG\to\udP\backslash\udG=\cmX$ to~$\dunimP$. (Note that we identify $\dunim\subset\dG$ and its universal cover in $\udG$.)

The map $\varphi:\mX_\Lie\overset\sim\longrightarrow\dunimP$ is an isomorphism that cannot be expressed so explicitly. It is given by mapping $u_+\bwop\dborelm\in\mX_\Lie$ (with $u_+\in\dunip$) to the unique $u_-\in\dunimP$ such that $u_+\bwop u_-\in\decomps$. Here, the variety $\decomps$ was defined in equation \eqref{eq:decomps} as
\[
\decomps = \dborelm\bwo^{-1}\cap\dunip\invdtorus\bwop\dunim~\subset~\dG,
\]
and we noted in Lemma \ref{lem:z_has_unique_decomposition} that each $z\in\decomps$ has a unique decomposition as $z=u_+t\bwop u_-$ with $u_+\in\dunip$, $t\in\invdtorus$ and $u_-\in\dunimP$. The map $\varphi$ can also be obtained by composing the isomorphism $\RisoZ:\mX_\Lie\times\invdtorus\to\decomps$ of equation \eqref{eq:Decomposition_Isomorphism} mentioned in Remark \ref{rem:Decomposition_Isomorphism} with the isomorphism $\decomps\to\dunimP\times\invdtorus$ sending $z$ to the unique $(u_-,t)$ such that $z=u_+t\bwop u_-$ for some $u_+\in\dunip$, and then restricting the resulting map $\mX_\Lie\times\invdtorus\to\dunimP\times\invdtorus$ to the unit element in $\invdtorus$.

To show that $\pi$ is an isomorphism between $\dunimP$ and $\mX_\can\subset\cmX$ we use the description of the coordinate ring of $\dunipP=(\dunimP)^\T$ given by \cite{GLS_Kac_Moody_groups_and_cluster_algebras}, which we discussed in Subsection \ref{sec:Plucker_coords_and_ring_of_dunimP}. The first step for this is to express the coordinate ring in terms of the Pl\"ucker coordinates.

It turns out that there is an easy criterion for a generalized Pl\"ucker coordinate $\p_i$ defined in equation \eqref{eq:Generalized-Plucker-coordinates} to coincide with a coordinate $\pGLS_m$ used by \cite{GLS_Kac_Moody_groups_and_cluster_algebras} under transposition. Recall from Subsection \ref{sec:Plucker_coords_and_ring_of_dunimP} that $\pGLS_m$ was defined as the dual map associated to a basis element $\wPdChe(m)$ for $\dudrt[-\wPdrt(m)]$, where $\wPdrt(m)$ is the $m$th positive root mapped to a negative root by $\wP$, see equation \eqref{eq:wP_positive_roots_mapped_to_negative_roots}. 
\begin{rem}
This criterion holds for general cominuscule homogeneous spaces, and we will prove it in this context. Thus, until the end of this subsection, $\X=\G/\P_k$ refers to a general cominuscule homogeneous space with the $k$th vertex of the Dynkin diagram corresponding to a cominuscule fundamental weight.
\end{rem}

We first need the following characterization of the set $\wPpdroots=\{\wPdrt(m)~|~m=1,\ldots,\ell\}$, and of its complement $\pdrootsP$ which coincides with the positive roots mapped to negative roots by $\wop$. This characterization is a classical result, although not readily available in literature, so we will include its proof using fundamental results for the reader's convenience. It can also easily be obtained from the bijection between Weyl group elements and inversion sets of roots, see e.g.~Section 3.1 of \cite{BCMP_Chevalley_formula_for_equivariant_quantum_K-theory_of_com_vars}.
%This characterization can in essence be traced back to the fact that the length of a Weyl group element coincides with the number of positive roots mapped to negative roots, see e.g.~Section 3.1 of \cite{BCMP_Chevalley_formula_for_equivariant_quantum_K-theory_of_com_vars} where it is phrased as a bijection between Weyl group elements and inversion sets of roots. We will include the proof in any case for the reader's convenience.
\begin{lem}\label{lem:characterization_of_wP_and_P_positive_coroots}
Every positive root $\drt=\sum_{i=1}^n d_i\sdr_i \in \pdroots=\wPpdroots\sqcup\pdrootsP$ has $d_k\in\{0,1\}$. Moreover, $\drt\in\pdrootsP$ if and only if $d_k=0$; and $\drt\in\wPpdroots$ if and only if $d_k=1$.
\end{lem}
\begin{proof}
The fact that $\dfwt[k]$ is minuscule is well-known to be equivalent to the fact that the longest root $\drtzero\in\droots$ has coefficient $1$ in front of $\sdr_k$ in the decomposition of $\drtzero$ in terms of simple roots. (See e.g.~exercise VI.24(c) of \cite{Bourbaki}.) Thus, we know that $d_k\in\{0,1\}$.

Now, it is well known that there is a sequence of (not necessarily distinct) simple roots $(\sdr_{i_j})_{j=1}^h$, where $h=\height(\drt)=\sum_{i=1}^n d_i$ is the height of $\drt$, such that we have $\drt=\sum_{j=1}^h \sdr_{i_j}$ and such that for every $h'\le h$ the partial sum $\sum_{j=1}^{h'}\sdr_{i_j}\in\droots$ is a root as well. (See Corollary 10.2 of \cite{Humphreys_Lie_Algebras}.) 

Suppose $d_k=0$, then we clearly have that $i_j\neq k$ for all $j$. Since $\{\sdr_i~|~i\neq k\}\subset\pdrootsP$ and $\pdrootsP$ is bracket closed, every partial sum lies in $\pdrootsP$ and thus $\drt$ as well. Conversely, since $\pdrootsP$ can be considered as the positive roots of the root system with as Dynkin diagram the diagram of $\dG$ with the $k$th vertex removed, we have that $\drt\in\pdrootsP$ implies that $d_k=0$.

As $\wPpdroots$ is the complement of $\pdrootsP$ in $\pdroots$, the statement follows.
\end{proof}
Recall that we defined a $\Nzero\sdroots$-grading on the completed universal enveloping algebra $\dCUEAp$ of $\dup$ in Section \ref{sec:Plucker_coords_and_ring_of_dunimP}. Here, $\Nzero\sdroots\cong\Nzero^n$ is the positive root lattice of all non-negative integral linear combinations of the simple roots $\sdr_i$. We defined $\dCUEAp(\drt)\subset\dCUEAp$ as the subspace spanned by pure tensors $\dChe_{i_1}\cdots\dChe_{i_j}$ such that exactly $d_i$ of the factors have index $i$, where $\drt=\sum_{i=1}^n d_i\sdr_i\in\Nzero\sdroots$. We noted there that $\dudrt[+\drt]\subset\dCUEAp(\drt)$ for $\drt\in\pdroots$, so we directly conclude that:
\begin{cor}\label{cor:PBW_basis_elts_degrees}
For $j\in\{1,\ldots,\ellwo\}$, $\wPdChe(j)\in\dCUEAp(\wPdrt(j))$. Let $d_i\in\Nzero$ such that $\wPdrt(j) = \sum_{i=1}^n d_i\sdr_i$. If $j\in\{1,\ldots,\ellwP\}$, then $d_k=1$. Else, i.e.~$j\in\{\ellwP+1,\ldots,\ellwo\}$, we have $d_k=0$.
%$j\in\{1,\ldots,\ellwP\}$, $\wPdChe(j)\in\dCUEAp(\drt)$ with $\drt=\wPdrt(j)\in\wPpdroots$ and thus $d_k=1$. For $j\in\{\ellwP+1,\ldots,\ellwo\}$, $\wPdChe(j)\in\dCUEAp(\drt)$ with $\drt\in\pdrootsP$ and thus $d_k=0$.
\end{cor}
We also need the following simple observation regarding the action of the elements of the completed universal enveloping algebra on $\PluckerRep$:
\begin{lem}\label{lem:action_of_graded_universal_enveloping_elements_on_weight_vectors}
If $u\in\dCUEAp(\drt)$ for $\drt\in\Nzero\sdroots$ and $v\in\PluckerRep$ a vector of weight $\mu$, then $u\cdot v$ is either a vector of weight $\mu+\drt$ or the zero vector. 
\end{lem}
\begin{proof}
By Theorem \ref{thm:Green_structure_minuscule_reps}, the statement holds for the Chevalley generators $\dChe_i\in\dCUEAp(\sdr_i)$. By definition $u\in\dCUEAp(\drt)$ for $\drt=\sum_{i=1}^n d_i\sdr_i$ can be decomposed into a linear combination of pure tensors in which each $\dChe_i$ appears $d_i$ times. The statement now follows by induction.
\end{proof}
Now, recall that the Pl\"ucker coordinates $\p_i:g\mapsto v_0^*(g\cdot v_i)$ are defined using the minuscule representation $\PluckerRep$ and its dual, where $-\wo\cdot\dfwt[k]=-\wP\cdot\dfwt[k] = \dfwt[\DynkinSymmetry(k)]$; see the discussion of Pl\"ucker coordinates of Section \ref{sec:Plucker_coords_and_ring_of_dunimP}. The criterion relating Pl\"ucker coordinates to the dual PBW basis defined in \cite{GLS_Kac_Moody_groups_and_cluster_algebras} is given as follows:
\begin{prop}\label{prop:Criterion_Plucker_coordinates_coincide_with_GLS}
If $v_i$ is of weight $-\dfwt[k]+\drt(v_i)$ and $\drt(v_i)=\wPdrt(m)\in\wPpdroots$, then $u\mapsto\p_i(u^\T)$ coincides up to a constant with $\pGLS_m$ on $\dunipP$.
\end{prop}
\begin{proof}
Let $u\in\dunipP\hookrightarrow\dCUEAp$ and consider the decomposition of $u$ with respect to the PBW basis. If the coefficient in front of $\wPdChe(m)$ is $c$, then we have by definition that $\pGLS_m(u)=c$.

On the other hand, $\p_i(u^\T)$ is the coefficient in front of $v_0$ of the projection of $u^\T\cdot v_i$ to $\C v_0\subset\PluckerRep$. It suffices to show that $\wPdChe(m)^\T$ is the only transposed PBW basis element mapping $v_i$ to a non-zero multiple of $v_0$. Equivalently, we need to show that $\wPdChe(m)$ is the only PBW basis element mapping $v_0$ to a non-zero multiple of $v_i$. One can see this equivalence as follows: For Chevalley generators, Theorem \ref{thm:Green_structure_minuscule_reps} implies that $\dChe_j$ maps $v$ to $v'$ if and only if $(\dChe_i)^\T=\dChf_j$ maps $v'$ to $v$; by induction, the same holds for all pure tensors; finally, each PBW basis element is a linear combination of these pure tensors, so that the equivalence follows.

By Lemma \ref{lem:action_of_graded_universal_enveloping_elements_on_weight_vectors}, only the elements of $\dCUEAp\bigl(\wPdrt(m)\bigr)$ map $v_0$ to $\C v_i$. Thus, we can restrict ourselves to PBW basis elements $\wPdChe(\bm)$ lying in $\dCUEAp\bigl(\wPdrt(m)\bigr)$. Now, $\wPdChe(m)\in\dudrt[+\wPdrt(m)]\subset\dCUEAp\bigl(\wPdrt(m)\bigr)$. Moreover, Lemma \ref{lem:characterization_of_wP_and_P_positive_coroots} implies that $d_k=1$ in the decomposition $\wPdrt(m) = \sum_{i=1}^n d_i\sdr_i$.

First, let us consider $\wPdChe(\bm)\in\dCUEAp\bigl(\wPdrt(m)\bigr)$ not equal to $\wPdChe(m)$. Since for $j\neq m$ we have $\wPdChe(j)\notin\dCUEAp\bigl(\wPdrt(m)\bigr)$, this assumption means that we are considering those $\wPdChe(\bm)$ with $\bm=(m_1,m_2,\ldots,m_\ellwo)$ satisfying $\sum_{j=1}^{\ellwo} m_j\ge2$. Corollary \ref{cor:PBW_basis_elts_degrees} implies that of the coefficients $m_1,\ldots,m_\ellwP$ (where $\ell=\ell(\wP)$) exactly one is non-zero, and that coefficient equals $1$, as these are the only coefficients contributing to $d_k=1$. Thus, at least one of $m_{\ellwP+1},\ldots,m_{\ellwo}$ is non-zero; let $j\in\{\ellwP+1,\ldots,\ellwo\}$ be the largest index such that $m_j\neq0$. Recalling the definition of the PBW basis elements from equation \eqref{eq:PBW_basis_elements}, we find that $\wPdChe(\bm)\cdot v_0= \bigl(\wPdChe(1)\bigr)^{m_1}\cdots\bigl(\wPdChe(j)\bigr)^{m_{j}}\cdot v_0$. However, $\wPdChe(j)\cdot v_0=0$: we have by Corollary \ref{cor:PBW_basis_elts_degrees} that $\wPdChe(j)\in\dCUEAp\bigl(\wPdrt(m)\bigr)$ where $\wPdrt=\sum_{j'=1}^n d_{j'}'\sdr_i$ with $d_k'=0$, which implies that $\wPdChe(j)$ can be written as a linear combination of pure tensors involving the Chevalley generators $\dChe_{i'}$ for $i'\neq k$, but by Theorem \ref{thm:Green_structure_minuscule_reps} we know that $\dChe_{i'}\cdot v_0=0$ for $i'\neq k$.

Thus, it remains to show that $\wPdChe(m)\cdot v_0$ is a non-zero multiple of $v_i$. Suppose it acts trivially, then we find that $\C v_i\cap\dCUEAp v_0 = \{0\}$, which contradicts the irreducibility of $\PluckerRep$.
\end{proof}
Using this criterion, we will be able to identify all of the generators of the coordinate ring of $\dunimP$ given by \cite{GLS_Kac_Moody_groups_and_cluster_algebras} with generalized Pl\"ucker coordinates.

Moreover, we can identify the generalized minors defined in \eqref{eq:GLS_def_of_minor} with the denomi\-nators of the superpotentials for the cominuscule exceptional homogeneous spaces as given in \eqref{eq:CayleyPlane_denominators_of_potential} and \eqref{eq:Freudenthal_denominators_of_potential}. We will identify these functions by checking that their expressions in terms of the coordinates $\{a_i\}_{i=1}^\ell$ on the algebraic torus $\opendunim$ of \eqref{eq:df_opendunim} agree.

%\pagebreak
\subsection{The coordinate ring of \texorpdfstring{$\dunimP$}{U\_-\textasciicircum{}P} for the Cayley plane}\label{sec:coord_ring_dunimP_E6}
Recall the conventions for the Cayley plane $\ECHS[6]$ from Section \ref{sec:ECHS}. In particular, we have $\ellwP=\ell(\wP)=16$. Applying this to the coordinate ring of $\dunipP$ described in Proposition \ref{prop:GLS_coord_ring_dunipP} (see Proposition 8.5 of \cite{GLS_Kac_Moody_groups_and_cluster_algebras}), we find that it is generated by the dual PBW elements $\pGLS_m$, $m\in\{1,\ldots,16\}$, and localized at the generalized minors $\minor_{\dfwt[j],\wPinv(\dfwt[j])}$, $j\in\{1,\ldots,6\}$.
%\begin{equation}
%\C[\pGLS_{1},\ldots,\pGLS_{16}][\minor_{\dfwt[1],\wPinv(\dfwt[1])}^{-1},\ldots,\minor_{\dfwt[6],\wPinv(\dfwt[6])}^{-1}].
%\label{eq:CayleyPlane_GLS_expression_for_coordinate_ring_of_dunimP}
%\end{equation}
Using the criterion in Proposition \ref{prop:Criterion_Plucker_coordinates_coincide_with_GLS}, we find that the generators $\pGLS_m$ can be expressed in terms of Pl\"ucker coordinates as follows:
\begin{lem}\label{lem:CayleyPlane_Plucker_coordinates_coincide_with_GLS}
On $\dunipP\subset\udG$, we have the following identifications:
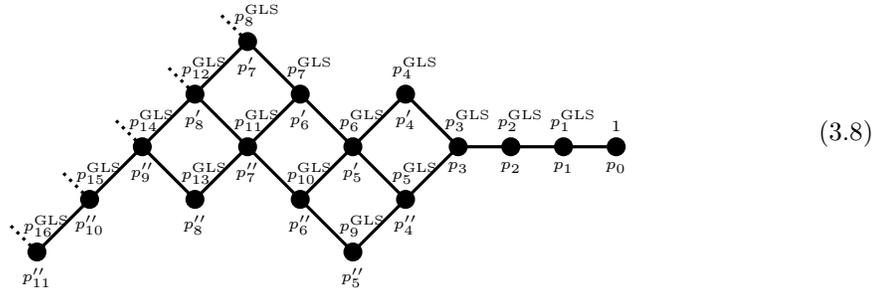
\begin{equation}
\begin{tikzpicture}[rotate=90,scale=.7,style=very thick,baseline=0.25em]
	\draw 
	(0,-8)--(0,-5)--(-2,-3)--(0,-1)--(-1,0)--(0,1)--(-2,3)%--(0,5)--(0,8)
	(0,-5)--(1,-4)--(0,-3)--(2,-1) %(3,0)--(0,3)--(1,4)--(0,5)
	(-1,-4)--(0,-3)--(-1,-2) %(-1,4)--(0,3)--(-1,2)
	%(1,2)--(0,1)--(2,-1) 
	(0,1)--(2,-1)
	(1,-2)--(0,-1)--(1,0)%(2,1)
	;
	\draw[dotted]
	(2,-1) -- (2.5,-0.5)
	(1,0) -- (1.5,0.5)
	(0,1) -- (0.5,1.5)
	(-1,2) -- (-0.5,2.5)
	(-2,3) -- (-1.5,3.5)
	;
	\draw[black, fill=black] 
	(0,-8) circle (.15) 
	(0,-7) circle (.15)
	(0,-6) circle (.15)
	(0,-5) circle (.15)
	(1,-4) circle (.15)
	(-1,-4) circle (.15)
	(0,-3) circle (.15)
	(-2,-3) circle (.15)
	(1,-2) circle (.15)
	(-1,-2) circle (.15)
	(2,-1) circle (.15)
	(0,-1) circle (.15)
	%(3,0) circle (.15)
	(1,0) circle (.15)
	(-1,0) circle (.15)
	%(2,1) circle (.15)
	(0,1) circle (.15)
	%(1,2) circle (.15)
	(-1,2) circle (.15)
	%(0,3) circle (.15)
	(-2,3) circle (.15)
	%(1,4) circle (.15)
	%(-1,4) circle (.15)
	%(0,5) circle (.15)
	%(0,6) circle (.15)
	%(0,7) circle (.15)
	%(0,8) circle (.15)
	;
	\node at (0,-8)[below=2pt]{\tiny$\p_0$};
	\node at (0,-7)[below=2pt]{\tiny$\p_1$};
	\node at (0,-6)[below=2pt]{\tiny$\p_2$};
	\node at (0,-5)[below=2pt]{\tiny$\p_3$};
	\node at (1,-4)[below=2pt]{\tiny$\p_4'$};
	\node at (-1,-4)[below=2pt]{\tiny$\p_4''$};
	\node at (0,-3)[below=2pt]{\tiny$\p_5'$};
	\node at (-2,-3)[below=2pt]{\tiny$\p_5''$};
	\node at (1,-2)[below=2pt]{\tiny$\p_6'$};
	\node at (-1,-2)[below=2pt]{\tiny$\p_6''$};
	\node at (2,-1)[below=2pt]{\tiny$\p_7'$};
	\node at (0,-1)[below=2pt]{\tiny$\p_7''$};
	%\node at (3,0)[below=2pt]{\tiny$\p_8$};
	\node at (1,0)[below=2pt]{\tiny$\p_8'$};
	\node at (-1,0)[below=2pt]{\tiny$\p_8''$};
	%\node at (2,1)[below=2pt]{\tiny$\p_9'$};
	\node at (0,1)[below=2pt]{\tiny$\p_9''$};
	%\node at (1,2)[below=2pt]{\tiny$\p_{10}'$};
	\node at (-1,2)[below=2pt]{\tiny$\p_{10}''$};
	%\node at (0,3)[below=2pt]{\tiny$\p_{11}'$};
	\node at (-2,3)[below=2pt]{\tiny$\p_{11}''$};
	%\node at (1,4)[below=2pt]{\tiny$\p_{12}'$};
	%\node at (-1,4)[below=2pt]{\tiny$\p_{12}''$};
	%\node at (0,5)[below=2pt]{\tiny$\p_{13}$};	
	%\node at (0,6)[below=2pt]{\tiny$\p_{14}$};
	%\node at (0,7)[below=2pt]{\tiny$\p_{15}$};
	%\node at (0,8)[below=2pt]{\tiny$\p_{16}$};
%-----------------------------------------	
	\node at (0,-8)[above=2pt]{\tiny$1$};
	\node at (0,-7)[above=2pt]{\tiny$\pGLS_1\hspace{-1em}$};
	\node at (0,-6)[above=2pt]{\tiny$\pGLS_2\hspace{-1em}$};
	\node at (0,-5)[above=2pt]{\tiny$\pGLS_3\hspace{-1em}$};
	\node at (1,-4)[above=2pt]{\tiny$\pGLS_4\hspace{-1em}$};
	\node at (-1,-4)[above=2pt]{\tiny$\pGLS_5\hspace{-1em}$};
	\node at (0,-3)[above=2pt]{\tiny$\pGLS_6\hspace{-1em}$};
	\node at (-2,-3)[above=2pt]{\tiny$\pGLS_9\hspace{-1em}$};
	\node at (1,-2)[above=2pt]{\tiny$\pGLS_7\hspace{-1em}$};
	\node at (-1,-2)[above=2pt]{\tiny$\pGLS_{10}\hspace{-1em}$};
	\node at (2,-1)[above=2pt]{\tiny$\pGLS_8\hspace{-1em}$};
	\node at (0,-1)[above=2pt]{\tiny$\pGLS_{11}\hspace{-1em}$};
	\node at (1,0)[above=2pt]{\tiny$\pGLS_{12}\hspace{-1em}$};
	\node at (-1,0)[above=2pt]{\tiny$\pGLS_{13}\hspace{-1em}$};
	\node at (0,1)[above=2pt]{\tiny$\pGLS_{14}\hspace{-1em}$};
	\node at (-1,2)[above=2pt]{\tiny$\pGLS_{15}\hspace{-1em}$};
	\node at (-2,3)[above=2pt]{\tiny$\pGLS_{16}\hspace{-1em}$};
\end{tikzpicture}
\label{eq:Cayley-Plucker-and-GLS-diagram}
\end{equation}
Here, a vertex marked with $\p_i^{(j)}$ below and $\pGLS_{m}$ above denotes that $\p_i^{(j)}(u_+^\T)$ equals $\pGLS_m(u_+)$ up to a constant for $u_+\in\dunimP$.
\end{lem}
For example, we have that $\p_7''(u_+^\T)$ equals $\pGLS_{11}(u_+)$ up to a constant for $u_+\in\dunipP$. In particular, the rightmost vertex tells us that $\p_0(u_+^\T)$ is constant, and it is easy to see that this constant is $1$. Note that the diagram of equation \eqref{eq:Cayley-Plucker-and-GLS-diagram} is a subdiagram of the one in equation \eqref{eq:Cayley-Hasse-diagram}. Upon removing the rightmost vertex, this subdiagram is in fact isomorphic to the Hasse diagram of $\OG(5,10)=\PSO_{10}/\P_5=\LGD_5^\SC/\P_5$. This variety is also known as a \emph{spinor variety}. See \cite[Section 4]{Iliev_Manivel_Chow_ring_of_the_Cayley_plane}, for its role in the geometry of the Cayley plane. The appearance of this variety is not too much of a surprise, as we obtain the Dynkin diagram of type $\LGD_5$ upon removing the sixth vertex from the diagram of type $\LGE_6$; in other words $\weylp$ is of type $\LGD_5$.%, while the diagram of equation \eqref{eq:Cayley-Hasse-diagram} is isomorphic to the Hasse diagram of $\cayley$.
\begin{proof}
We can compute the weights of the minuscule fundamental weight representation $\PluckerRep[6]$ directly using diagram \eqref{eq:Cayley-Hasse-diagram} and Theorem \ref{thm:Green_structure_minuscule_reps}. On the other hand, the positive roots $\wPdrt(m)$, $m\in\{1,\ldots,16\}$ are straightforward to compute using the expression in equation \eqref{eq:wP_positive_roots_mapped_to_negative_roots}. A comparison of the results gives the following equalities:
\[
\hspace{-2em}\begin{array}{llll}
\drt(v_1)=\wPdrt(1) & 
\drt(v_2)=\wPdrt(2) &
\drt(v_3)=\wPdrt(3) & 
\drt(v_4')=\wPdrt(4) \\
\drt(v_4'')=\wPdrt(5) & 
\drt(v_5')=\wPdrt(6) &
\drt(v_5'')=\wPdrt(9) & 
\drt(v_6')=\wPdrt(7) \\
\drt(v_6'')=\wPdrt(10) & 
\drt(v_7')=\wPdrt(8) &
\drt(v_7'')=\wPdrt(11) & 
\drt(v_8')=\wPdrt(12) \\
\drt(v_8'')=\wPdrt(13) & 
\drt(v_9'')=\wPdrt(14) &
\drt(v_{10}'')=\wPdrt(15) & 
\drt(v_{11}'')=\wPdrt(16)
\end{array}
\]
Using Proposition \ref{prop:Criterion_Plucker_coordinates_coincide_with_GLS}, we find that $\p_i^{(j)}(u_+^\T)$ equals $\pGLS_m(u_+)$ up to a constant as indicated in \eqref{eq:Cayley-Plucker-and-GLS-diagram} for $u_+\in\dunipP$. The only exception is the identification $\p_0(u_-) = 1$, but this identity is obvious from the fact that $u_+^\T\in\dunim$ is a lower unipotent element and thus acts trivially on the lowest weight vector $v_0$.
\end{proof}
To simplify expressions in the following, we will refer to the Pl\"ucker coordinates appearing in the above as
\begin{equation}
\p_{\drt} = \p_i^{(j)} \text{ when $\drt=\drt(v_i^{(j)})\in\wPpdroots$.}
\label{eq:root_notation_for_Plucker_coords}
\end{equation}
Note that the set $\{\p_{\drt}~|~\drt\in\wPpdroots\}$ consists of the Pl\"ucker coordinates appearing in Lemma \ref{lem:CayleyPlane_Plucker_coordinates_coincide_with_GLS} (it does not include $\p_0$). Moreover, we will write $\p_{\drt}^\T:\dunipP\to\C$ for the map $u_+\mapsto\p_{\drt}(u_+^\T)$.% and refer to these as \emph{transposed Pl\"ucker coordinates} when necessary.
%To express the generalized minors in terms of Pl\"ucker coordinates, we will have to consider the open, dense torus $\opendunim\subset\dunimP$. (See Definition \ref{df:opendunim} and Lemma \ref{lem:open_dunim_inside_dunimP}.)

\begin{lem}\label{lem:CayleyPlane_minors_coincide_with_denominators_of_potential}
The generalized minors defined in equation \eqref{eq:GLS_def_of_minor} can be expressed on $\dunipP$ in terms of Pl\"ucker coordinates as
\begin{equation}
\begin{array}{lll}
\minor_{\dfwt[1],\wPinv(\dfwt[1])}(u_+) = \p_{16}(u_+^\T), & 
\minor_{\dfwt[2],\wPinv(\dfwt[2])}(u_+) = q_{12}(u_+^\T), \\ 
\minor_{\dfwt[3],\wPinv(\dfwt[3])}(u_+) = q_{20}(u_+^\T), &
\minor_{\dfwt[4],\wPinv(\dfwt[4])}(u_+) = q_{24}(u_+^\T), \\
\minor_{\dfwt[5],\wPinv(\dfwt[5])}(u_+) = q_{16}(u_+^\T), &
\minor_{\dfwt[6],\wPinv(\dfwt[6])}(u_+) = \p_{8}(u_+^\T),
\end{array}
\label{eq:CayleyPlane_minors_and_plucker_expressions}
\end{equation}
where the expressions $q_i$ in Pl\"ucker coordinates were defined in \eqref{eq:CayleyPlane_denominators_of_potential}.
\end{lem}
\begin{proof}
We will begin by considering the generalized minors. The action of $\dwPinv$ on the highest weight vector $\hwt[j]$ of $\dfwtrep[j]$ is straightforward to calculate using the action of $\wPinv$ on the weight $\dfwt[j]$ and the following equality:
\begin{equation}
\bs_i\cdot\wtvmu = \tfrac1{m!}(\dChf_i)^m\cdot\wtvmu \quad\text{when $s_i(\wt)=\wt-m\sdr_i$,}
\label{eq:CayleyPlane_action_of_si}
\end{equation}
where $\wtvmu\in\dfwtrep[j]$ is assumed to be of weight $\wt$. As the expression can get quite large, we will use the following abbreviation:
\begin{equation}
(\dChf_{i_1})^{m_1}(\dChf_{i_2})^{m_2}\cdots = \dChf_{i_1^{m_1}i_2^{m_2}\cdots}
\label{eq:Abbreviation_of_products_of_Chevalley_generators}
\end{equation}
These calculations yield (cf.~equation \eqref{eq:CayleyRedExp_wP}):
%\ali{
%&\dwPinv\cdot\hwt[1] = \dChf_{6542345613452431}\cdot\hwt[1], \qquad
%\dwPinv\cdot\hwt[4] = \tfrac1{2^8}\tfrac1{3^2}\dChf_{6^3 5^3 4^2 2^2 3^2 4^2 561^2 3^2 4524}\cdot\hwt[4], \\
%&\dwPinv\cdot\hwt[2] = \tfrac12\dChf_{6^2 5423451342}\cdot\hwt[2], \qquad
%\dwPinv\cdot\hwt[5] = \tfrac1{2^4}\dChf_{6^2 5^2 4^2 2^2 34561345}\cdot\hwt[5], \\
%&\dwPinv\cdot\hwt[3] = \tfrac1{2^5}\dChf_{6^2 5^2 4^2 23^2 4561^2 345243}\cdot\hwt[3], \qquad
%\dwPinv\cdot\hwt[6] = \dChf_{65423456}\cdot\hwt[6].	
%}
\[\hspace{-.75em}
\begin{array}{ll}
\dwPinv\cdot\hwt[1] = \dChf_{6542345613452431}\cdot\hwt[1], &
\dwPinv\cdot\hwt[2] = \tfrac12\dChf_{6^2 5423451342}\cdot\hwt[2], \\
\dwPinv\cdot\hwt[3] = \tfrac1{2^5}\dChf_{6^2 5^2 4^2 23^2 4561^2 345243}\cdot\hwt[3], &
\dwPinv\cdot\hwt[4] = \tfrac1{2^8}\tfrac1{3^2}\dChf_{6^3 5^3 4^2 2^2 3^2 4^2 561^2 3^2 4524}\cdot\hwt[4], \\
\dwPinv\cdot\hwt[5] = \tfrac1{2^4}\dChf_{6^2 5^2 4^2 2^2 34561345}\cdot\hwt[5], &
\dwPinv\cdot\hwt[6] = \dChf_{65423456}\cdot\hwt[6].	
\end{array}
\]
%\ali{
%\dwPinv\cdot\hwt[1] &= \dChf_6\dChf_5\dChf_4\dChf_2\dChf_3\dChf_4\dChf_5\dChf_6\dChf_1\dChf_3\dChf_4\dChf_5\dChf_2\dChf_4\dChf_3\dChf_1\cdot\hwt[1], \\
%\dwPinv\cdot\hwt[2] &= \tfrac12(\dChf_6)^2\dChf_5\dChf_4\dChf_2\dChf_3\dChf_4\dChf_5\dChf_1\dChf_3\dChf_4\dChf_2\cdot\hwt[2], \\
%\dwPinv\cdot\hwt[3] &= \tfrac1{2^5}(\dChf_6)^2(\dChf_5)^2(\dChf_4)^2\dChf_2(\dChf_3)^2\dChf_4\dChf_5\dChf_6(\dChf_1)^2\dChf_3\dChf_4\dChf_5\dChf_2\dChf_4\dChf_3\cdot\hwt[3], \\
%\dwPinv\cdot\hwt[4] &= \tfrac1{2^6}\tfrac1{3^2}(\dChf_6)^3(\dChf_5)^3(\dChf_4)^2(\dChf_2)^2(\dChf_3)^2(\dChf_4)^2\dChf_5\dChf_6(\dChf_1)^2(\dChf_3)^2\dChf_4\dChf_5\dChf_2\dChf_4\cdot\hwt[4], \\
%\dwPinv\cdot\hwt[5] &= \tfrac1{2^4}(\dChf_6)^2(\dChf_5)^2(\dChf_4)^2(\dChf_2)^2\dChf_3\dChf_4\dChf_5\dChf_6\dChf_1\dChf_3\dChf_4\dChf_5\cdot\hwt[5], \\
%\dwPinv\cdot\hwt[6] &= \dChf_6\dChf_5\dChf_4\dChf_2\dChf_3\dChf_4\dChf_5\dChf_6\cdot\hwt[6].
%}
We now need to find the coefficient of $\hwt[j]$ in the product $u_+\dwPinv\cdot\hwt[j]$ to calculate the minors.

Instead of calculating these coefficients for general $u_+\in\dunipP$, we will restrict to the open, dense algebraic torus $\opendunip\subset\dunipP$ that is the image of $\opendunim$ under the transposition isomorphism $\cdot^\T:\dunimP\to\dunipP$. Recall from \eqref{eq:df_opendunim} that the elements of $\opendunim$ have a specific decomposition; after transposing these elements, we find that every $u_+\in\opendunip$ has a decomposition of the form
%\[
%\opendunim = \{\dy_{r_{16}}(a_1)\dy_{r_{15}}(a_2)\cdots\dy_{r_1}(a_{16})~|~a_i\in\C^*\}.
%\]
%\marginbox{Indices have been inverted!!!!}
%  1,  2,  3, 4,  5,  6, 7,8,9,10,11,12,13,14,15,16
%16,15,14,13,12,11,10,9,8,  7,  6,  5, 4,  3,  2, 1
\begin{equation}
%\opendunip = \{
u_+ = \dx_{r_{1}}(a_{1})\dx_{r_{2}}(a_{2})\cdots\dx_{r_{16}}(a_{16}), \quad\text{where $a_i\in\C^*$.}
%u_+ = \dx_{r_{1}}(a_{16})\dx_{r_{2}}(a_{15})\cdots\dx_{r_{16}}(a_{1}), \quad\text{where $a_i\in\C^*$.}
%~|~a_i\in\C^*\}.
\label{eq:elements_of_dunipP}
\end{equation}

Note that the sequence $(r_1,r_2,\ldots,r_{16})=(1,3,4,2,5,4,3,1,6,5,4,3,2,4,5,6)$ is the same as that of the fixed reduced expression for $\wP$ (see \eqref{eq:CayleyRedExp_wP}), and thus the inverse sequence compared to that of $\wPinv$. 
%\[
%1342543165432456
%\]

Since $\dx_i(a)$ acts on $v\in\dfwtrep[j]$ as
\begin{equation}
\dx_i(a)\cdot v = \bigl(1+a\,\dChe_i+\tfrac12a^2(\dChe_i)^2+\tfrac1{3!}a^3(\dChe_i)^3+\ldots\bigr)\cdot v,
\label{eq:action_of_one-parameter_subgroups_on_representation}
\end{equation}
we need to find all terms of $u_+$ that cancel each the factors $\dChf_i$ of $\dwPinv\cdot\hwt[j]$. However, Theorem 1.10 of \cite{Fomin_Zelevinsky_Double_Bruhat_Cells_and_Total_Positivity} tells us that the toric coordinates are related to the generalized minors by an invertible monomial transformation; in other words, there is exactly one term in the action of $u_+$ that cancels out the $\dChf_i$. In the following, we will argue that this term is the one whose $\dChe_i$ have the opposite sequence of indices (counting doubles and triples).

First, consider $u_+\dwPinv\cdot\hwt[1]=u_+\dChf_{6542345613452431}\cdot\hwt[1]$. As $\dwPinv\cdot\hwt[1]$ acts with 16 factors $\dChf_i$, lowering the weight by subtracting 16 simple roots, we need to act with 16 factors $\dChe_i$ to get back to a vector of the correct weight. However, since the representation in question is minuscule, we know that $(\dChe_i)^2$ acts trivially on the representations for all $i\in\{1,\ldots,6\}$. (See Theorem \ref{thm:Green_structure_minuscule_reps}.) Thus, the decomposition of $u_+$ in equation \eqref{eq:elements_of_dunipP} only has one term with 16 factors acting non-trivially: 
%For example, for $\dwPinv\cdot\hwt[1]$ this term would be (using an analogous abbreviation for products of $\dChe_i$ to \eqref{???}\marginbox{ref!})
\[
%a_1a_2a_3a_4a_5a_6a_7a_8a_9a_{10}a_{11}a_{12}a_{13}a_{14}a_{15}a_{16}\,\dChe_{1342543165432456}.
\left(\tprod[_{j=1}^{16}] a_j\right)\,\dChe_{1342543165432456}.
\]

For $u_+\dwPinv\cdot\hwt[6] = u_+\dChf_{65423456}\cdot\hwt[6]$, the computation is simplified by the Hasse diagram in \eqref{eq:Cayley-Hasse-diagram}: we find $\dwPinv\cdot\hwt[6]=v_8$. (Note that in interpreting diagram \eqref{eq:Cayley-Hasse-diagram} as the weight spaces of the minuscule representation $\dfwtrep[6]$, we have put the highest weight vector $\hwt[6]$ on the right and denoted it $v_0$, which is the opposite of the way we interpreted it for $\dfwtrep[1]$.) The diagram shows that both $\dChe_{65423456}$ and $\dChe_{65432456}$ map $v_8$ back to $v_0=\hwt[6]$. However, the decomposition of $u_+$ show that only the second occurs, so the only term in $u_+$ acting non-trivially is
\[
\left(\tprod[_{j=9}^{16}]a_j\right)\dChe_{65432456}.
%a_9a_{10}a_{11}a_{12}a_{13}a_{14}a_{15}a_{16}\,\dChe_{65432456}.
%a_1a_2a_3a_4a_5a_6a_7a_8\,\dChe_{65432456}.
\]

The arguments for $u_+\dwPinv\cdot\hwt[2] = u_+\,\tfrac12\dChf_{6^2 5423451342}\cdot\hwt[2]$ are more involved. First, note that $s_i(\dfwt[2])=\dfwt[2]-\de_{i2}\sdr_2$ implies that of $\dfwt[2]-\sdr_i$ only $i=2$ is a weight occurring in the representation $\dfwtrep[2]$. Moreover, it tells us that $\dfwt[2]-2\sdr_2$ is \emph{not} a weight of the representation. Thus, the last factor in any term of $u_+$ acting non-trivially must be $\dChe_2$. There are only two indices $j$ such that $r_j=2$ in the decomposition in \eqref{eq:elements_of_dunipP}, namely $j=4$ and $j=13$. However, only the $j=4$ factor $\dx_2(a_i)$ can contribute a factor $\dChe_2$ non-trivially: we will need the second factor to cancel out the second factor $\dChf_2$ in $\dwPinv\cdot\hwt[2]$. Thus, we reduced $u_+\dwPinv\cdot\hwt[2]$ to
\[
a_{4}\,\dChe_2\,\dx_{r_5}(a_{5})\dx_{r_6}(a_{6})\cdots\dx_{r_{16}}(a_{16})\,\dwPinv\cdot\hwt[2].
%a_{13}\,\dChe_2\,\dx_{r_5}(a_{12})\dx_{r_6}(a_{11})\cdots\dx_{r_{16}}(a_1)\cdot\dwPinv\cdot\hwt[2].
\]
For the second step, we have $s_i(\dfwt[2]-\sdr_2)=\dfwt[2]-\sdr_2-\de_{i4}\sdr_4$ for $i\neq2$ ($s_2$ will raise the weight back to $\dfwt[2]$), so that the next factor has to be $\dChe_4$. Now, $r_j=4$ for $j=3,6,11,14$, but we need $j>5$ by the above. Moreover, if we take $j=11$ or $j=14$ we will not have sufficient factors to reach $\hwt[2]$ from $\dwPinv\cdot\hwt[2]$ (we need ten more factors). Thus, we reduced $u_+\dwPinv\cdot\hwt[2]$ to
\[
a_{4}a_{6}\,\dChe_{24}\,\dx_{r_7}(a_{7})\dx_{r_8}(a_8)\cdots\dx_{r_{16}}(a_{16})\,\dwPinv\cdot\hwt[2].
%a_{11}a_{13}\,\dChe_{24}\,\dx_{r_7}(a_{10})\dx_{r_8}(a_9)\cdots\dx_{r_{16}}(a_{16})\cdot\dwPinv\cdot\hwt[2].
\]
In the third step, we have $s_i(\dfwt[2]-\sdr_2-\sdr_4)=\dfwt[2]-\sdr_2-\de_{i3}\sdr_3-\sdr_4-\de_{i5}\sdr_5$ for $i\neq4$, so there are two ``paths'': using an $\dChe_3$ or an $\dChe_5$. For the latter path, we are looking for a $j$ with $r_j=5$ and $j>6$. This means that $j=10$ or $j=15$, but this would give a problem: as $\dwPinv\cdot\hwt[2]=\tfrac12\dChf_{6^2 5423451342}\cdot\hwt[2]$, we need at least one factor $\dChe_1$, but $r_j\neq1$ for $j>10$. Thus, we need to take the path with $\dChe_3$. We find that $j=7$ gives the only possible factor $\dChe_3$ using an argument similar to the second step. Continuing to apply similar arguments, we conclude that only the term
\[
\tfrac12 a_4a_6a_7a_8a_{10}a_{11}a_{12}a_{13}a_{14}a_{15}(a_{16})^2\,\dChe_{24315432456^2}
%\tfrac12 (a_1)^2a_2a_3a_4a_5a_6a_7a_9a_{10}a_{11}a_{13}\,\dChe_{24315432456^2}
\]
of $u_+$ gives a non-trivial contribution. One can easily see that
\[
\tfrac12 \,\dChe_{24315432456^2}\cdot\tfrac12\dChf_{6^2 5423451342}\cdot\hwt[2] = \hwt[2], 
\]
so that
\[
\minor_{\dfwt[2],\wPinv(\dfwt[2])}(u_+) = a_4a_6a_7a_8a_{10}a_{11}a_{12}a_{13}a_{14}a_{15}a_{16}^2.
\]

Analogous arguments to those used in the previous three calculations imply that $u_+\dwPinv\cdot\hwt[i]$ for $i=3,4,5$ only have the claimed contribution as well. Altogether, we find for $u_+\in\opendunip$
\ali{
\minor_{\dfwt[1],\wPinv(\dfwt[1])}(u_+) &= a_1a_2a_3a_4a_5a_6a_7a_8a_9a_{10}a_{11}a_{12}a_{13}a_{14}a_{15}a_{16}, \\
\minor_{\dfwt[2],\wPinv(\dfwt[2])}(u_+) &= a_4a_6a_7a_8a_{10}a_{11}a_{12}a_{13}a_{14}a_{15}a_{16}^2, \\
\minor_{\dfwt[3],\wPinv(\dfwt[3])}(u_+) &= a_2a_3a_4a_5a_6a_7a_8^2a_9a_{10}a_{11}a_{12}^2a_{13}a_{14}^2a_{15}^2a_{16}^2, \\
\minor_{\dfwt[4],\wPinv(\dfwt[4])}(u_+) &= a_3a_4a_5a_6a_7^2a_8^2a_9a_{10}a_{11}^2a_{12}^2a_{13}^2a_{14}^2a_{15}^3a_{16}^3, \\
\minor_{\dfwt[5],\wPinv(\dfwt[5])}(u_+) &= a_5a_6a_7a_8a_9a_{10}a_{11}a_{12}a_{13}^2a_{14}^2a_{15}^2a_{16}^2, \\
\minor_{\dfwt[6],\wPinv(\dfwt[6])}(u_+) &= a_9a_{10}a_{11}a_{12}a_{13}a_{14}a_{15}a_{16}.
%\minor_{\dfwt[1],\wPinv(\dfwt[1])}(u_+) &= a_1a_2a_3a_4a_5a_6a_7a_8a_9a_{10}a_{11}a_{12}a_{13}a_{14}a_{15}a_{16}, \\
%\minor_{\dfwt[2],\wPinv(\dfwt[2])}(u_+) &= \tfrac1{2^2}a_1^2a_2a_3a_4a_5a_6a_7a_9a_{10}a_{11}a_{13}, \\
%\minor_{\dfwt[3],\wPinv(\dfwt[3])}(u_+) &= \tfrac1{2^{10}}a_1^2a_2^2a_3^2a_4a_5^2a_6a_7a_8a_9^2a_{10}a_{11}a_{12}a_{13}a_{14}a_{15}, \\
%\minor_{\dfwt[4],\wPinv(\dfwt[4])}(u_+) &= \tfrac1{2^{16}}\tfrac1{3^4}a_1^3a_2^3a_3^2a_4^2a_5^2a_6^2a_7a_8a_9^2a_{10}^2a_{11}a_{12}a_{13}a_{14}, \\
%\minor_{\dfwt[5],\wPinv(\dfwt[5])}(u_+) &= \tfrac1{2^8}a_1^2a_2^2a_3^2a_4^2a_5a_6a_7a_8a_9a_{10}a_{11}a_{12}, \\
%\minor_{\dfwt[6],\wPinv(\dfwt[6])}(u_+) &= a_1a_2a_3a_4a_5a_6a_7a_8.
}

To prove the equalities in \eqref{eq:CayleyPlane_minors_and_plucker_expressions}, the next step is to evaluate the Pl\"ucker coordinate expressions (defined in equation \eqref{eq:CayleyPlane_denominators_of_potential}) on the transpose of $u_+\in\opendunip$. We will find that they have the same monomial expressions in the toric coordinates $\{a_i\}$ of $\opendunip$. Recall from Section \ref{sec:Plucker_coords_and_ring_of_dunimP} that the generalized Pl\"ucker coordinates are defined on the fundamental weight representation $\dfwtrep[1]$, which is minuscule. As we saw in the above, this simplifies calculations considerably.% (i.e. only $n=0$ and $n=1$ can occur in equation \eqref{eq:CayleyPlane_action_of_si}).

Recall from equation \eqref{eq:Generalized-Plucker-coordinates} that $\p_i(g) = v_{0}^*(g \cdot v_i)$, where the $v_i$ have been defined using diagram \eqref{eq:Cayley-Hasse-diagram}. Thus, we need to find the terms of $u_+^\T\cdot v_i$ that are multiples of $v_0$. As $\dfwtrep[1]$ is minuscule, we know that each $\dy_i(a_j)$ in the decomposition of $u_+^\T\in\opendunim$ (see \eqref{eq:df_opendunim}) acts as $1+a_j\dChf_i$ on the representation. Now, every path from $v_i$ to $v_0$ in diagram \eqref{eq:Cayley-Hasse-diagram} gives a sequence $(b_1,\ldots,b_i)$ of indices such that $\bs_{b_1}\cdots\bs_{b_i}\cdot v_i=v_0$. By Corollary \ref{cor:Action_of_s_e_f}, we find that each $\bs_{b_j}$ acts as $\dChf_{b_j}$, so that $\dChf_{b_1}\cdots\dChf_{b_i}\cdot v_i=v_0$. Thus, for each sequence of indices $(b_1,\ldots,b_i)$ obtained from a path from $v_i$ to $v_0$, the term $\dChf_{b_1}\cdots\dChf_{b_i}$ in $u_+^\T$ will contribute its coefficient to $p_i^{(j)}(u_+^\T)$. 

To find these coefficients, apply the following algorithm: determine all the paths $(b_1,\ldots,b_i)$ from $v_i$ to $v_0$ in diagram \eqref{eq:Cayley-Hasse-diagram}; find all the subexpressions of the form $s_{b_1}\cdots s_{b_i}$ in $\wPinv$; for each of the subexpressions, take the respective coefficients in the decomposition of $u_+^\T\in\opendunim$. Note that this algorithm is reminiscent of the methods in \cite{Spacek_LP_LG_models}; we can use the result in Proposition 8.11 (applied to $\wPinv$; in other words flipping all the arrows) there to obtain these subexpressions combinatorially.

Let us illustrate with $p_7'(u_+^\T)$ for $u_+^\T\in\opendunim$. Using diagram \eqref{eq:Cayley-Hasse-diagram} we find that there are two paths: $(6,5,4,2,3,4,5)$ and $(6,5,4,3,2,4,5)$. The first gives the following subexpressions:%\marginbox{Check Indexing conventions!}
\ali{
&\textcolor{red}{\underline{\boldsymbol{s_6s_5s_4s_2s_3s_4s_5}}}s_6s_1s_3s_4s_5s_2s_4s_3s_1, \quad 
\textcolor{red}{\underline{\boldsymbol{s_6s_5s_4s_2s_3s_4}}}s_5s_6s_1s_3s_4\textcolor{red}{\underline{\boldsymbol{s_5}}}s_2s_4s_3s_1, \\
&\textcolor{red}{\underline{\boldsymbol{s_6s_5s_4s_2s_3}}}s_4s_5s_6s_1s_3\textcolor{red}{\underline{\boldsymbol{s_4s_5}}}s_2s_4s_3s_1, \quad
\textcolor{red}{\underline{\boldsymbol{s_6s_5s_4s_2}}}s_3s_4s_5s_6s_1\textcolor{red}{\underline{\boldsymbol{s_3s_4s_5}}}s_2s_4s_3s_1,
}
The second does not give any subexpressions in $\wPinv$. We now compare it with the decomposition of $u_+^\T$ as
\ali{
u_+^\T&=\dy_{6}(a_{16})\dy_{5}(a_{15})\dy_{4}(a_{14})\dy_{2}(a_{13})\dy_{3}(a_{12})\dy_{4}(a_{11})\dy_{5}(a_{10})\dy_{6}(a_{9}) \\
&\qquad\cdot\dy_{1}(a_{8})\dy_{3}(a_{7})\dy_{4}(a_{6})\dy_{5}(a_{5})\dy_{2}(a_{4})\dy_{4}(a_{3})\dy_{3}(a_{2})\dy_{1}(a_{1}).
}
to find the contributing terms (in order of the subexpressions):
\ali{
p_7'(u_+^\T) &=  a_{10}a_{11}a_{12}a_{13}a_{14}a_{15}a_{16} + a_5a_{11}a_{12}a_{13}a_{14}a_{15}a_{16} \\
&\qquad + a_5a_6a_{12}a_{13}a_{14}a_{15}a_{16} + a_5a_6a_7a_{13}a_{14}a_{15}a_{16} \\
&= \bigl((a_5+a_{10})a_{11}a_{12}+a_5a_6(a_7+a_{12})\bigr)a_{13}a_{14}a_{15}a_{16}.
}
%\marginbox{HERE!}
%\[
%p_1(u_-) = a_1+a_8 \qand p_9'(u_-) = a_1a_2a_3a_4a_5a_6a_7a_8(a_9+a_{16})
%\]

Once we obtained all toric expressions for the Pl\"ucker coordinates, calculating the expressions of equation \eqref{eq:CayleyPlane_denominators_of_potential} is straightforward. We find the following for $u_+\in\opendunip$
\ali{
p_8(u_+^\T) &= a_9a_{10}a_{11}a_{12}a_{13}a_{14}a_{15}a_{16}, \\
p_{16}(u_+^\T) &= a_1a_2a_3a_4a_5a_6a_7a_8a_9a_{10}a_{11}a_{12}a_{13}a_{14}a_{15}a_{16}, \\
q_{12}(u_+^\T) &= a_{4}a_{6}a_{7}a_8a_{10}a_{11}a_{12}a_{13}a_{14}a_{15}a_{16}^2, \\
q_{16}(u_+^\T) &= a_5a_6a_7a_8a_9a_{10}a_{11}a_{12}a_{13}^2a_{14}^2a_{15}^2a_{16}^2, \\
q_{20}(u_+^\T) &= a_2a_3a_4a_5a_6a_7a_8^2a_9a_{10}a_{11}a_{12}^2a_{13}a_{14}^2a_{15}^2a_{16}^2, \\
q_{24}(u_+^\T) &= a_3a_4a_5a_6a_7^2a_8^2a_9a_{10}a_{11}^2a_{12}^2a_{13}^2a_{14}^2a_{15}^3a_{16}^3.
}
%\ali{
%p_8(u_-) &= a_1a_2a_3a_4a_5a_6a_7a_8 \\
%p_{16}(u_-) &= a_1a_2a_3a_4a_5a_6a_7a_8a_9a_{10}a_{11}a_{12}a_{13}a_{14}a_{15}a_{16} \\
%q_{12}(u_-) &= a_1^2a_2a_3a_4a_5a_6a_7a_9a_{10}a_{11}a_{13} \\
%q_{16}(u_-) &= a_1^2a_2^2a_3^2a_4^2a_5a_6a_7a_8a_9a_{10}a_{11}a_{12} \\
%q_{20}(u_-) &= a_1^2a_2^2a_3^2a_4a_5^2a_6a_7a_8a_9^2a_{10}a_{11}a_{12}a_{13}a_{14}a_{15} \\
%q_{24}(u_-) &= a_1^3a_2^3a_3^2a_4^2a_5^2a_6^2a_7a_8a_9^2a_{10}^2a_{11}a_{12}a_{13}a_{14}
%}
Comparing these with the expressions for the minors, we conclude that the equalities in \eqref{eq:CayleyPlane_minors_and_plucker_expressions} hold, at least for $u_+\in\opendunip$. However, $\opendunip\subset\dunipP$ is open and dense (this follows from $\opendunim\subset\dunimP$ as discussed in Section \ref{sec:Lie_theoretic_Mirror}) and $\dunipP$ is irreducible, so these equalities extend to all of $\dunimP$.
\end{proof}

We will need to calculate local expressions for generalized minors on the torus $\opendunip$ again for the superpotential of the Freudenthal variety, as well for the cluster structure of the mirrors of both varieties, so we will summarize the algorithm performed in the proof above of Lemma \ref{lem:CayleyPlane_minors_coincide_with_denominators_of_potential}, for ease of reference. We will need only generalized minors of the form
\[
\minor_{\dfwt[i],w(\dfwt[i])}(u_+)=\Bigl \lan u_+\bw\cdot \hwt,\hwt\Bigr \ran,
\]
where $w\in\weyl$, though the general form is entirely analogous.

\begin{algo}\phantomsection\label{alg:torus-expansion}
  \begin{enumerate}
  \item In the weight space for the irreducible representation $V(\dfwt[i])$ with highest weight $\dfwt[i]$, find the weight $w(\dfwt[i])$ by acting on $\dfwt[i]$ with the Weyl group element $w=s_{i_1}\dots s_{i_k}$. As shown in the proof above, this corresponds to acting on $\hwt$ with the element $\frac{1}{m_1!}\dots\frac{1}{m_k!}\dChf_{i_1^{m_1}\dots i_k^{m_k}}$. 
  \item In order to find the coefficient of $\hwt$ in $u_+\bw\cdot \hwt$, we must find those terms in $u_+$ which act on $\bw\cdot\hwt$ by sending it to a multiple of $\hwt$. To do this, first find all sequences $(i_1,\dots, i_k)$ such that $\dChe_{i_1\dots i_k}(\bw\cdot\hwt)$ is a multiple of $\hwt$. (We refer to these sequences as \emph{paths} in the weight space of $V(\dfwt[i])$ from $w(\dfwt[i])$ to $\dfwt[i]$.)
  \item For each such $\dChe_{i_1,\dots, i_k}$, determine its coefficient in $u_+$. One way to do this is to expand $u_+$ as a product of the $\dx_{r_i}(a_i)=(1+a_i\dChe_{r_i}+\dots)$ and to collect terms by the index sequence of $\dChe_{i_1\dots i_k}$.
  \item The coefficient of each such term of $u_+$ above gives a polynomial in the torus coordinates, and the sum of all these is the local expression of $\minor_{\dfwt[i],w(\dfwt[i])}(u_+)$ for $u_+\in\opendunip$.
  \end{enumerate}
  While calculating the superpotential of the Freudenthal variety and the cluster structure for its mirror, we found that the algorithm is performed much more efficiently by a computer algebra program when steps (2) and (3) are performed simultaneously: to be more precise, while searching for paths from $w(\dfwt[i])$ to $\dfwt[i]$ as in step (2), we discard a partial sequence that does not appear in the expansion of $u_+$.
\end{algo}
In Appendix \ref{sec:expansion-appendix}, we give an example of performing this algorithm for the Freudenthal variety.

Note that the calculation of Pl\"ucker coordinates in the proof of Lemma \ref{lem:CayleyPlane_minors_coincide_with_denominators_of_potential} above is actually a simplification (and translation in terms of combinatorial objects) of the same algorithm. This simplification is mostly due to the fact that Pl\"ucker coordinates are defined on \emph{minuscule} highest weight representations. Since the algorithm works for the Pl\"ucker coordinates of any minuscule homogeneous space, we will summarize it as well:
\begin{algo}\phantomsection\label{alg:Plucker_torus_expansion}\begin{enumerate}
	\item For a given Pl\"ucker coordinate $\p_i$ on a \emph{minuscule} homogeneous space $\dP\backslash\dG$, find the corresponding vertex $v_i$ in the weight diagram of the corresponding minuscule representation.
	\item Determine all the paths $(b_1,\ldots,b_i)$ in the weight diagram from $v_i$ to $v_0$, the lowest weight vector.
	\item Find all the subexpressions of the form $s_{b_1}\cdots s_{b_i}$ in $\wPinv$, using Proposition 8.11 of \cite{Spacek_LP_LG_models} (replacing $\wPrime$ by $s_{b_1}\cdots s_{b_i}$ and $\wP$ by $\wPinv$).
	\item For each of the subexpressions, take the respective coefficients in the decomposition of $u_+^\T\in\opendunim$, noting that at most linear contributions occur.
\end{enumerate}
\end{algo}

Recalling that we have written $\p_i^\T:u_+\mapsto\p_i(u_+^\T)$ for $i=8,16$ and $q_i^\T:u_+\mapsto q_i(u_+^\T)$ for $i=12,16,20,24$, we conclude the following by combining Lemmas \ref{lem:CayleyPlane_Plucker_coordinates_coincide_with_GLS} and \ref{lem:CayleyPlane_minors_coincide_with_denominators_of_potential} with Proposition \ref{prop:GLS_coord_ring_dunipP} giving the expression for the coordinate ring of $\dunipP$:
\begin{cor}\label{cor:CayleyPlane_coordinate_ring_of_dunimP}
The coordinate ring of the affine variety $\dunipP$ is given by
\[
\C[\dunipP] = \C[\p_{\drt}^\T\,|\,\drt\in\wPpdroots][(\p_8^\T)^{-1},(\p_{16}^\T)^{-1},(q_{12}^\T)^{-1},(q_{16}^\T)^{-1},(q_{20}^\T)^{-1},(q_{24}^\T)^{-1}].
\]
This implies that the coordinate ring of the affine variety $\dunimP=(\dunipP)^\T$ is given by
\[
\C[\dunimP]=\C[\p_{\drt}\,|\,\drt\in\wPpdroots][\p_8^{-1},\p_{16}^{-1},q_{12}^{-1},q_{16}^{-1},q_{20}^{-1},q_{24}^{-1}].
%\C[\p_1,\p_2,\p_3,\p_4',\p_4'',\p_5',\p_5'',\p_6',\p_6'',\p_7',\p_7'',\p_8',\p_8'',\p_9'',\p_{10}'',\p_{11}''][\p_8^{-1},\p_{16}^{-1},q_{12}^{-1},q_{16}^{-1},q_{20}^{-1},q_{24}^{-1}],
\]
This in turn implies that $\pi(\dunimP)\subseteq\mX_\can$.
\end{cor}
Recall that the generating Pl\"ucker coordinates are the ones appearing in Lemma \ref{lem:CayleyPlane_Plucker_coordinates_coincide_with_GLS}. and that the (generalized) Pl\"ucker coordinates are defined in equation \eqref{eq:Generalized-Plucker-coordinates}. The quadratic and cubic expressions $q_i$ are given in equation \eqref{eq:CayleyPlane_denominators_of_potential}.

\begin{rem}\label{rem:CayleyPlane_inverted_Plucker_and_generators}
Note that $\p_8$, $\p_{16}$, $q_{12}$, $q_{16}$, $q_{20}$ and $q_{24}$ are not a priori generated by the generators of the coordinate ring. 
%%%However, there are relations on $\dunimP$ between the remaining Pl\"ucker coordinates and the generating ones, for example, we have
%%%\[
%%%\p_8 = \p_1\p_7'-\p_2\p_6'+\p_3\p_5'-\p_4'\p_4''.
%%%\]
%%%Thus, the expression for the coordinate ring in Corollary \ref{cor:CayleyPlane_coordinate_ring_of_dunimP} is well-defined. We have given all these ``Pl\"ucker relations'' in section \ref{sec:Cayley_Plucker_relations}.\marginbox{Just put them here?}
Recall from Corollary \ref{cor:CayleyPlane_coordinate_ring_of_dunimP} that the coordinate ring is generated by the Pl\"ucker coordinates $\p_{\drt}$ for $\drt\in\wPpdroots$, i.e.
\[
\p_1,\p_2,\p_3,\p_4',\p_4'',\p_5',\p_5'',\p_6',\p_6'',\p_7',\p_7'',\p_8',\p_8'',\p_9'',\p_{10}'',\p_{11}''.
\]
(Note that $\p_0\equiv1$ on $\dunim$.) However, it turns out that there are ``Pl\"ucker relations'' expressing the remaining Pl\"ucker coordinates in terms of these. We have found these using the toric expressions for the Pl\"ucker coordinates restricted to $\opendunim$; as $\opendunim\subset\dunimP$ is dense and $\dunimP$ is irreducible, this suffices. The relations are given as follows:
\[
\begin{array}{lcl}
\p_8 = \p_1\p_7'-\p_2\p_6'+\p_3\p_5'-\p_4'\p_4'' ,&&
\p_9' = \p_1\p_8'-\p_2\p_7''+\p_3\p_6''-\p_4'\p_5'' ,\\
\p_{10}'= \p_1\p_9''-\p_2\p_8''+\p_4''\p_6''-\p_5'\p_5'' ,&&
\p_{11}'= \p_1\p_{10}''-\p_3\p_8''+\p_4''\p_7''-\p_5''\p_6' ,\\
\p_{12}'=\p_2\p_{10}''-\p_3\p_9''+\p_4''\p_8'-\p_5''\p_7',&&
\p_{12}''=\p_1\p_{11}''-\p_4'\p_8''+\p_5'\p_7''-\p_6'\p_6'',\\                                                               
% \p_{12}'= \p_2\p_{10}''-\p_3\p_8''+\p_4''\p_8'-\p_5''\p_7'',&&
%\p_{12}''= \p_1\p_{11}'' -\p_4''\p_8'+\p_5''\p_7'-\p_4'\p_8'',\\
\p_{13}= \p_2\p_{11}''-\p_4'\p_9''+\p_5'\p_8'-\p_6''\p_7',&&
\p_{14}= \p_3\p_{11}''-\p_4'\p_{10}''+\p_6'\p_8'-\p_7'\p_7'',\\
\p_{15}= \p_4''\p_{11}''-\p_5'\p_{10}''+\p_6'\p_9''-\p_7'\p_8'',&&
\p_{16}= \p_5''\p_{11}''-\p_6''\p_{10}''+\p_7''\p_9''-\p_8'\p_8''.
\end{array}
\]
Thus, the localized expressions are in fact expressible in terms of the generators.
\end{rem}

\subsection{The coordinate ring of \texorpdfstring{$\dunimP$}{U\_-\textasciicircum{}P} for the Freudenthal variety}\label{sec:coord_ring_dunimP_E7}
The approach for the Freudenthal variety $\ECHS[7]$ is analogous to the Cayley plane. %Therefore, we will only highlight the differences.
Recall the from Section \ref{sec:ECHS} that $\ellwP=\ell(\wP)=27$. Thus, Proposition \ref{prop:GLS_coord_ring_dunipP} (see Proposition 8.5 of \cite{GLS_Kac_Moody_groups_and_cluster_algebras}) tells us that the coordinate ring of $\dunipP$ is generated by the dual PBW elements $\pGLS_m$, $m\in\{1,\ldots,27\}$, and localized at the generalized minors $\minor_{\dfwt[j],\wPinv(\dfwt[j])}$, $j\in\{1,\ldots,7\}$. The generators are expressed as follows:
%\begin{equation}
%\C[\pGLS_{1},\ldots,\pGLS_{27}][\minor_{\dfwt[1],\wPinv(\dfwt[1])}^{-1},\ldots,\minor_{\dfwt[7],\wPinv(\dfwt[7])}^{-1}].
%\label{eq:Freudenthal_GLS_expression_for_coordinate_ring_of_dunimP}
%\end{equation}
%Using the criterion in Proposition \ref{prop:Criterion_Plucker_coordinates_coincide_with_GLS}, we find that the generators $\pGLS_m$ can be expressed in terms of Pl\"ucker coordinates as follows:

\begin{lem}\label{lem:Freudenthal_Plucker_coordinates_coincide_with_GLS}
On $\dunipP=(\dunimP)^\T\subset\udG$, we have the following identifications:
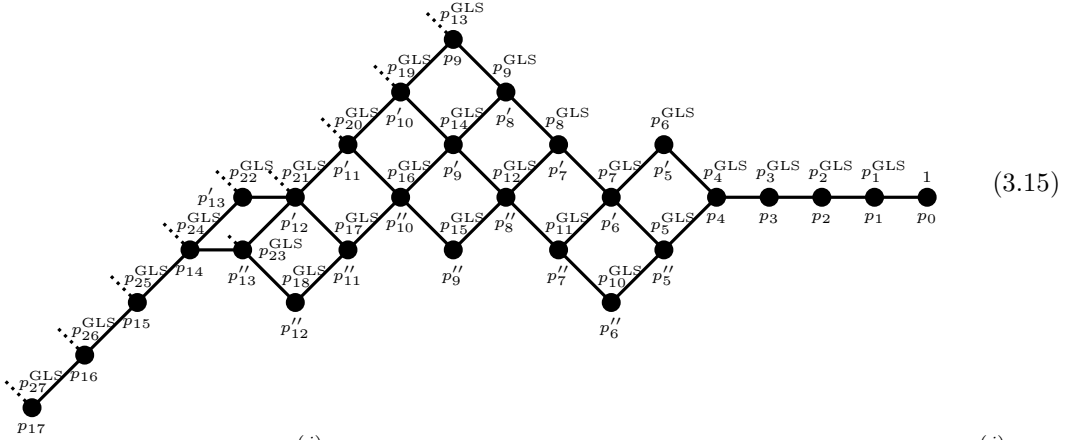
\begin{equation}
\begin{tikzpicture}[rotate=90,scale=.7,style=very thick,baseline=0.25em]
	\draw 
	(0,-9)--(0,-5)--(-2,-3)--(0,-1)--(-1,0)--(0,1)--(-2,3)--(-1,4)
	(0,-5)--(1,-4)--(0,-3)--(3,0)--(0,3)
	(-1,-4)--(0,-3)--(-1,-2) (-1,4)--(0,3)--(-1,2)
	(1,2)--(0,1)--(2,-1) (1,-2)--(0,-1)--(2,1)
	(0,4)--(0,3)
	(-4,8)--(-1,5)--(-1,4)
	(-1,5)--(0,4);
	\draw[dotted]
	(-1,4) -- (-.67,4.33)
	(3.5,0.5)--(3,0)
	(2.5,1.5)--(2,1)
	(1.5,2.5)--(1,2)
	(0.5,3.5)--(0,3)
	(0.5,4.5)--(0,4)
	(-0.5,5.5)--(-1,5)
	(-1.5,6.5)--(-2,6)
	(-2.5,7.5)--(-3,7)
	(-3.5,8.5)--(-4,8);
	\draw[black, fill=black] 
	(0,-9) circle (.15)
	(0,-8) circle (.15) 
	(0,-7) circle (.15)
	(0,-6) circle (.15)
	(0,-5) circle (.15)
	(1,-4) circle (.15)
	(-1,-4) circle (.15)
	(0,-3) circle (.15)
	(-2,-3) circle (.15)
	(1,-2) circle (.15)
	(-1,-2) circle (.15)
	(2,-1) circle (.15)
	(0,-1) circle (.15)
	(3,0) circle (.15)
	(1,0) circle (.15)
	(-1,0) circle (.15)
	(2,1) circle (.15)
	(0,1) circle (.15)
	(1,2) circle (.15)
	(-1,2) circle (.15)
	(0,3) circle (.15)
	(-2,3) circle (.15)
	(0,4) circle (.15)
	(-1,4) circle (.15)
	(-1,5) circle (.15)
	(-2,6) circle (.15)
	(-3,7) circle (.15)
	(-4,8) circle (.15);
	\node at (0,-9)[above=2pt]{\tiny$1$};
	\node at (0,-9)[below=2pt]{\tiny$\p_0$};
	\node at (0,-8)[above=2pt]{\tiny$\pGLS_1\hspace{-1em}$};
	\node at (0,-8)[below=2pt]{\tiny$\p_1$};
	\node at (0,-7)[above=2pt]{\tiny$\pGLS_2\hspace{-1em}$};
	\node at (0,-7)[below=2pt]{\tiny$\p_2$};
	\node at (0,-6)[above=2pt]{\tiny$\pGLS_3\hspace{-1em}$};
	\node at (0,-6)[below=2pt]{\tiny$\p_3$};
	\node at (0,-5)[above=2pt]{\tiny$\pGLS_4\hspace{-1em}$};
	\node at (0,-5)[below=2pt]{\tiny$\p_4$};
	\node at (1,-4)[above=2pt]{\tiny$\pGLS_6\hspace{-1em}$};
	\node at (1,-4)[below=2pt]{\tiny$\p_5'$};
	\node at (-1,-4)[above=2pt]{\tiny$\pGLS_5\hspace{-1em}$};
	\node at (-1,-4)[below=2pt]{\tiny$\p_5''$};
	\node at (0,-3)[above=2pt]{\tiny$\pGLS_7\hspace{-1em}$};
	\node at (0,-3)[below=2pt]{\tiny$\p_6'$};
	\node at (-2,-3)[above=2pt]{\tiny$\pGLS_{10}\hspace{-1em}$};
	\node at (-2,-3)[below=2pt]{\tiny$\p_6''$};
	\node at (1,-2)[above=2pt]{\tiny$\pGLS_8\hspace{-1em}$};
	\node at (1,-2)[below=2pt]{\tiny$\p_7'$};
	\node at (-1,-2)[above=2pt]{\tiny$\pGLS_{11}\hspace{-1em}$};
	\node at (-1,-2)[below=2pt]{\tiny$\p_7''$};
	\node at (2,-1)[above=2pt]{\tiny$\pGLS_9\hspace{-1em}$};
	\node at (2,-1)[below=2pt]{\tiny$\p_8'$};
	\node at (0,-1)[above=2pt]{\tiny$\pGLS_{12}\hspace{-1em}$};
	\node at (0,-1)[below=2pt]{\tiny$\p_8''$};
	\node at (3,0)[above=2pt]{\tiny$\pGLS_{13}\hspace{-1em}$};
	\node at (3,0)[below=2pt]{\tiny$\p_9$};
	\node at (1,0)[above=2pt]{\tiny$\pGLS_{14}\hspace{-1em}$};
	\node at (1,0)[below=2pt]{\tiny$\p_9'$};
	\node at (-1,0)[above=2pt]{\tiny$\pGLS_{15}\hspace{-1em}$};
	\node at (-1,0)[below=2pt]{\tiny$\p_9''$};
	\node at (2,1)[above=2pt]{\tiny$\pGLS_{19}\hspace{-1em}$};
	\node at (2,1)[below=2pt]{\tiny$\p_{10}'$};
	\node at (0,1)[above=2pt]{\tiny$\pGLS_{16}\hspace{-1em}$};
	\node at (0,1)[below=2pt]{\tiny$\p_{10}''$};
	\node at (1,2)[above=2pt]{\tiny$\pGLS_{20}\hspace{-1em}$};
	\node at (1,2)[below=2pt]{\tiny$\p_{11}'$};
	\node at (-1,2)[above=2pt]{\tiny$\pGLS_{17}\hspace{-1em}$};
	\node at (-1,2)[below=2pt]{\tiny$\p_{11}''$};
	\node at (0,3)[above=2pt]{\tiny$\pGLS_{21}\hspace{-1em}$};
	\node at (0,3)[below=2pt]{\tiny$\p_{12}'$};
	\node at (-2,3)[above=2pt]{\tiny$\pGLS_{18}\hspace{-1em}$};
	\node at (-2,3)[below=2pt]{\tiny$\p_{12}''$};
	\node at (0,4)[above=2pt]{\tiny$\pGLS_{22}\hspace{-1em}$};
	\node at (0,4)[left=2pt]{\tiny$\p_{13}'$};
	\node at (-1,4)[right=2pt]{\tiny$\pGLS_{23}\hspace{-1em}$};
	\node at (-1,4)[below=2pt]{\tiny$\p_{13}''$};
	\node at (-1,5)[above=2pt]{\tiny$\pGLS_{24}\hspace{-1em}$};
	\node at (-1,5)[below=2pt]{\tiny$\p_{14}$};
	\node at (-2,6)[above=2pt]{\tiny$\pGLS_{25}\hspace{-1em}$};
	\node at (-2,6)[below=2pt]{\tiny$\p_{15}$};
	\node at (-3,7)[above=2pt]{\tiny$\pGLS_{26}\hspace{-1em}$};
	\node at (-3,7)[below=2pt]{\tiny$\p_{16}$};
	\node at (-4,8)[above=2pt]{\tiny$\pGLS_{27}\hspace{-1em}$};
	\node at (-4,8)[below=2pt]{\tiny$\p_{17}$};
\end{tikzpicture}
\label{eq:Freudenthal-Plucker-and-GLS-diagram}
\vspace{-1em}\end{equation}
Here, a vertex marked with $\p_i^{(j)}$ below (or left) and $\pGLS_{m}$ above (or right) denotes that $\p_i^{(j)}(u_+^\T)$ equals $\pGLS_m(u_+)$ up to a constant for $u_+\in\dunipP$. %On $\pi(\dunimP)=\udP\backslash\dunimP\subset\cmX$,
\end{lem}
Note that the diagram of equation \eqref{eq:Freudenthal-Plucker-and-GLS-diagram} is a subdiagram of the one in equation \eqref{eq:Freudenthal-Hasse-diagram}. Upon removing the rightmost vertex, this subdiagram is isomorphic to the Hasse diagram of the Cayley plane $\cayley$ (see equation \eqref{eq:Cayley-Hasse-diagram}). % while the diagram of equation \eqref{eq:Freudenthal-Hasse-diagram} is isomorphic to the Hasse diagram of the Freudenthal variety.
\begin{proof}
This is proven by computing the weight of each of the basis vectors $v_i^{(j)}$ and the positive roots $\wPdrt(m)$ that are mapped to negative roots by $\wP$ ($m\in\{1,\ldots,27\}$), and using Proposition \ref{prop:Criterion_Plucker_coordinates_coincide_with_GLS}. Compare Lemma \ref{lem:CayleyPlane_Plucker_coordinates_coincide_with_GLS}.
\end{proof}

\begin{lem}\label{lem:Freudenthal_minors_coincide_with_denominators_of_potential}
The generalized minors defined in equation \eqref{eq:GLS_def_of_minor} can be expressed on $\dunipP$ in terms of Pl\"ucker coordinates as
\begin{equation}
\begin{array}{lll}
\minor_{\dfwt[1],\wPinv(\dfwt[1])}(u_+) = q_{18}(u_+^\T), & 
\minor_{\dfwt[2],\wPinv(\dfwt[2])}(u_+) = q_{27}(u_+^\T), \\ 
\minor_{\dfwt[3],\wPinv(\dfwt[3])}(u_+) = q_{36}'(u_+^\T), &
\minor_{\dfwt[4],\wPinv(\dfwt[4])}(u_+) = q_{54}(u_+^\T), \\
\minor_{\dfwt[5],\wPinv(\dfwt[5])}(u_+) = q_{45}(u_+^\T), &
\minor_{\dfwt[6],\wPinv(\dfwt[6])}(u_+) = q_{36}(u_+^\T), \\
\minor_{\dfwt[7],\wPinv(\dfwt[7])}(u_+) = \p_{27}(u_+^\T),
\end{array}
\label{eq:Freudenthal_minors_and_plucker_expressions}
\end{equation}
where the expressions $q_i$ in Pl\"ucker coordinates were defined in \eqref{eq:Freudenthal_denominators_of_potential}.
\end{lem}
\begin{proof}
The arguments are analogous to the proof of Lemma \ref{lem:CayleyPlane_minors_coincide_with_denominators_of_potential}. Restricting to $u_+\in\opendunip = \{\dx_{r_{1}}(a_1)\dx_{r_{2}}(a_2)\cdots\dx_{r_{27}}(a_{27})~|~a_i\in\C^*\}$, where the sequence of indices equals $(r_1,r_2,\ldots,r_{27})=(7,6,5,4,3,2,4,5,6,1,3,4,2,5,7,4,3,1,6,5,4,2,3,4,5,6,7)$, we find using Algorithm \ref{alg:torus-expansion}
%Restricting to $u_-\in\opendunim = \{\dy_{r_{27}}(a_1)\dy_{r_{26}}(a_2)\cdots\dy_{r_1}(a_{27})~|~a_i\in\C^*\}$, where $\wP=s_{r_1}s_{r_2}\cdots s_{r_{27}}$, we find that
\ali{
\minor_{\dfwt[1],\wPinv(\dfwt[1])}(u_+) &=
\frac{a_{27}}{a_{15}}\bigl(\tprod[_{i=10}^{27}] a_i\bigr), \\
\minor_{\dfwt[2],\wPinv(\dfwt[2])}(u_+) &= 
\frac{a_{18}a_{23}a_{24}a_{25}a_{26}a_{27}}{a_{10}}\bigl(\tprod[_{i=6}^{27}]a_i\bigr), \\
\minor_{\dfwt[3],\wPinv(\dfwt[3])}(u_+) &= 
\frac{a_{26}a_{27}}{a_6a_{14}a_{15}a_{19}}\bigl(\tprod[_{i=13}^{27}]a_i\bigr)\bigl(\tprod[_{j=5}^{27}]a_j\bigr), \\
%\frac{a_{13}a_{16}a_{17}a_{18}a_{20}a_{21}a_{22}a_{23}a_{24}a_{25}a_{26}^2a_{27}^2}{a_6}\tprod[_{i=5}^{27}]a_i, \\
\minor_{\dfwt[4],\wPinv(\dfwt[4])}(u_+) &= 
\frac{a_{25}a_{26}a_{27}}{a_{10}a_{11}a_{19}a_{20}}\bigl(\tprod[_{i_1=17}^{27}]a_{i_1}\bigr)\bigl(\tprod[_{i_2=8}^{27}]a_{i_2}\bigr)\bigl(\tprod[_{i_3=4}^{27}]a_{i_3}\bigr), \\
\minor_{\dfwt[5],\wPinv(\dfwt[5])}(u_+) &= 
a_9a_{22}a_{24}a_{25}a_{26}a_{27}\bigl(\tprod[_{i_1=14}^{27}]a_{i_1}\bigr)\bigl(\tprod[_{i_2=3}^{27}]a_{i_2}\bigr), \\
\minor_{\dfwt[6],\wPinv(\dfwt[6])}(u_+) &= 
a_{15}\bigl(\tprod[_{i_1=19}^{27}]a_{i_1}\bigr)\bigl(\tprod[_{i_2=2}^{27}]a_{i_2}\bigr), \\
\minor_{\dfwt[7],\wPinv(\dfwt[7])}(u_+) &= \tprod[_{i=1}^{27}]a_i.
}
%\ali{
%\hspace{-1em}
%\minor_{\dfwt[1],\wPinv(\dfwt[1])}(u_-) &= \tfrac1{2^2}a_1^2a_2a_3a_4a_5a_6a_7a_8a_9a_{10}a_{11}a_{12}a_{14}a_{15}a_{16}a_{17}a_{18}, \\
%\hspace{-1em}
%\minor_{\dfwt[2],\wPinv(\dfwt[2])}(u_-) &= \tfrac1{2^{12}}a_1^2a_2^2a_3^2a_4^2a_5^2a_6a_7a_8a_9a_{10}^2a_{11}a_{12}a_{13}a_{14}a_{15}a_{16}a_{17}a_{19}a_{20}a_{21}a_{22}, \\
%\hspace{-1em}
%\minor_{\dfwt[3],\wPinv(\dfwt[3])}(u_-) &= \tfrac1{2^{24}}\tfrac1{3^4}a_1^3a_2^3a_3^2a_4^2a_5^2a_6^2a_7^2a_8^2a_9a_{10}^2a_{11}^2a_{12}^2a_{13}a_{14}a_{15}^2a_{16}a_{17}a_{18}a_{19}a_{20}a_{21}a_{23}, \\
%\hspace{-1em}
%\minor_{\dfwt[4],\wPinv(\dfwt[4])}(u_-) &= \tfrac1{2^{36}}\tfrac1{3^{18}}\tfrac1{4^6}a_1^4a_2^4a_3^4a_4^3a_5^3a_6^3a_7^3a_8^2a_9^2a_{10}^3a_{11}^3a_{12}^2a_{13}^2a_{14}^2a_{15}^2a_{16}^2a_{17}a_{18}a_{19}^2a_{20}^2a_{21}a_{22}a_{23}a_{24}, \\
%\hspace{-1em}
%\minor_{\dfwt[5],\wPinv(\dfwt[5])}(u_-) &= \tfrac1{2^{30}}\tfrac1{3^{10}}a_1^3a_2^3a_3^3a_4^3a_5^2a_6^3a_7^2a_8^2a_9^2a_{10}^2a_{11}^2a_{12}^2a_{13}^2a_{14}^2a_{15}a_{16}a_{17}a_{18}a_{19}^2a_{20}a_{21}a_{22}a_{23}a_{24}a_{25}, \\
%\hspace{-1em}
%\minor_{\dfwt[6],\wPinv(\dfwt[6])}(u_-) &= \tfrac1{2^{20}}a_1^2a_2^2a_3^2a_4^2a_5^2a_6^2a_7^2a_8^2a_9^2a_{10}a_{11}a_{12}a_{13}^2a_{14}a_{15}a_{16}a_{17}a_{18}a_{19}a_{20}a_{21}a_{22}a_{23}a_{24}a_{25}a_{26}, \\
%\hspace{-1em}
%\minor_{\dfwt[7],\wPinv(\dfwt[7])}(u_-) &= a_1a_2a_3a_4a_5a_6a_7a_8a_9a_{10}a_{11}a_{12}a_{13}a_{14}a_{15}a_{16}a_{17}a_{18}a_{19}a_{20}a_{21}a_{22}a_{23}a_{24}a_{25}a_{26}a_{27}.
%}

Using Algorithm \ref{alg:Plucker_torus_expansion} to determine expressions in toric coordinates for the homogeneous polynomials in Pl\"ucker coordinates, we find that the expressions coincide with those of the generalized minors above, so that we obtain the equalities in \eqref{eq:Freudenthal_minors_and_plucker_expressions} on $\opendunip$, and therefore on $\dunipP$ as well.
\end{proof}

Using the same notation as in \eqref{eq:root_notation_for_Plucker_coords}, Lemmas \ref{lem:Freudenthal_Plucker_coordinates_coincide_with_GLS}, \ref{lem:Freudenthal_minors_coincide_with_denominators_of_potential}, and Proposition \ref{prop:GLS_coord_ring_dunipP} imply:
\begin{cor}\label{cor:Freudenthal_coordinate_ring_of_dunimP}
The coordinate ring of the affine variety $\dunipP$ is given by:
\[
\C[\dunipP]=\C[\p_{\drt}^\T\,|\,\drt\in\wPpdroots][(\p_{27}^\T)^{-1}, (q_{18}^\T)^{-1},(q_{27}^\T)^{-1}, (q_{36}^\T)^{-1}, (q_{36}^{\prime\T})^{-1}, (q_{45}^\T)^{-1}, (q_{54}^\T)^{-1}].
%\C&[\p_1,\p_2,\p_3,\p_4,\p_5',\p_5'',\p_6',\p_6'',\p_7',\p_7'',\p_8',\p_8'',\p_9,\p_9',\p_9'',\p_{10}',\p_{10}'',\p_{11}',\p_{11}'',\p_{12}',\p_{12}'',\p_{13}',\p_{13}'',\p_{14},\p_{15},\p_{16},\p_{17}]\\
%&[\p_{27}^{-1},q_{18}^{-1},q_{36}^{-1},q_{36}'{}^{-1},q_{45}^{-1},q_{54}^{-1}],
\]
This implies that the coordinate ring of the affine variety $\dunimP=(\dunipP)^\T$ is given by
\[
\C[\dunimP]=\C[\p_{\drt}\,|\,\drt\in\wPpdroots][\p_{27}^{-1}, q_{18}^{-1},q_{27}^{-1}, q_{36}^{-1}, (q_{36}^{\prime})^{-1}, q_{45}^{-1}, q_{54}^{-1}].
%\C&[\p_1,\p_2,\p_3,\p_4,\p_5',\p_5'',\p_6',\p_6'',\p_7',\p_7'',\p_8',\p_8'',\p_9,\p_9',\p_9'',\p_{10}',\p_{10}'',\p_{11}',\p_{11}'',\p_{12}',\p_{12}'',\p_{13}',\p_{13}'',\p_{14},\p_{15},\p_{16},\p_{17}]\\
%&[\p_{27}^{-1},q_{18}^{-1},q_{36}^{-1},q_{36}'{}^{-1},q_{45}^{-1},q_{54}^{-1}],
\]
Consequently, $\pi(\dunimP)\subseteq\mX_\can$.
\end{cor}
Recall that the (generalized) Pl\"ucker coordinates $\p_i^{(j)}$ are defined in equation \eqref{eq:Generalized-Plucker-coordinates} and the quadratic, cubic and quartic expressions $q_i$ are given in equation \eqref{eq:Freudenthal_denominators_of_potential}. 
\begin{rem}\label{rem:Freudenthal_inverted_Plucker_and_generators}
Just as in the Remark \ref{rem:CayleyPlane_inverted_Plucker_and_generators}, the localized expressions $\p_{27}$, $q_{18}$, $q_{36}$, $(q_{36}')$, $q_{45}$ and $q_{54}$ are not a priori generated by the generators of the coordinate ring. 
%%Again, there are relations on $\dunimP$ between the remaining Pl\"ucker coordinates and the generating ones, for example, we have\marginbox{Just put them here?}
%%\[
%%q_{18} = p_{18}-p_{1}p_{17} = -p_{6}''p_{12}''+p_{7}''p_{11}''-p_{8}''p_{10}''+p_{9}'p_{9}''
%%\]
%%Thus, the expression for the coordinate ring in Corollary \ref{cor:CayleyPlane_coordinate_ring_of_dunimP} is well-defined.
Recall that the generators were $\p_{\drt}$ for $\drt\in\wPpdroots$, i.e.
\ali{
&\p_1,\p_2,\p_3,\p_4,\p_5',\p_5'',\p_6',\p_6'',\p_7',\p_7'',\p_8',\p_8'',\p_9,\p_9',\p_9'',\p_{10}',\p_{10}'',\\
&\p_{11}',\p_{11}'',\p_{12}',\p_{12}'',\p_{13}',\p_{13}'',\p_{14},\p_{15},\p_{16},\p_{17}.
}
The remaining Pl\"ucker coordinates can be expressed using the following ``Pl\"ucker relations'':
\begingroup\allowdisplaybreaks\begin{align*}
\p_{10}&=		\p_{1}\p_{9}		-\p_{2}\p_{8}'			+\p_{3}\p_{7}'			-\p_{4}\p_{6}'				+\p_{5}'\p_{5}'', &
\p_{11}&=		\p_{1}\p_{10}'	-\p_{2}\p_{9}'			+\p_{3}\p_{8}''		-\p_{4}\p_{7}''				+\p_{5}'\p_{6}'',\\
\p_{12}&=		\p_{1}\p_{11}'	-\p_{2}\p_{10}''		+\p_{3}\p_{9}''		-\p_{5}''\p_{7}''			+\p_{6}'\p_{6}'',& 
\p_{13}&=		\p_{1}\p_{12}'	-\p_{2}\p_{11}''		+\p_{4}\p_{9}''		-\p_{5}''\p_{8}''			+\p_{6}''\p_{7}', \\
\p_{14}'&=	\p_{1}\p_{13}''	-\p_{2}\p_{12}''		+\p_{5}'\p_{9}''		-\p_{6}'\p_{8}''			+\p_{7}'\p_{7}'', &
\p_{14}''&=	\p_{1}\p_{13}'	-\p_{3}\p_{11}''		+\p_{4}\p_{10}''	-\p_{5}''\p_{9}'			+\p_{6}''\p_{8}', \\
\p_{15}'&=	\p_{1}\p_{14}	-\p_{3}\p_{12}''		+\p_{5}'\p_{10}''	-\p_{6}'\p_{9}'				+\p_{7}''\p_{8}', & 
\p_{15}''&=	\p_{2}\p_{13}'	-\p_{3}\p_{12}'		+\p_{4}\p_{11}'		-\p_{5}''\p_{10}'			+\p_{6}''\p_{9}, \\
\p_{16}'&=	\p_{1}\p_{15}	-\p_{4}\p_{12}''		+\p_{5}'\p_{11}''	-\p_{7}'\p_{9}'				+\p_{8}'\p_{8}'', & 
\p_{16}''&=	\p_{2}\p_{14}	-\p_{3}\p_{13}''		+\p_{5}'\p_{11}'	-\p_{6}'\p_{10}'			+\p_{7}''\p_{9} \\
\p_{17}'&=	\p_{1}\p_{16}	-\p_{5}''\p_{12}''	+\p_{6}'\p_{11}''	-\p_{7}'\p_{10}''			+\p_{8}'\p_{9}'', & 
\p_{17}''&=	\p_{2}\p_{15}	-\p_{4}\p_{13}''		+\p_{5}'\p_{12}'	-\p_{7}'\p_{10}'			+\p_{8}''\p_{9}, \\
\p_{18}&=		\p_{1}\p_{17}	-\p_{6}''\p_{12}''	+\p_{7}''\p_{11}''	-\p_{8}''\p_{10}''		+\p_{9}'\p_{9}'', & 
\p_{18}'&=	\p_{2}\p_{16}	-\p_{5}''\p_{13}''	+\p_{6}'\p_{12}'	-\p_{7}'\p_{11}'			+\p_{9}\p_{9}'', \\
\p_{18}''&=	\p_{3}\p_{15}	-\p_{4}\p_{14}			+\p_{5}'\p_{13}'	-\p_{8}'\p_{10}'			+\p_{9}\p_{9}', &
\p_{19}'&=	\p_{2}\p_{17}	-\p_{6}''\p_{13}''	+\p_{7}''\p_{12}'	-\p_{8}''\p_{11}'			+\p_{9}''\p_{10}', \\
\p_{19}''&=	\p_{3}\p_{16}	-\p_{5}''\p_{14}		+\p_{6}'\p_{13}'	-\p_{8}'\p_{11}'			+\p_{9}\p_{10}'', & 
\p_{20}'&=	\p_{3}\p_{17}	-\p_{6}''\p_{14}		+\p_{7}''\p_{13}'	-\p_{9}'\p_{11}'			+\p_{10}'\p_{10}'', \\
\p_{20}''&=	\p_{4}\p_{16}	-\p_{5}''\p_{15}		+\p_{7}'\p_{13}'	-\p_{8}'\p_{12}'			+\p_{9}\p_{11}'', &
\p_{21}'&=	\p_{4}\p_{17}	-\p_{6}''\p_{15}		+\p_{8}''\p_{13}'	-\p_{9}'\p_{12}'			+\p_{10}'\p_{11}'', \\
\p_{21}''&=	\p_{5}'\p_{16}	-\p_{6}'\p_{15}		+\p_{7}'\p_{14}		-\p_{8}'\p_{13}''			+\p_{9}\p_{12}'', &
\p_{22}'&=	\p_{5}''\p_{17}	-\p_{6}''\p_{16}		+\p_{9}''\p_{13}'	-\p_{10}''\p_{12}'		+\p_{11}'\p_{11}'', \\
\p_{22}''&=	\p_{5}'\p_{17}	-\p_{7}''\p_{15}		+\p_{8}''\p_{14}	-\p_{9}'\p_{13}''			+\p_{10}'\p_{12}'', &
\p_{23}&=		\p_{6}'\p_{17}	-\p_{7}''\p_{16}		+\p_{9}''\p_{14}	-\p_{10}''\p_{13}''	+\p_{11}'\p_{12}'', \\
\p_{24}&=		\p_{7}'\p_{17}	-\p_{8}''\p_{16}		+\p_{9}''\p_{15}	-\p_{11}''\p_{13}''	+\p_{12}'\p_{12}'', &
\p_{25}&=		\p_{8}'\p_{17}	-\p_{9}'\p_{16}		+\p_{10}''\p_{15}	-\p_{11}''\p_{14}		+\p_{12}''\p_{13}', \\
\p_{26}&=		\p_{9}\p_{17}	-\p_{10}'\p_{16}	+\p_{11}'\p_{15}	-\p_{12}'\p_{14}		+\p_{13}'\p_{13}''.
\hspace{-.8em}%To get the equations to fit
\end{align*}\endgroup
We can express $\p_{27}$ in terms of the generating Pl\"ucker coordinates as well, but it is more conveniently expressed in terms of the other Pl\"ucker coordinates as follows:
\[
\p_{27}=\p_{1}\p_{26}-\p_{2}\p_{25}+\p_{3}\p_{24}-\p_{4}\p_{23}+\p_{5}'\p_{22}'+\p_{5}''\p_{22}''-\p_{6}'\p_{21}'+\p_{7}'\p_{20}'-\p_{8}'\p_{19}'+\p_{9}\p_{18}-\p_{10}\p_{17}.
\]
\end{rem}

\subsection%[Proof of the isomorphism of LG-models on cominuscule \texorpdfstring{$\ECHS$}{E\_n/P\_n}]
{Proof of the isomorphism of LG-models on cominuscule \texorpdfstring{$\ECHS$}{E\_n\textasciicircum{}sc/P\_n}}
%{The canonical models for cominuscule \texorpdfstring{$\ECHS$}{E\_n\textasciicircum{}sc/P\_n}}\label{sec:Canonical_LG_models_for_ECHS}
\label{sec:Proof_of_isomorphisms_for_ECHS}
Recall that we want to show that $\pi:\dunimP\to\cmX$ is an isomorphism to $\mX_\can\subset\cmX$

\begin{rem}
The following proof follows the same structure as the proof of Theorem 3.9 in Section 8 of \cite{Pech_Rietsch_Lagrangian_Grassmannians}.
\end{rem}

\begin{proof}[Proof of Theorem \ref{thm:ExcFam_CanAndLie_Vars}]
Corollaries \ref{cor:CayleyPlane_coordinate_ring_of_dunimP} and \ref{cor:Freudenthal_coordinate_ring_of_dunimP} imply $\pi(\dunimP)\subseteq\mX_\can$ for cominuscule $\ECHS$.

To show that $\pi:\dunimP\to\mX_\can$ is an isomorphism, it suffices to show that $\pi^*:\C[\mX_\can]\to\C[\dunimP]$ is an isomorphism since both $\dunimP$ and $\mX_\can$ are affine varieties (for $\dunimP$, this follows from Proposition \ref{prop:GLS_coord_ring_dunipP}; for $\mX_\can$, this follows from the fact that it is the complement of an ample divisor). The fact that $\pi^*$ is injective follows from the fact that $\pi:\dunimP\to\mX_\can$ is dominant, which follows from the fact that $\dunimP$ and $\mX_\can$ have the same dimension and that $\pi$ is an open map (as it is the quotient map $u_-\mapsto \udP u_-$).

For surjectivity of $\pi^*:\C[\mX_\can]\to\C[\dunimP]$, it is straightforward to check that each of the generators of $\C[\dunimP]$ as given in Corollaries \ref{cor:CayleyPlane_coordinate_ring_of_dunimP} and \ref{cor:Freudenthal_coordinate_ring_of_dunimP} is a well-defined function on $\mX_\can$, and that each localized expression is non-zero on $\mX_\can$. This is clear from the fact that they are all generalized Pl\"ucker coordinates or expressions in these, and from the definition of $\mX_\can$.
\end{proof}
\begin{rem}
Note that this proof also shows that every coset in $\mX_\can\subset\udP\backslash\udG$ can uniquely be represented by an element of $\dunimP=\dunim\cap\dborelp\bwop\bwo\dborelp$.
\end{rem}

Now we know that $\pi:\dunimP\to\mX_\can$ is an isomorphism, showing that $\pi\circ\varphi$ is an isomorphism of Landau-Ginzburg models reduces to checking that $(\pi\circ\varphi)^*(\pot_\can) = \pot_\Lie$. 
\begin{rem}
We are abusing notation here: a more proper way to write this would be $\bigl((\pi\circ\varphi)\times\Id\bigr)^*(\pot_\can) = \pot_\Lie$, as $\pot_\can:\mX_\can\times\invdtorus\to\C$ and $\pot_\Lie:\mX_\Lie\times\invdtorus\to\C$. We will continue to use this notation as it should not give rise to ambiguities.
\end{rem}

The proof of Theorem \ref{thm:ExcFam_CanAndLie_LGmodels} is obtained by restricting to the torus $\opendunim\subset\dunimP$ of \eqref{eq:df_opendunim} and showing that $(\varphi^{-1})^*(\pot_\Lie) = \pi^*(\pot_\can)$ on $\opendunim\times\invdtorus$. This is most conveniently done separately for the Cayley plane and the Freudenthal variety:
\begin{proof}[Proof of Theorem \ref{thm:ExcFam_CanAndLie_LGmodels} for the Cayley plane] Applying the results of Proposition \ref{prop:LP_LG_model} to the Cayley plane, we see that $(\varphi^{-1})^*(\pot_\Lie)$ can be written as follows on $\opendunim\times\invdtorus$:
\[
[(\varphi^{-1})^*(\pot_\Lie)](u_-,t) = a_1+a_2+\ldots+a_{16} + q\frac{\sum_{(i_j)\in\wPrimeSubExp}\prod_{j=1}^{\ellwPrime} a_{i_j}}{a_1\cdots a_{16}},
\]
where $u_-=\dy_{r_{16}}(a_{16})\dy_{r_{15}}(a_{15})\cdots\dy_{r_1}(a_{1})$ and $q=\sdr_6(t)$. (Recall that $(r_1,\ldots,r_{16})$ is the sequence of indices in the reduced expression of $\wP$ in equation \eqref{eq:CayleyRedExp_wP}.) The set $\wPrimeSubExp$ consists of the reduced subexpressions of the element $\wPrime$ in the fixed reduced expression of $\wP$ as discussed above Propostion \ref{prop:LP_LG_model}.

On the other hand we need to calculate the terms of $\pot_\can$ as given in equation \eqref{eq:pot-for-cayley} applied to an element of $\opendunim$. 
We can use Algorithm \ref{alg:Plucker_torus_expansion} to calculate the toric expressions for the Pl\"ucker coordinates, followed by elementary operations on polynomials (addition, multiplication, and so on) that are somewhat cumbersome to find toric expressions for the terms of the superpotential; we used a computer algebra program and simply state the results here:
%The algorithm to compute these is the same as the one explained at the end of the proof of Lemma \ref{lem:CayleyPlane_minors_coincide_with_denominators_of_potential}. 
%The computations are straightforward, consisting of elementary operations on polynomials (addition, multiplication, and so on). As this is somewhat cumbersome, we used a computer algebra program and simply state the results here:
\ali{
\frac{\p_1}{\p_0} &= a_9 + a_{16}, \qquad %= \dChfdual[_6](u_-,t), \\
\frac{\p_9'}{\p_8} = a_1+a_{8}, \qquad %= \dChfdual[_1](u_-,t), \\
\frac{q_{13}}{q_{12}} = a_5+a_{10}+a_{15}, \\ %= \dChfdual[_5](u_-,t),\\
\frac{q_{17}}{q_{16}} &= a_2+a_{7}+a_{12}, \qquad %= \dChfdual[_3](u_-,t),\\
\frac{q_{21}}{q_{20}} = a_4+a_{13}, \qquad %= \dChfdual[_2](u_-,t), \\
\frac{q_{25}}{q_{24}} = a_3+a_6+a_{11}+a_{14}, \\ %= \dChfdual[_4](u_-,t),\\
q\frac{\p_5''}{\p_{16}} &= q\frac{\sum_{(i_j)\in\wPrimeSubExp}\prod_{j=1}^{\ellwPrime} a_{i_j}}{a_1\cdots a_{16}}, %= \dChedual[_6](u_-,t),
%
%\frac{\p_1}{\p_0} &= a_9 + a_{16}, \qquad %= \dChfdual[_6](u_-,t), \\
%\frac{\p_9'}{\p_8} = a_1+a_{8}, \qquad %= \dChfdual[_1](u_-,t), \\
%\frac{q_{21}}{q_{20}} = a_4+a_{13}, \\%= \dChfdual[_2](u_-,t), \\
%\frac{q_{13}}{q_{12}} &= a_5+a_{10}+a_{15}, \qquad %= \dChfdual[_5](u_-,t),\\
%\frac{q_{25}}{q_{24}} &= a_3+a_6+a_{11}+a_{14}, \qquad %= \dChfdual[_4](u_-,t),\\
%q\frac{\p_5''}{\p_{16}} = q\frac{\sum_{S\in\wPrimeSubsets}\prod_{(\rtb,j)\in S} a_{\wPindex(\rtb,j)}}{a_1\cdots a_{16}}, %= \dChedual[_6](u_-,t),
%
%\frac{\p_1}{\p_0} &= a_1 + a_8, \qquad %= \dChfdual[_6](u_-,t), \\
%\frac{\p_9'}{\p_8} = a_9+a_{16}, \qquad %= \dChfdual[_1](u_-,t), \\
%\frac{q_{21}}{q_{20}} = a_4+a_{13}, \\%= \dChfdual[_2](u_-,t), \\
%\frac{q_{13}}{q_{12}} &= a_2+a_7+a_{12}, \qquad %= \dChfdual[_5](u_-,t),\\
%\frac{q_{17}}{q_{16}} = a_5+a_{10}+a_{15}, \\%= \dChfdual[_3](u_-,t),\\
%\frac{q_{25}}{q_{24}} &= a_3+a_6+a_{11}+a_{14}, \qquad %= \dChfdual[_4](u_-,t),\\
%q\frac{\p_5''}{\p_{16}} = q\frac{\sum_{S\in\wPrimeSubsets}\prod_{(\rtb,j)\in S} a_{\wPindex(\rtb,j)}}{a_1\cdots a_{16}}, %= \dChedual[_6](u_-,t),
}
where we suppressed the arguments of the Pl\"ucker coordinates (they are all applied to $u_-$). Clearly, we have $(\varphi^{-1})^*(\pot_\Lie) = \pi^*(\pot_\can)$ on $\opendunim\times\invdtorus$, but since $\opendunim\subset\dunimP$ is open and dense and $\dunimP$ is irreducible, we conclude that the equality holds on all of $\dunimP\times\invdtorus$. The theorem now follows in the case of the Cayley plane.
\end{proof}

\begin{proof}[Proof of Theorem \ref{thm:ExcFam_CanAndLie_LGmodels} for the Freudenthal variety]
In this case, the results of Propositon \ref{prop:LP_LG_model} yield the following expression for $(\varphi^{-1})^*(\pot_\Lie)$ on $\opendunim\times\invdtorus$:
\[
[(\varphi^{-1})^*(\pot_\Lie)](u_-,t) = a_1+a_2+\ldots+a_{27} + q\frac{\sum_{(i_j)\in\wPrimeSubExp}\prod_{j=1}^{\ellwPrime} a_{i_j}}{a_1\cdots a_{27}},
\]
where $u_-=\dy_{r_{27}}(a_{27})\dy_{r_{26}}(a_{26})\cdots\dy_{r_1}(a_{1})$ and $q=\sdr_7(t)$. (Recall that $(r_1,\ldots,r_{27})$ is the sequence of indices in the reduced expression of $\wP$ in equation \eqref{eq:E7P7_wP_reduced_expression}.)

The terms of $\pot_\can$ applied to an element of $\opendunim$ can be computed using Algorithm \ref{alg:Plucker_torus_expansion} and a computer algebra program to be the following:
\ali{
\frac{\p_1}{\p_0} &= a_1+a_{15}+a_{27}, \qquad%\\%= \dChfdual[_7](u_-,t), \\
\frac{q_{19}}{q_{18}} = a_2+a_9+a_{19}+a_{26}, \qquad%\\%= \dChfdual[_6](u_-,t), \\
\frac{q_{28}}{q_{27}} = a_6+a_{13}+a_{22} \\%= \dChfdual[_2](u_-,t) \\
\frac{q_{37}}{q_{36}} &= a_{10}+a_{18}, \qquad%\\%= \dChfdual[_1](u_-,t), \\
\frac{q_{37}'}{q_{36}'} = a_3+a_8+a_{14}+a_{20}+a_{25}, \qquad%\\%= \dChfdual[_5](u_-,t), \\
\frac{q_{46}}{q_{45}} = a_5+a_{11}+a_{17}+a_{23} \\%= \dChfdual[_3](u_-,t), \\
\frac{q_{55}}{q_{54}} &= a_4+a_7+a_{12}+a_{16}+a_{21}+a_{24}, \qquad%\\%= \dChfdual[_4](u_-,t), \\
q\frac{\p_{10}}{\p_{27}} = q\frac{\sum_{(i_j)\in\wPrimeSubExp}\prod_{j=1}^{\ellwPrime} a_{i_j}}{a_1\cdots a_{27}} %= \dChedual[_7](u_-,t).
}
Here all the Pl\"ucker coordinates are applied to $u_-\in\opendunim$, so this is suppressed. This shows that $(\varphi^{-1})^*(\pot_\Lie) = \pi^*(\pot_\can)$ on $\opendunim\times\invdtorus$, implying that the equality holds on $\dunimP\times\invdtorus$, as $\opendunim\subset\dunimP$ is open and dense and $\dunimP$ is irreducible.
\end{proof}
Thus, we have shown that the canonical models we defined in Section \ref{sec:Canonical_LG_models_for_ECHS} for the exceptional cominuscule homogeneous spaces are isomorphic to the Lie-theoretic Landau-Ginzburg models given in \cite{Rietsch_Mirror_Construction} (discussed in Section \ref{sec:Lie_theoretic_Mirror}). From this we conclude that the canonical models are indeed Landau-Ginzburg models, satisfying
\[
qH^*(\X)_\loc\cong \C[\X_\can\times\invdtorus]/\Jac{\can},
\]
where $\Jac{\can}$ is generated by the derivatives of $\pot_\can$ along $\X_\can$. Due to the geometric Satake correspondence, this isomorphism is given by associating to $\si_i\in H^*(\X)$ the Pl\"ucker coordinate $\p_i \in\C[\X_\can]$. Furthermore, this isomorphism of Landau-Ginzburg models implies in particular that our canonical models have the expected number of critical points, namely $\dim H^*(\X,\C)$, although the expressions of the superpotentials are too unwieldly to calculate them directly.%, and we verify this explicitly in Proposition \ref{prop:crit-pts}. 

\section{Cluster structures on \texorpdfstring{$\cmX$}{X\textasciicircum{}v}}
\label{sec:cluster}

In this section we first review the cluster structures presented in \cite{GLS_partial_flag_varieties_and_preprojective_algebras} for homogeneous spaces of simply-laced type. We then give Pl\"ucker coordinate descriptions for these cluster structures using an algorithm we formulated for expressing generalized minors as polynomials in the Pl\"ucker coordinates. 

\subsection{Cluster algebras and varieties}
\label{sec:cluster-prelims}
A \emph{cluster structure} on a commutative algebra $A$ is a combinatorial presentation of $A$ in terms of \emph{seeds} and \emph{mutations}. Cluster algebras come in two types, $\ACS$-cluster algebras and $\XCS$-cluster algebras, differing in the definition of their mutations. We will only consider $\ACS$-cluster algebras here. To define seeds, we need to fix the field $\C(x_1,\dots, x_n)$ of rational functions in $n$ variables.

A seed $\seed$ of rank $l$ is a pair $\seed=(\bx,B)$ and an integer $m\ge l$, where $\bx\in \C(x_1,\dots, x_n)^m$ is a tuple of algebraically independent elements, and $B$ is an $m\times l$ integer matrix whose upper $l\times l$ submatrix $M$ is skew-symmetrizable, meaning that there is a diagonal matrix $D$ with positive integer entries on the diagonal such that $DM=-(DM)^T$. The elements of $\bx$ are called \emph{cluster variables}, and the first $l$ are called \emph{mutable} while the remaining $m-l$ are called \emph{frozen}. Furthermore, $B$ is called an \emph{exchange matrix}, and in the case that $D=I$, then we may think of $B$ as the adjacency matrix of a directed graph $Q$ known as a \emph{quiver}. Mutation is an operation on seeds that produces new seeds. For $1\le k\le l$, the mutation of $\seed$ in direction $k$ is a new seed $\seed'=(\mu_k(\bx),\mu_k(B))$. We will not work with mutations here, so we omit their full definition.

%, where $\mu_k(B)$ and $\mu_k(\bx)$ are defined as follows:
%\ali{
%b'_{ij}&=\left\{\begin{array}{ll}
%    -b_{ij}, & \text{if $i=k$ or $j=k$,}\\
%    b_{ij}+b_{ik}b_{kj}, & \text{if $i\neq k\neq j$ and $b_{ik},b_{kj}>0$,}\\
%    b_{ij}-b_{ik}b_{kj}, & \text{if $i\neq k\neq k$ and $b_{ik},b_{kj}<0$,}\\
%    b_{ij} & \text{otherwise.}
%\end{array}\right. \\
%\qand
%x_j'&=
%  \left\{\begin{array}{ll}
%    x_j^{-1}\left(\prod_{b_{ij}>0} x_i^{b_{ij}} + \prod_{b_{ij}<0}x_i^{-b_{ij}}\right), & \text{if $j=k$,}\\
%    x_j, & \text{otherwise.}
%	\end{array}\right.
%}
%The $\XCS$-mutation $\bx'=\mu_k^{\XCS}(\bx)$ is given by:
%\[x_j'=
  %\begin{cases}
    %x_j^{-1} & j=k\\
    %x_j(1+x_i)^{b_{ij}} & b_{ij}>0\\
    %x_j(1+x_i^{-1})^{-b_{ji}} & b_{ji}>0\\
    %x_j & \textrm{else}
  %\end{cases}
%\]
%while the $\ACS$-mutation $\bx'=\mu_k^{\ACS}(\bx)$ is given by:
%\[x_j'=
  %\begin{cases}
    %x_j^{-1}\left(\prod_{b_{ij}>0} x_i^{b_{ij}} + \prod_{b_{ij}<0}x_i^{-b_{ij}}\right) & j=k\\
    %x_j & \textrm{else}.
  %\end{cases}
%\]

Given an \emph{initial seed} $\seed$, let $\ACVars$ denote the set of cluster variables appearing in any seed $\seed'$ reachable from $\seed$ by any (finite) sequence of $\ACS$-mutations. Then $A$ is an $\ACS$-cluster algebra if $A=\C[S_{\ACS}]$. Finally, an affine variety $X$ is an $\ACS$-cluster variety if $\C[X]$ is an $\ACS$-cluster algebra. To each frozen variable $x$, we associate a divisor $D_x=\{p\in X~|~ x(p)=0\}$, and to each seed $\seed$ we associate a torus $\CTorusSeed\hookrightarrow X$. The union of all tori $\CTorus=\bigcup_\seed \CTorusSeed$ is a dense, open subset of $X$ often called the \emph{open cluster variety}, and $\CTorus\cup (\bigcup_{x} D_x)$ is a \emph{(partial) compactification} of $\CTorus$, where the union is over the set of frozen variables. Roughly, \emph{cluster varieties} can be thought of as generalizations of toric varieties where the dense torus is replaced by a union of (infinitely many) \emph{cluster tori} identified along birational \emph{mutation} maps, and the torus-invariant divisors are replaced by divisors corresponding to \emph{frozen variables}.

For our purposes, we will need a concrete description of the cluster structure in terms of Pl\"ucker coordinates, and we will provide such descriptions for the cominuscule exceptional types in Section \ref{sec:cluster}. Cluster structures for homogeneous spaces $P^-\backslash G$ were given in \cite{GLS_partial_flag_varieties_and_preprojective_algebras} for $G$ of simply-laced Dynkin type in terms of generalized minors. Their general construction depends on a choice of Lie group $G$, parabolic subgroup $P$, and reduced expression for the longest word $\wo$ of the Weyl group $W$ of $G$, having as prefix a reduced expression for the longest word $w_P$ of the subgroup $W_{P}\subset W$. We present a simplified version of their construction for our particular cases of the Cayley plane and the Freudenthal variety. 

\subsection{Cluster structures on homogeneous spaces}\label{sec:GLS_Cluster_structure_on_HS}
%\begin{itemize}
	%\item \cite{GLS_partial_flag_varieties_and_preprojective_algebras} defines cluster structure on $\C[P^T\backslash G]$ using an embedding of $U_+^r\hookrightarrow P^T\backslash G$, where $U_+^r$ is the unipotent radical of $P=P_k$. ($U_+^r\cong\dunipP$?)
	%\item the initial seed is defined depending on a reduced expression for $\wo$, 
%\end{itemize}
This section follows Sections 9 and 10 of \cite{GLS_partial_flag_varieties_and_preprojective_algebras}. Recall our fixed reduced expression for $\wo=\wo^{-1}=\wop^{-1}\wPinv$ in equation \eqref{eq:Fixed_reduced_expression_for_wo}:
\[
\wo=(s_{r_{\ell_0}}\dots s_{r_{\ell+1}})(s_{r_\ell}\dots s_{r_1}),
\]
which has the reverse numbering of indices compared to \cite{GLS_partial_flag_varieties_and_preprojective_algebras}.
%%FIX: clarified "last occurrence"
To each simple reflection $s_{r_i}$ in the reduced expression for $\wo$, we associate the generalized minor $\phiGLS(i,\dfwt[r_i])=\minor_{u_{\ge i}(\dfwt[r_i]),\wo(\dfwt[r_i])}$ where $u_{\ge i}=s_{r_{\ellwo}}\dots s_{r_i}$ and $\minor_{u(\dfwt),v(\dfwt)}$ is the generalized minor defined in equation \eqref{eq:FZ_gen_minor}. There are $n$ indices $r_i$ such that $\phiGLS(i,\dfwt[r_i])|_{\dunipP}\equiv1$: one for the last occurrence (from left to right in equation \eqref{eq:Fixed_reduced_expression_for_wo}, given by the formula $\min \{l\mid r_l=i\}$) of each simple transposition $s_i$ in our reduced expression for $\wo$. To see this, observe that for these indices we have that $u_{\ge i}(\dfwt[r_i])=\wo(\dfwt[r_i])$, which follows from the fact that the remaining simple transpositions in $\wo$ act trivially on the highest weight. Thus, we have for these last occurrences $r_i$ that
\[
\phiGLS(i,\dfwt[r_i])|_{\dunipP}=\Delta_{\wo(\dfwt[r_i]),\wo(\dfwt[r_i])}|_{\dunipP} =  {}\Bigr\lan(\overline{\wo})^{-1} g\, \overline{\wo}\cdot\hwt[r_i],\hwt[r_i] \Bigr\ran.
\]
In particular for any $g\in\opendunip=(\opendunim)^\T$ (see equation \eqref{eq:df_opendunim}), we have that
\[
\Bigr\lan(\overline{\wo})^{-1} g\, \overline{\wo}\cdot\hwt[r_i],\hwt[r_i] \Bigr\ran=1,
\]
obtained from the constant term after expanding of the decomposition of $g=\dx_{r_1}(a_1)\cdots\dx_{r_{\ellwP}}(a_{\ellwP})$. This equality extends to $\dunipP$ as $\opendunip\subset\dunipP$ dense, so that the desired equality holds. The $\ellwo-n$ indices for which $\phiGLS(i,\dfwt[r_i])|_{\dunipP}\not\equiv1$ are called $\wo$-exchangeable, and we denote the set of $\wo$-exchangeable indices by $\woExch$.

We now describe the generalized minors appearing in our initial seed. We have mutable variables given by $\phiGLS(i,\dfwt[r_i])$ for $i\in \{1,\dots, \ellwP\}\cap\woExch$. In order to define the frozen variables, we set $t_i=\min\{t~|~r_t=i, t>l\}$ unless $i\in I^P$, in which case we set $t_i=\ellwo+i$, $u_{\ge t_i}=e$, and $r_{t_i}=i$. Then we have frozen variables given by %%FIX: defined r_{t_i}=i for i in I^P
\[
\{\phiGLS(t_i,\dfwt[r_{t_i}])~|~1\le i\le n\}\cup \{\Delta_{\dfwt[i],\dfwt[i]}\mid i\in I^P\}.
\]
Note that in the case of the exceptional cominuscule family, $I^P=\{n\}$ is a singleton set. 

\begin{rem}
  Note that $\minor_{\dfwt[i],\dfwt[i]}|_{U_+^P}=1$ by an analogous computation to the one above for $\minor_{\wo(\dfwt[r_i]),\wo(\dfwt[r_i])}|_{U_+^P}=1$. In fact, the $\phiGLS|_{U_+^P}$ are the cluster variables for a cluster structure on $U_+^P$ (without any of the generalized minors $\minor_{\dfwt[i],\dfwt[i]}|_{U_+^P}$). However, we have for each $i$ that $\minor_{\dfwt[i],\dfwt[i]}\neq 1$ on $\cmX$, and we require these additional generalized minors as frozen variables for the cluster structures on $\cmX$ given in this section. 
\end{rem}

Finally, we define the quiver $\ClusterQuiver$ with vertex set
\[
V(\ClusterQuiver)=\bigl(\{1,2,\dots, \ellwP\}\cap e(\wo)\bigr)\cup \bigl(\{t_i ~|~ 1\le i\le n\}\cup \{\Delta_{\dfwt[i],\dfwt[i]}~|~ i\in\IP\}\bigr),
\]
where the last two sets of vertices are frozen. To describe the arrows in the quiver, we define $m^+=\max\{1\le l\le \ellwo~|~l<m,~r_l=r_m\}$ for notational convenience. We now describe when two vertices $l<m\in V(\ClusterQuiver)$ are connected by an arrow. Since we do not allow mutation at frozen vertices, we assume that at least one of $m,l$ is mutable, i.e. $\{m,l\}\cap (\{1,\dots, \ellwP\}\cap e(\wo))\neq\varnothing$. Then there is an arrow $m\rightarrow l$ if and only if $m^+=l$. On the other hand, there is an arrow $m\leftarrow l$ if and only if $l > m^+ > l^+$ and $a_{r_m,r_l}<0$, where $a_{ij}$ denotes the $(i,j)$-entry of the Cartan matrix of $\G$. The edges involving the remaining frozen vertices $\Delta_{\dfwt[i],\dfwt[i]}$ are determined through cluster category theory (see Theorem 10.2 of \cite{GLS_partial_flag_varieties_and_preprojective_algebras}). In our case, the extra vertex corresponding to $\Delta_{\dfwt[n],\dfwt[n]}$ has an arrow to the vertex labeled by $\p_1$ (which in each case corresponds to the generalized minor $\phiGLS(i,\dfwt[r_i])$ where $i$ is the second-smallest index such that $r_i=\DynkinSymmetry(n)$).

\subsection{Cluster algebras and mirror symmetry}
In the version of mirror symmetry we have been working with so far, the superpotentials $\pot_\can$ defined on the Langlands dual spaces encode enumerative information about the original spaces via their Jacobi rings. One of the main reuslts in \cite{Rietsch_Williams_NO_bodies_cluster_duality_and_mirror_symmetry_for_Grassmannians} provides a combinatorial interpretation of mirror symmetry using cluster structures on Grassmannians through two families of polytopes, \emph{Newton-Okounkov bodies} and \emph{superpotential polytopes}. Although it has not yet been studied, we expect that this statement holds for our homogeneous spaces as well, and we now give $\ACS$-cluster structures in terms of Pl\"ucker coordinates for the Langlands dual spaces (by giving an initial seed).

%\subsection{Cluster Structures on Homogeneous Spaces}
%
%In this section, we recall the relevant portions of \cite[Section 9]{GLS_partial_flag_varieties_and_preprojective_algebras} for our situation. Recall our decomposition of $w_0$:
%\[w_0=w_P(w^P)^{-1}=(s_{r_{\ell_0}}\dots s_{r_{\ell+1}})(s_{r_\ell}\dots s_{r_1}),\]
%which has the reverse numbering of indices compared to their conventions. We describe the general construction here, and then give our particular choices in the following sections.
%
\begin{rem}\label{rmk: transposes}
%\marginbox{I changed parts of this remark}In \cite{GLS_partial_flag_varieties_and_preprojective_algebras}, the homogeneous varieties are assumed to be of the form $\P_-\backslash\G$, where $\P_-\supset\unim$; however, we are considering $\cmX=\dP\backslash\dG$ with $\dP\supset\dunip$. This means that we will again have to use the transposition anti-automorphism to translate the results to our context. However, where for the coordinate ring $\dunimP$ we will work on the transpose space $\dunipP=(\dunimP)^\T$, for the cluster structure we will instead pull back the cluster variables under the transposition map.
  While \cite{GLS_partial_flag_varieties_and_preprojective_algebras} works with $P^-\backslash G$, where $P^-$ refers to an opposite parabolic subgroup, we prefer to work with $P\backslash G$. Because of this conventional difference, we worked with $u_-^T\in U_+^\circ=(U_-^\circ)^T$ in Lemma \ref{lem:CayleyPlane_minors_coincide_with_denominators_of_potential}. In this section, we will prefer to work directly on $U_-^\circ$ by using the ``transposes'' of the generalized minors, $\Delta^T_{w_1(\dfwt[i]),w_2(\dfwt[i])}(u_-)$. This corresponds in Algorithm \ref{alg:torus-expansion} to looking for paths in $V(\dfwt[i])$ from $w_1(\dfwt[i])$ to $w_2(\dfwt[i])$ (using $\dChf_i$'s), as opposed to paths from $w_2(\dfwt[i])$ to $w_1(\dfwt[i])$ (using $\dChe_i$'s), and replacing $u_+\in \{\dx_{r_1}(a_1)\dots \dx_{r_\ell}(a_\ell)\}$ with its transpose $u_+^\T\in\{\dy_{r_\ell}(a_\ell)\dots \dy_{r_1}(a_1)\}$. The rest of the computation is entirely analogous. With respect to the notation, the transpose of a generalized minor corresponds to switching the two subscripts and transposing the input. To be more consistent with our previous choice, we could also have worked in the representation with highest weight $\dfwt[\DynkinSymmetry(i)]= -w_0(\dfwt[i])$ instead:
\[
  \minor^T_{w_1(\dfwt[i]),w_2(\dfwt[i])}(g) = \minor^\T_{w_1\wo\cdot\wo(\dfwt[i]),w_2\wo\cdot\wo(\dfwt[i])}(g) = \minor_{w_2\wo(\dfwt[\DynkinSymmetry(i)]),w_1\wo(\dfwt[\DynkinSymmetry(i)])}(g^T).
\]
Note that the equality above is of polynomials in the torus coordinates. The form of generalized minor that will appear in our cluster variables is $\minor^\T_{u_{\ge i}(\dfwt[i]),\wo(\dfwt[i])}$. In the notation of Lemma \ref{lem:CayleyPlane_minors_coincide_with_denominators_of_potential}, these would be written as
\[
  \minor^T_{u_{\ge i}(\dfwt[i]),w_0(\dfwt[i])}(g)=\Delta_{\dfwt[\DynkinSymmetry(i)],u_{\ge i}w_0(\dfwt[\DynkinSymmetry(i)])}(g^T),
\]
and one can immediately check that the frozen variables introduced below are exactly equal to those computed in Lemma \ref{lem:CayleyPlane_minors_coincide_with_denominators_of_potential}. For example, if $u_{\ge i}=\wop$, then we have that
\[
  \minor^T_{\wop(\dfwt[i]),\wo(\dfwt[i])}(u_-)=\minor_{\dfwt[\DynkinSymmetry(i)],(w^P)^{-1}(\dfwt[\DynkinSymmetry(i)])}(u_-^T)
\]
which is exactly one of the denominators that we computed above.
\end{rem}

\subsection{The Cayley plane}

For the Cayley plane, we use the following reduced expression for the longest word $w_0$ of the Weyl group of $E_6$ to carry out the construction of an initial seed:
\[
\wo = (s_1s_2s_3s_1s_4s_3s_1s_2s_4s_3s_5s_4s_2s_3s_4\framebox{\!$s_5s_1s_3s_4s_2$\!})(s_6s_5s_4s_2s_3s_4s_5\underline{s_6}s_1s_3s_4\underline{s_5s_2s_4s_3s_1}).
\]
The last $16$ indices correspond to our previous choice of $(w^P)^{-1}$. We have boxed the indices $t_i$ for $1\le i<6$ and underlined the last occurrence of each $s_i$ ($1\le i\le 6$). Thus, the set of $\wo$-exchangeable indices is given by the complement of the underlined indices:
\[
\woExch=\{1,\dots, 36\}\setminus \{1,2,3,4,5,9\}.
\]
This gives the mutable cluster variables
\begin{equation}\label{eqn:clust-vars-E6}
  \begin{array}{lllll}
    \phiGLS(6,\dfwt[4]) & \phiGLS(7,\dfwt[3]) & \phiGLS(8,\dfwt[1]) & \phiGLS(10,\dfwt[5]) & \phiGLS(11,\dfwt[4]) \\
    \phiGLS(12,\dfwt[3]) & \phiGLS(13,\dfwt[2]) & \phiGLS(14,\dfwt[4]) & \phiGLS(15,\dfwt[5]) & \phiGLS(16,\dfwt[6])
	\end{array}
\end{equation}
and the frozen variables %%FIX: fixed index for t_6
\begin{equation}\label{eqn:frozen-clust-vars-E6}
	\begin{array}{lllll}
    \phiGLS(t_1=20,\dfwt[1]) & \phiGLS(t_2=17,\dfwt[2]) & \phiGLS(t_3=19,\dfwt[3]) & \phiGLS(t_4=18,\dfwt[4]) \\ 
		\phiGLS(t_5=21,\dfwt[5]) & \phiGLS(t_6=42,\dfwt[6]) & \minor^T_{\dfwt[6],\dfwt[6]}.
  \end{array}
\end{equation}
For ease of notation, we define the following polynomials in Pl\"ucker coordinates, in addition to the ones appearing in \eqref{eq:CayleyPlane_denominators_of_potential} and \eqref{eq:CayleyPlane_numerators_of_potential}:
\begin{align*}
  q_8 &= p_1p_7'-p_0p_8, \quad
  q_8' = p_2p_6' - p_1p_7' + p_0p_8, \quad
  q_{10} = p_2p_8''-p_1p_9''+p_0p_{10}', \\
  %q_{12}&=p_1p_{11}''-p_{12}'' \\
  q_{15}&=p_6'(p_1p_8' - p_0p_9') - p_7''(p_1p_7' + p_0p_8)
  %q_{16}&=p_7'p_9' - p_8p_8'\\
  %q_{20} &= p_5''p_{15}-p_4''p_{16} \\
  %q_{24}&=p_2p_{10}''p_{12}'' - p_2p_{11}'p_{11}''+p_1p_{11}''p_{12}'-p_1p_{10}''p_{13}+p_{11}'p_{13}-p_{12}'p_{12}''.
\end{align*}

\begin{lem}\label{lem:plucker-expressions-E6}
The mutable cluster variables in equation \eqref{eqn:clust-vars-E6} can be expressed on $U_-^\circ$ in terms of Pl\"ucker coordinates as:
  \begin{equation}
    \begin{array}{llll}
      \phiGLS(6,\dfwt[4]) = p_3, & \phiGLS(7,\dfwt[3]) = p_2, & \phiGLS(8,\dfwt[1]) = p_1, & \phiGLS(10,\dfwt[5]) = p_4'',\\
      \phiGLS(11,\dfwt[4]) = q_8', & \phiGLS(12,\dfwt[3]) = q_8, &  \phiGLS(13,\dfwt[2]) = p_4', & \phiGLS(14,\dfwt[4]) = q_{15}, \\
      
      \phiGLS(15,\dfwt[5]) = q_{10},&  \phiGLS(16,\dfwt[6]) = p_5'',
    \end{array}
  \end{equation}
whereas the frozen cluster variables in equation \eqref{eqn:frozen-clust-vars-E6} can be expressed as: %%FIX: fixed index for \dfwt[6]
  \begin{equation}
    \begin{array}{llll}
      \phiGLS(17,\dfwt[2]) = q_{12}, & \phiGLS(18,\dfwt[4]) = q_{24}, & \phiGLS(19,\dfwt[3]) = q_{16}, & \phiGLS(20,\dfwt[1]) = p_8, \\
      \phiGLS(21,\dfwt[5]) = q_{20}, & \phiGLS(42,\dfwt[6]) = p_{16}, & \minor_{\dfwt[6],\dfwt[6]} = p_0.
    \end{array}
  \end{equation}
\end{lem}

\begin{proof}
The cluster variables, being generalized minors, are calculated using Algorithm \ref{alg:torus-expansion} (after applying the change of looking for paths going down from $u_{\ge k}(\dfwt[i])$ to $w_0(\dfwt[i])$ as mentioned in Remark \ref{rmk: transposes}), while we use Algorithm \ref{alg:Plucker_torus_expansion} to compute the torus expansions of the Pl\"ucker coordinate polynomials; afterwards, it is a straightforward check that each equality claimed above holds.
%The proof is entirely analogous to that of Lemma \ref{lem:CayleyPlane_minors_coincide_with_denominators_of_potential} after applying the change of looking for paths going down from $u_{\ge k}\dfwt[i]$ to $w_0\dfwt[i]$. (See Remark \ref{rmk: transposes} above.) Concretely, we use Algorithm \ref{alg:torus-expansion} to compute the torus expansions of both sides for each equality claimed above and check that they are equal.
\end{proof}

We provide in Figure \ref{fig:E6-quiver} (left) the cluster quiver as defined by \cite{GLS_partial_flag_varieties_and_preprojective_algebras}, which we described in Subsection \ref{sec:GLS_Cluster_structure_on_HS}. We have labelled the vertices by the cluster variables above in Lemma \ref{lem:plucker-expressions-E6}. 

\begin{figure}[h!t]
  \begin{center}
    \begin{tikzpicture}[thick, scale=0.6, every node/.style={scale=0.6}]
      \tikzset{>=Latex}
      \tikzstyle{df}=[draw,fill]
      \coordinate (one) at (0,.5);
      \coordinate (two) at (3,.25);
      \coordinate (three) at (1,0.5);
      \coordinate (four) at (2,.5);
      \coordinate (five) at (4,.5);
      \coordinate (six) at (5,.5);

      \draw[help lines] 
      (one)--(0,11.5)
      (two)--(3,11.5)
      (three)--(1,11.5)
      (four)--(2,11.5)
      (five)--(4,11.5)
      (six)--(5,11.5);

      \draw[gray, ultra thick] 
      (one)--(three)--(four)--(five)--(six)
      (four)--(two);

      \draw[gray,fill] 
      (one) circle (3pt)
      (two) circle (3pt)
      (three) circle (3pt)
      (four) circle (3pt)
      (five) circle (3pt);

      \draw[gray,fill=white]
      (six) circle (3pt);

      \draw[ultra thin,blue] 
      (5,1) node[df](62){}
      (4,2) node[df](53){}
      (2,3) node[df](44){}
      (1,4) node[df](33){}
      (3,4) node[df](22){}
      (0,5) node[df](12){}
      (-1,12) node[df](00){};

      \draw[circle,ultra thin,black] 
      (2,5) node[df](43){}
      (1,6) node[df](32){}
      (4,6) node[df](52){}
      (2,7) node[df](42){}
      (5,7) node[df](61){}
      (3,8) node[df](21){}
      (4,8) node[df](51){}
      (2,9) node[df](41){}
      (1,10) node[df](31){}
      (0,11) node[df](11){};

      \node at (one) [gray,below = 3pt] {$1$};
      \node at (two) [gray,below = 3pt] {$2$};
      \node at (three) [gray,below = 3pt] {$3$};
      \node at (four) [gray,below = 3pt] {$4$};
      \node at (five) [gray,below = 3pt] {$5$};
      \node at (six) [gray,below = 3pt] {$6$};

      \draw%[red]
      (12) edge[->] (11)
      (22) edge[->] (21)
      (33) edge[->] (32)
      (32) edge[->] (31)
      (44) edge[->] (43)
      (43) edge[->] (42)
      (42) edge[->] (41)
      (53) edge[->] (52)
      (52) edge[->] (51)
      (62) edge[->] (61);

      %\draw[blue, dotted] 
      %(12)--(33)--(44)--(53)--(62)
      %(44)--(22);

      \draw%[black!40!green] 
      (11) edge[->] (32)
      (31) edge[->] (42)
      (21) edge[->] (43)
      (61) edge[->] (53)
      (52) edge[->] (44);

      \draw[black]
      (00) edge[->] (11)
      (41) edge[->] (21)
      (41) edge[->] (51)
      (51) edge[->] (42)
      (51) edge[->] (61)
      (42) edge[->] (32)
      (42) edge[->] (52)
      (32) edge[->] (12)
      (32) edge[->] (43)
      (43) edge[->] (33)
      (43) edge[->] (22);

      %\draw[black, dotted] 
      %(11)--(31)
      %(31)--(41)
      %(21)--(42)
      %(61)--(52)
      %(52)--(43);

      \draw
      (62) node[right=3pt] {$p_{16}$}
      (53) node[below=1pt,left=3pt] {$q_{20}$}
      (44) node[below=1pt,left=3pt] {$q_{24}$}
      (22) node[right=3pt] {$q_{12}$}
      (33) node[below=1pt, left=3pt] {$q_{16}$}
      (12) node[left=3pt] {$p_8$}
      (43) node[below=1pt,right=3pt] {$q_{15}$}
      (32) node[left=3pt] {$q_8$}
      (52) node[right=3pt] {$q_{10}$}
      (42) node[left=3pt] {$q_8'$}
      (61) node[right=3pt] {$p_5''$}
      (21) node[below=1pt,left=3pt] {$p_4'$}
      (51) node[right=3pt] {$p_4''$}
      (41) node[right=3pt] {$p_3$}
      (31) node[right=3pt] {$p_2$}
      (11) node[left=3pt] {$p_1$}
      (00) node[left=3pt] {$p_0$};
    \end{tikzpicture}
		\qquad
    \begin{tikzpicture}[thick, scale=0.6, every node/.style={scale=0.6}]
      \tikzset{>=Latex}
      \tikzstyle{df}=[draw,fill]
      \coordinate (one) at (0,.5);
      \coordinate (two) at (3,.25);
      \coordinate (three) at (1,0.5);
      \coordinate (four) at (2,.5);
      \coordinate (five) at (4,.5);
      \coordinate (six) at (5,.5);

      \draw[help lines] 
      (one)--(0,11.5)
      (two)--(3,11.5)
      (three)--(1,11.5)
      (four)--(2,11.5)
      (five)--(4,11.5)
      (six)--(5,11.5);

      \draw[gray, ultra thick] 
      (one)--(three)--(four)--(five)--(six)
      (four)--(two);

      \draw[gray,fill] 
      (one) circle (3pt)
      (two) circle (3pt)
      (three) circle (3pt)
      (four) circle (3pt)
      (five) circle (3pt);

      \draw[gray,fill=white]
      (six) circle (3pt);

      \draw[circle,ultra thin,black] 
      (5,1) node[df](62){}
      (4,2) node[df](53){}
      (2,3) node[df](44){}
      (1,4) node[df](33){}
      (3,4) node[df](22){}
      (0,5) node[df](12){}
      (2,5) node[df](43){}
      (1,6) node[df](32){}
      (4,6) node[df](52){}
      (2,7) node[df](42){}
      (5,7) node[df](61){}
      (3,8) node[df](21){}
      (4,8) node[df](51){}
      (2,9) node[df](41){}
      (1,10) node[df](31){}
      (0,11) node[df](11){};

      \node at (one) [gray,below = 3pt] {$1$};
      \node at (two) [gray,below = 3pt] {$2$};
      \node at (three) [gray,below = 3pt] {$3$};
      \node at (four) [gray,below = 3pt] {$4$};
      \node at (five) [gray,below = 3pt] {$5$};
      \node at (six) [gray,below = 3pt] {$6$};

      \draw[black] 
      (12) edge[->] (33)
      (33) edge[->] (44)
      (44) edge[->] (53)
      (53) edge[->] (62)
      (22) edge[->] (44)
      (41) edge[->] (21)
      (41) edge[->] (51)
      (51) edge[->] (42)
      (51) edge[->] (61)
      (42) edge[->] (32)
      (42) edge[->] (52)
      (32) edge[->] (12)
      (32) edge[->] (43)
      (43) edge[->] (33)
      (43) edge[->] (22)
      (11) edge[->] (31)
      (31) edge[->] (41)
      (21) edge[->] (42)
      (61) edge[->] (52)
      (52) edge[->] (43);
    \end{tikzpicture}
  \end{center}
  \caption{The cluster quiver (left) as defined in \cite{GLS_partial_flag_varieties_and_preprojective_algebras} and quiver (right) defined in \cite{CMP_Quantum_cohomology_of_minuscule_homogeneous_spaces} mentioned in Remark \ref{rem:GLS_and_CMP_quivers} for the Cayley plane. The frozen variables are blue squares in the cluster quiver.}
  \label{fig:E6-quiver}
\end{figure}
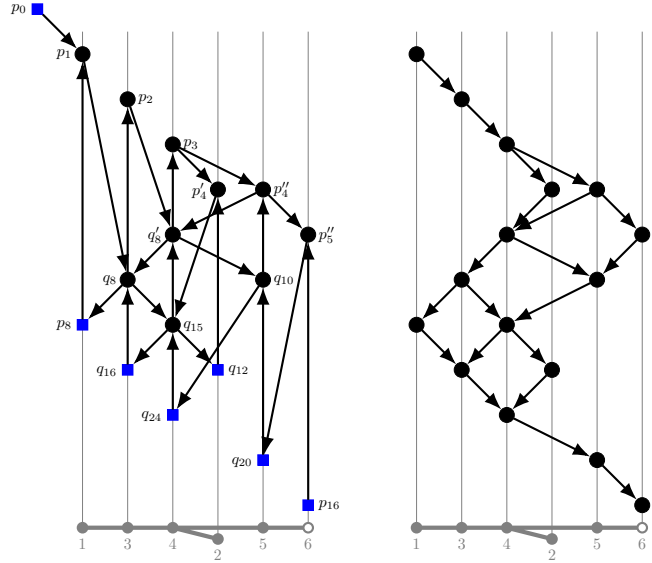
\begin{rem}\label{rem:GLS_and_CMP_quivers}
The cluster quiver is remarkably similar to a quiver defined in \cite{CMP_Quantum_cohomology_of_minuscule_homogeneous_spaces}, which we have drawn to the right in Figure \ref{fig:E6-quiver} for comparison, although this quiver is not known to be a cluster quiver. It is defined using a fixed reduced expression for $\wP$ (denoted $w_X$ there), having as vertices the simple reflections of the expression, and an arrow from each simple reflection to the first occurrence of each non-commuting simple reflection to the right of it. See Definition 2.1 of \cite{CMP_Quantum_cohomology_of_minuscule_homogeneous_spaces} for more details.

This quiver is shown to provide a combinatorial tool used to determine Poincar\'e dual Schubert classes that generalizes Young diagrams for Grassmannians to general (co-) minuscule homogeneous spaces. They also use the quivers to formulate a quantum Chevalley formula, and show that they can be used to determine ``higher quantum Poincar\'e duality'', as well as to determine the smallest power of the quantum parameter appearing in the quantum product of two Schubert classes. The quivers thus play an important role in the Schubert calculus of (co-) minuscule homogeneous spaces, but are not shown to be related to any cluster structure in \cite{CMP_Quantum_cohomology_of_minuscule_homogeneous_spaces}. Note that these quivers are also used in Algorithm \ref{alg:Plucker_torus_expansion} as well as in Corollary 8.12 of \cite{Spacek_LP_LG_models} as combinatorial tools to find subexpressions of a given Weyl group element inside the fixed reduced expression of another element.

We have formulated an algorithm to convert the quiver of \cite{CMP_Quantum_cohomology_of_minuscule_homogeneous_spaces} into the cluster quiver as defined in \cite{GLS_partial_flag_varieties_and_preprojective_algebras}, and we have verified the algorithm in the cases of the Cayley plane, the Freudenthal variety and a number of cases of Grassmannians and quadrics. We intend to verify this more generally and explore this relation between the cluster structure and quantum cohomology further.
\end{rem}

Although it is simple to check that the generalized minor is equal to the given Pl\"ucker coordinate expression restricted to $U_-^\circ$ using Algorithms \ref{alg:torus-expansion} and \ref{alg:Plucker_torus_expansion}, similar to the proof of Lemma \ref{lem:CayleyPlane_minors_coincide_with_denominators_of_potential}, it is difficult to obtain these Pl\"ucker coordinate expressions. We were able to perform the computations directly in this case using a computer algebra program, but the computations for the Freudenthal variety require a more sophisticated approach, which we describe in the following subsection.

\subsection{The Freudenthal variety}\label{sec:Cluster_struc_Freudenthal}
For the Freudenthal variety, we use the following reduced expression for the longest word $\wo$ of the Weyl group of $\LGE_7$ to construct an initial seed:
\begin{align*}  
  \wo=& s_1s_2s_3s_1s_4s_2s_3s_1s_4s_3s_5s_4s_2s_3s_1s_4s_3s_5s_4s_2s_6s_5s_4s_2s_3s_1s_4s_3s_5s_4\fbox{\!$s_2s_6s_5s_4s_3s_1$\!}\\
      &s_7s_6s_5s_4s_3s_2s_4s_5s_6s_1s_3s_4s_7s_5s_2s_4s_3\underline{s_1}s_6s_5s_4\underline{s_2s_3s_4s_5s_6s_7}.
\end{align*}

We have boxed the indices $t_i$ for $1\le i< 7$ and underlined the last occurrence of each $s_i$ ($1\le i\le 7$). Thus, the set of $w_0$-exchangeable indices is given by the complement of the underlined indices:
\[e(w_0)=\{1,\dots, 63\}\setminus \{1,2,3,4,5,6,10\}.\]
This gives the mutable cluster variables:
\begin{equation}\label{eqn:clust-vars-E7}
  \begin{array}{lllll}
    \phiGLS(7,\dfwt[4]) & \phiGLS(8,\dfwt[5]) & \phiGLS(9,\dfwt[6]) & \phiGLS(11,\dfwt[3]) & \phiGLS(12,\dfwt[4]) \\
    \phiGLS(13,\dfwt[2]) & \phiGLS(14,\dfwt[5]) & \phiGLS(15,\dfwt[7]) & \phiGLS(16,\dfwt[4]) & \phiGLS(17,\dfwt[3]) \\
    \phiGLS(18,\dfwt[1]) & \phiGLS(19,\dfwt[6]) & \phiGLS(20,\dfwt[5]) & \phiGLS(21,\dfwt[4]) & \phiGLS(22,\dfwt[2]) \\
    \phiGLS(23,\dfwt[3]) & \phiGLS(24,\dfwt[4]) & \phiGLS(25,\dfwt[5]) & \phiGLS(26,\dfwt[6]) & \phiGLS(27,\dfwt[7])
	\end{array}
\end{equation}
and the frozen variables: %%FIX: fixed index for t_7
\begin{equation}\label{eqn:frozen-clust-vars-E7}
	\begin{array}{llll}
    \phiGLS(t_1=28,\dfwt[1]) & \phiGLS(t_2=33,\dfwt[2]) & \phiGLS(t_3=29,\dfwt[3]) & \phiGLS(t_4=30,\dfwt[4]) \\
		\phiGLS(t_5=31,\dfwt[5]) & \phiGLS(t_6=32,\dfwt[6]) & \phiGLS(t_7=70,\dfwt[7]) & \Delta^T_{\dfwt[7],\dfwt[7]}.
  \end{array}
\end{equation}
As before, we define the following polynomials in Pl\"ucker coordinates in addition to the ones appearing in \eqref{eq:Freudenthal_denominators_of_potential} and \eqref{eq:Freudenthal_numerators_of_potential}, for ease of notation:
\begingroup\allowdisplaybreaks\begin{align*}
  q_{10} &=p_1p_9 - p_{10}, \qquad
  q_{10}' =p_2p_8' - p_1p_9 + p_{10}, \qquad
  q_{10}'' =p_3p_7'-p_2p_8'+p_1p_9-p_{10},\\
  q_{12} &= p_5''p_7''-p_6'p_6'', \qquad
	%=p_3p_9'' - p_2p_{10}'' + p_1p_{11}' - p_{12}, \qquad
  q_{14} =p_2p_{12}'' - p_1p_{13}'' + p_{14}'\qquad 
  %q_{18} &=p_1p_{17} - p_{18}\\
  q_{19}' =p_9'(p_{10}-p_1p_9) - p_8'(p_{11}-p_1p_{10}') \\
	%=p_1(p_3p_{15} - p_4p_{14} + p_5'p_{13}') - p_5'p_{14}'' + p_4p_{15}' - p_3p_{16}'\\
  q_{20} &=p_5'p_{15}''-p_4p_{16}''+p_3p_{17}''-p_2p_{18}'', \qquad
  q_{23} =p_5''(p_{18} -p_1p_{17}) - p_6''(p_{17}'-p_1p_{16})\\
  q_{18}' &=p_2(p_5'p_{11}'' - p_4p_{12}'')+ p_1(p_4p_{13}''-p_5'p_{12}') + p_5'p_{13} - p_4p_{14}'\\
  %q_{27} &=p_6''p_{21}'' - p_5''p_{22}'' + p_4p_{23} - p_3p_{24} + p_2p_{25} - p_1p_{26} + p_{27}\\
  q_{28}' &=(p_1p_{15}-p_{16}')(p_5''p_7''-p_6'p_6'') - (p_1p_{16}-p_{17}')(p_4p_7''-p_5'p_6'')
									+ (p_1p_{17}-p_{18})(p_4p_6'-p_5'p_5'')\\
  q_{30} &=p_{18}''(p_3p_9'' - p_2p_{10}'' + p_1p_{11}' - p_{12}) - p_{19}''(p_3p_8'' - p_2p_9' + p_1p_{10}' - p_{11})\\
									&\quad- p_{20}'(-p_3p_7' + p_2p_8' - p_1p_9 + p_{10})
											+ p_3q_{27}\\
           %&+ p_3(p_7''p_{20}'' - p_6''p_{21}'' - p_6'p_{21}' + p_5''p_{22}'' + p_5'p_{22}' - p_4p_{23} + p_3p_{24} - p_2p_{25} + p_1p_{26} - p_{27})\\
  %q_{36} &=p_{10}p_{26} - p_9p_{27}\\
  %q_{36}' &=p_{22}'(p_{14}'-p_1p_{13}''+p_2p_{12}'') - p_{23}(p_{13} - p_1p_{12}' + p_2p_{11}'') \\
           %&+ p_{24}(p_{12}-p_1p_{11}' + p_2p_{10}'') - p_9''(p_2p_{25} - p_1p_{26} + p_{27})\\
  q_{40} &=(p_2p_{12}''-p_1p_{13}''+p_{14}')(p_6''p_{20}''-p_5''p_{21}'+p_4p_{22}')\\
           &\quad+(p_2p_{11}''-p_1p_{12}'+p_{13})(-p_5'p_{22}'+p_6'p_{21}'-p_7''p_{20}''+q_{27}'')
           %&+(p_2p_{11}''-p_1p_{12}'+p_{13})(-p_5'p_{22}'+p_6'p_{21}'-p_6''p_{21}''-p_7'p_{20}'+p_8'p_{19}'-p_9p_{18}+p_{10}p_{17})
           %&+(p_2p_{11}''-p_1p_{12}'+p_{13})(-p_6''p_{21}''+p_5''p_{22}''-p_4p_{23}+p_3p_{24}-p_2p_{25}+p_1p_{26}-p_{27})
  %q_{45} &=p_{25}(p_9p_{11}-p_{10}p_{10}') + p_{26}(p_9'p_{10} - p_8'p_{11}) + p_{27}(p_8'p_{10}'-p_9p_9')\\
  %q_{54} &=(p_5''p_7'' - p_6'p_6'')(p_{20}''p_{22}'' - p_{21}'p_{21}'') - (p_4p_7'' - p_5'p_6'')(p_{20}''p_{23} - p_{21}''p_{22}')\\
           %&+ (p_4p_6' - p_5'p_5'')(p_{21}'p_{23} - p_{22}'p_{22}'')\\
           %&+(p_3p_{24} - p_2p_{25} + p_1p_{26} - p_{27})(p_{10}p_{17} - p_9p_{18} + p_8'p_{19}' + p_7''p_{20}'' - p_7'p_{20}' - p_6''p_{21}'')\\
\end{align*}\endgroup

\begin{lem} \label{lem:plucker-expressions-E7}
  The mutable cluster variables in equation \eqref{eqn:clust-vars-E7} can be expressed on $U_+^\circ$ in terms of Pl\"ucker coordinates as:
  \begin{equation}
    \begin{array}{llll}
      \phiGLS(7,\dfwt[4]) = p_4, & \phiGLS(8,\dfwt[5]) = p_3, & \phiGLS(9,\dfwt[6]) = p_2, &\phiGLS(11,\dfwt[3]) = p_5'', \\
      \phiGLS(12,\dfwt[4]) = q_{10}'', & \phiGLS(13,\dfwt[2]) = p_5', & \phiGLS(14,\dfwt[5]) = q_{10}', &      \phiGLS(15,\dfwt[7]) = p_1, \\
      \phiGLS(16,\dfwt[4]) = q_{18}',& \phiGLS(17,\dfwt[3]) = q_{12}, & \phiGLS(18,\dfwt[1]) = p_6'', & \phiGLS(19,\dfwt[6]) = q_{10}, \\
      \phiGLS(20,\dfwt[5]) = q_{19}',& \phiGLS(21,\dfwt[4]) = q_{28}', & \phiGLS(22,\dfwt[2]) = q_{14}, &\phiGLS(23,\dfwt[3]) = q_{23}, \\
      \phiGLS(24,\dfwt[4]) = q_{40}, & \phiGLS(25,\dfwt[5]) = q_{30}, & \phiGLS(26,\dfwt[6]) = q_{20}, & \phiGLS(27,\dfwt[7]) = p_{10},
    \end{array}
  \end{equation}
  whereas the frozen cluster variables can be expressed as: %FIX: fixed indices for t_7 = 70 instead of t_7 = -7
  \begin{equation}
    \begin{array}{llll}
      \phiGLS(28,\dfwt[1]) = q_{18},& \phiGLS(29,\dfwt[3]) = q_{36}', &\phiGLS(30,\dfwt[4]) = q_{54}, & \phiGLS(31,\dfwt[5]) = q_{45}, \\
      \phiGLS(32,\dfwt[6]) = q_{36}, & \phiGLS(33,\dfwt[2]) = q_{27}, & \phiGLS(70,\dfwt[7]) = p_{27}, & \Delta^T_{\dfwt[7],\dfwt[7]} = p_0   
    \end{array}
  \end{equation}
\end{lem}

\begin{proof}
The methods are the same as in the proof of Lemma \ref{lem:plucker-expressions-E6}: use Algorithm \ref{alg:torus-expansion} (with the direction change of Remark \ref{rmk: transposes}) for the cluster variables, and use Algorithm \ref{alg:Plucker_torus_expansion} for the Pl\"ucker coordinate polynomials, and verify that the equalities hold on $\opendunim$.
%The proof is entirely analogous to that of Lemma \ref{lem:CayleyPlane_minors_coincide_with_denominators_of_potential} after applying the change of looking for paths going down from $u_{\ge k}\dfwt[i]$ to $w_0\dfwt[i]$. For this, we use Algorithm \ref{alg:torus-expansion} to compute torus expansions for both sides of each equality claimed above and check that they are equal. 
\end{proof}

In Figure \ref{fig:E7-quiver} we provide the quiver, with vertices labelled by the cluster variables above in Lemma \ref{lem:plucker-expressions-E7}. As we noted in Remark \ref{rem:GLS_and_CMP_quivers} before, the cluster quiver is remarkably similar to a quiver defined in \cite{CMP_Quantum_cohomology_of_minuscule_homogeneous_spaces}, which we have drawn to the right for comparison. 

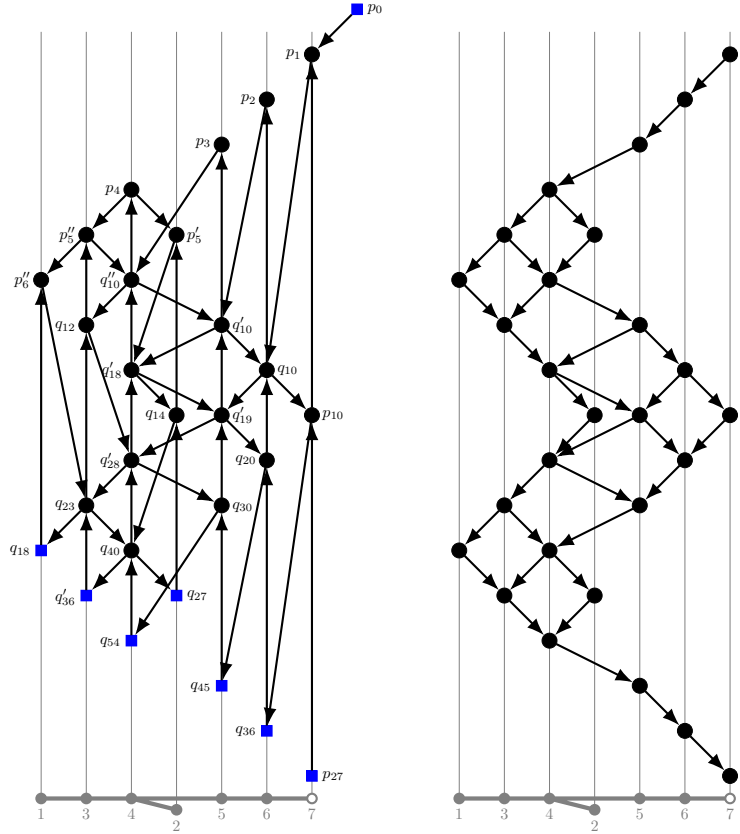
\begin{figure}[h!t]
  \begin{center}
    \begin{tikzpicture}[thick, scale=0.6, every node/.style={scale=0.6}]
      \tikzset{>=Latex}
      \tikzstyle{df}=[draw,fill]
      \coordinate (one) at (0,-.5);
      \coordinate (two) at (3,-.75);
      \coordinate (three) at (1,-0.5);
      \coordinate (four) at (2,-.5);
      \coordinate (five) at (4,-.5);
      \coordinate (six) at (5,-.5);
      \coordinate (seven) at (6,-.5);

      \draw[help lines] 
      (one)--(0,16.5)
      (two)--(3,16.5)
      (three)--(1,16.5)
      (four)--(2,16.5)
      (five)--(4,16.5)
      (six)--(5,16.5)
      (seven)--(6,16.5);

      \draw[gray, ultra thick] 
      (one)--(three)--(four)--(five)--(six)--(seven)
      (four)--(two);

      \draw[gray,fill] 
      (one) circle (3pt)
      (two) circle (3pt)
      (three) circle (3pt)
      (four) circle (3pt)
      (five) circle (3pt)
      (six) circle (3pt);

      \draw[gray,fill=white]
      (seven) circle (3pt);

      \draw[ultra thin,blue]
      (6,0) node[df](72){}
      (5,1) node[df](62){}
      (4,2) node[df](53){}
      (2,3) node[df](44){}
      (3,4) node[df](22){}
      (1,4) node[df](33){}
      (0,5) node[df](12){}
      (7,17) node[df](00){};

      \draw[circle,ultra thin,black] 
      (2,5) node[df](43){}
      (1,6) node[df](32){}
      (4,6) node[df](52){}
      (2,7) node[df](42){}
      (5,7) node[df](61){}
      (3,8) node[df](21){}
      (4,8) node[df](51){}
      (6,8) node[df](71){}
      (2,9) node[df](41){}
      (5,9) node[df](60){}
      (1,10) node[df](31){}
      (4,10) node[df](50){}
      (0,11) node[df](11){}
      (2,11) node[df](40){}
      (1,12) node[df](30){}
      (3,12) node[df](20){}
      (2,13) node[df](4-1){}
      (4,14) node[df](5-1){}
      (5,15) node[df](6-1){}
      (6,16) node[df](70){};

      \node at (one) [gray,below = 3pt] {$1$};
      \node at (two) [gray,below = 3pt] {$2$};
      \node at (three) [gray,below = 3pt] {$3$};
      \node at (four) [gray,below = 3pt] {$4$};
      \node at (five) [gray,below = 3pt] {$5$};
      \node at (six) [gray,below = 3pt] {$6$};
      \node at (seven) [gray,below = 3pt] {$7$};

      \draw%[red]
      (12) edge[->] (11)
      (22) edge[->] (21)
      (21) edge[->] (20)
      (33) edge[->] (32)
      (32) edge[->] (31)
      (31) edge[->] (30)
      (44) edge[->] (43)
      (43) edge[->] (42)
      (42) edge[->] (41)
      (41) edge[->] (40)
      (40) edge[->] (4-1)
      (53) edge[->] (52)
      (52) edge[->] (51)
      (51) edge[->] (50)
      (50) edge[->] (5-1)
      (62) edge[->] (61)
      (61) edge[->] (60)
      (60) edge[->] (6-1)
      (72) edge[->] (71)
      (71) edge[->] (70);

      %\draw[blue, dotted] 
      %(12)--(33)--(44)--(53)--(62)--(72)
      %(44)--(22);

      \draw%[black!40!green] 
      (70) edge[->] (60)
      (6-1) edge[->] (50)
      (5-1) edge[->] (40)
      (20) edge[->] (41)
      (11) edge[->] (32)
      (31) edge[->] (42)
      (21) edge[->] (43)
      (61) edge[->] (53)
      (52) edge[->] (44)
      (71) edge[->] (62);

      \draw[black]
      (4-1) edge[->] (20)
      (4-1) edge[->] (30)
      (30) edge[->] (11)
      (30) edge[->] (40)
      (40) edge[->] (31)
      (40) edge[->] (50)
      (50) edge[->] (41)
      (50) edge[->] (60)
      (60) edge[->] (51)
      (60) edge[->] (71)
      (41) edge[->] (21)
      (41) edge[->] (51)
      (51) edge[->] (42)
      (51) edge[->] (61)
      (42) edge[->] (32)
      (42) edge[->] (52)
      (32) edge[->] (12)
      (32) edge[->] (43)
      (43) edge[->] (33)
      (43) edge[->] (22)
      (00) edge[->] (70);

      %\draw[black, dotted] 
      %(70)--(6-1)
      %(6-1)--(5-1)
      %(5-1)--(4-1)
      %(20)--(40)
      %(11)--(31)
      %(31)--(41)
      %(21)--(42)
      %(61)--(52)
      %(52)--(43)
      %(71)--(61);

      \draw
      (72) node[right=3pt] {$p_{27}$}
      (71) node[right=3pt] {$p_{10}$}
      (70) node[left=3pt] {$p_1$}

      (62) node[left=3pt] {$q_{36}$}
      (61) node[left=3pt] {$q_{20}$}
      (60) node[right=3pt] {$q_{10}$}
      (6-1) node[left=3pt] {$p_2$}
      
      (53) node[below=1pt,left=3pt] {$q_{45}$}
      (52) node[below=1pt,right=3pt] {$q_{30}$}
      (51) node[right=3pt] {$q_{19}'$}
      (50) node[right=3pt] {$q_{10}'$}
      (5-1) node[above=1pt,left=3pt] {$p_3$}

      (44) node[below=1pt,left=3pt] {$q_{54}$}
      (43) node[left=3pt] {$q_{40}$}
      (42) node[left=3pt] {$q_{28}'$}
      (41) node[left=2pt] {$q_{18}'$}
      (40) node[left=3pt] {$q_{10}''$}
      (4-1) node[left=3pt] {$p_4$}

      (33) node[below=1pt, left=3pt] {$q_{36}'$}
      (32) node[left=3pt] {$q_{23}$}
      (31) node[below=1pt,left=3pt] {$q_{12}$}
      (30) node[left=3pt] {$p_5''$}
      
      (22) node[right=3pt] {$q_{27}$}
      (21) node[left=3pt] {$q_{14}$}
      (20) node[right=3pt] {$p_5'$}
      
      (12) node[left=3pt] {$q_{18}$}
      (11) node[left=3pt] {$p_6''$}

      (00) node[right=3pt] {$p_0$};
    \end{tikzpicture}
		\qquad
    \begin{tikzpicture}[thick, scale=0.6, every node/.style={scale=0.6}]
      \tikzset{>=Latex}
      \tikzstyle{df}=[draw,fill]
      \coordinate (one) at (0,-.5);
      \coordinate (two) at (3,-.75);
      \coordinate (three) at (1,-0.5);
      \coordinate (four) at (2,-.5);
      \coordinate (five) at (4,-.5);
      \coordinate (six) at (5,-.5);
      \coordinate (seven) at (6,-.5);

      \draw[help lines] 
      (one)--(0,16.5)
      (two)--(3,16.5)
      (three)--(1,16.5)
      (four)--(2,16.5)
      (five)--(4,16.5)
      (six)--(5,16.5)
      (seven)--(6,16.5);

      \draw[gray, ultra thick] 
      (one)--(three)--(four)--(five)--(six)--(seven)
      (four)--(two);

      \draw[gray,fill] 
      (one) circle (3pt)
      (two) circle (3pt)
      (three) circle (3pt)
      (four) circle (3pt)
      (five) circle (3pt)
      (six) circle (3pt);

      \draw[gray,fill=white]
      (seven) circle (3pt);

      \draw[circle,ultra thin,black]
      (6,0) node[df](72){}
      (5,1) node[df](62){}
      (4,2) node[df](53){}
      (2,3) node[df](44){}
      (3,4) node[df](22){}
      (1,4) node[df](33){}
      (0,5) node[df](12){} 
      (2,5) node[df](43){}
      (1,6) node[df](32){}
      (4,6) node[df](52){}
      (2,7) node[df](42){}
      (5,7) node[df](61){}
      (3,8) node[df](21){}
      (4,8) node[df](51){}
      (6,8) node[df](71){}
      (2,9) node[df](41){}
      (5,9) node[df](60){}
      (1,10) node[df](31){}
      (4,10) node[df](50){}
      (0,11) node[df](11){}
      (2,11) node[df](40){}
      (1,12) node[df](30){}
      (3,12) node[df](20){}
      (2,13) node[df](4-1){}
      (4,14) node[df](5-1){}
      (5,15) node[df](6-1){}
      (6,16) node[df](70){};

      \node at (one) [gray,below = 3pt] {$1$};
      \node at (two) [gray,below = 3pt] {$2$};
      \node at (three) [gray,below = 3pt] {$3$};
      \node at (four) [gray,below = 3pt] {$4$};
      \node at (five) [gray,below = 3pt] {$5$};
      \node at (six) [gray,below = 3pt] {$6$};
      \node at (seven) [gray,below = 3pt] {$7$};

      \draw 
      (12) edge[->] (33)
      (33) edge[->] (44)
      (44) edge[->] (53)
      (53) edge[->] (62)
      (62) edge[->] (72)
      (22) edge[->] (44)
      (4-1) edge[->] (20)
      (4-1) edge[->] (30)
      (30) edge[->] (11)
      (30) edge[->] (40)
      (40) edge[->] (31)
      (40) edge[->] (50)
      (50) edge[->] (41)
      (50) edge[->] (60)
      (60) edge[->] (51)
      (60) edge[->] (71)
      (41) edge[->] (21)
      (41) edge[->] (51)
      (51) edge[->] (42)
      (51) edge[->] (61)
      (42) edge[->] (32)
      (42) edge[->] (52)
      (32) edge[->] (12)
      (32) edge[->] (43)
      (43) edge[->] (33)
      (43) edge[->] (22)
      (70) edge[->] (6-1)
      (6-1) edge[->] (5-1)
      (5-1) edge[->] (4-1)
      (20) edge[->] (40)
      (11) edge[->] (31)
      (31) edge[->] (41)
      (21) edge[->] (42)
      (61) edge[->] (52)
      (52) edge[->] (43)
      (71) edge[->] (61);

    \end{tikzpicture}
  \end{center}
  \caption{The cluster quiver (left) as defined in \cite{GLS_partial_flag_varieties_and_preprojective_algebras} and quiver (right) defined in \cite{CMP_Quantum_cohomology_of_minuscule_homogeneous_spaces} mentioned in Remark \ref{rem:GLS_and_CMP_quivers} for the Freudenthal variety. The frozen variables are blue squares in the cluster quiver.}
  \label{fig:E7-quiver}
\end{figure}

It is straightforward to check that the generalized minor is equal to the given Pl\"ucker coordinate expression when restricted to $U_-^\circ$ using Algorithms \ref{alg:torus-expansion} and \ref{alg:Plucker_torus_expansion}. However, it is difficult to obtain the Pl\"ucker coordinate expressions themselves, and they are sufficiently complicated that we could not obtain them by brute force. Instead, we used the following algorithm. In a preprocessing step, we first compute the expansions of the Pl\"ucker coordinates in torus coordinates on $U_-^\circ$, yielding a list of homogeneous polynomials in the torus coordinates. Because we work in a minuscule representation, the highest power appearing in any Pl\"ucker coordinate expansion is $1$, see Theorem \ref{thm:Green_structure_minuscule_reps}.

\begin{algo}\phantomsection\label{alg:subduction-for-gen-minors}
  \begin{enumerate}
  \item For a given generalized minor $\phiGLS(i,\dfwt[r_i])$, first compute its expansion in torus coordinates using Algorithm \ref{alg:torus-expansion} with the direction change of Remark \ref{rmk: transposes}.
  \item Fix the monomial order $<$ on $\mathbb{C}[a_1,\dots, a_{\ell}]$ induced by the degree-lexicographic order using the ordering $a_1>a_2>\dots>a_{\ell}$ of the torus coordinates. 
  \item Let $L=\mathrm{in}_<(\phiGLS(i,\dfwt[r_i]))$ denote the minimal term of $\phiGLS(i,\dfwt[r_i])$ with respect to $<$. 
  \item Let $m$ denote the highest power of any torus coordinate appearing in $L$, and set $L_j$ to be the product of all torus coordinates appearing in $L$ with power at least $j$ for $1\le j\le m$.
  \item Next, let $p_{a_j}$ denote the Pl\"ucker coordinate with $L_j$ as leading term for $1\le j\le m$.
  \item If $\phiGLS(i,\dfwt[r_i])\pm \prod_{j=1}^m p_{a_j}\neq 0$ (the sign is chosen to cancel the leading term of $L$), then repeat the above steps (3)-(5) with the leading term of this new expression.
  \item Else if we obtain $0$, then we have found our desired Pl\"ucker coordinate expression.
  \end{enumerate}
\end{algo}

While it is not clear whether step (5) is always possible, we were able to perform it in every case that we computed. We now give an example of this reduction procedure for $\phiGLS(23,\dfwt[3])$. 

\begin{ex}
  \label{ex:subduction}
  Consider $f=\phiGLS(23,\dfwt[3])$. We compute its torus expansion (including only the last few terms for brevity, but see Appendix \ref{sec:expansion-appendix} for the full expansion):
  \begin{align*}
    &\dots+a_{5}a_{7}a_{13}a_{14}a_{16}a_{17}a_{18}^2a_{19}a_{20}a_{21}a_{22}a_{23}^2a_{24}^2a_{25}^2a_{26}^2a_{27}^3\\
    &+a_{5}a_{12}a_{13}a_{14}a_{16}a_{17}a_{18}^2a_{19}a_{20}a_{21}a_{22}a_{23}^2a_{24}^2a_{25}^2a_{26}^2a_{27}^3\\
    &+a_{11}a_{12}a_{13}a_{14}a_{16}a_{17}a_{18}^2a_{19}a_{20}a_{21}a_{22}a_{23}^2a_{24}^2a_{25}^2a_{26}^2a_{27}^3.
  \end{align*}
  Its $<$-minimal term is
  \[\mathrm{in}_<(\phiGLS(23,\dfwt[3]))=a_{11}a_{12}a_{13}a_{14}a_{16}a_{17}a_{18}^2a_{19}a_{20}a_{21}a_{22}a_{23}^2a_{24}^2a_{25}^2a_{26}^2a_{27}^3,\]
  which has highest exponent $m=3$. Thus we have
  \begin{align*}
    &L_1 = a_{27}\\
    &L_2 = a_{18}a_{23}a_{24}a_{25}a_{26}a_{27}\\
    &L_3 = a_{11}a_{12}a_{13}a_{14}a_{16}a_{17}a_{18}a_{19}a_{20}a_{21}a_{22}a_{23}a_{24}a_{25}a_{26}a_{27}
  \end{align*}
  These are the minimal terms of $p_1, p_6'',p_{16}$ respectively. $f-p_1p_6''p_{16}\neq 0$, and we find that its minimal term is
  \[-a_{10}a_{11}a_{12}a_{13}a_{14}a_{16}a_{17}a_{18}a_{19}a_{20}a_{21}a_{22}a_{23}^2a_{24}^2a_{25}^2a_{26}^2a_{27}^3,\]
  which has highest exponent $m=3$. Thus we have
  \begin{align*}
    &L_1 = a_{27}\\
    &L_2 = a_{23}a_{24}a_{25}a_{26}a_{27}\\
    &L_3 = a_{10}a_{11}a_{12}a_{13}a_{14}a_{16}a_{17}a_{18}a_{19}a_{20}a_{21}a_{22}a_{23}a_{24}a_{25}a_{26}a_{27}
  \end{align*}
  which are the minimal terms of $p_1,p_5'',p_{17}$ respectively. $f-p_1p_6''p_{16}+p_1p_5''p_{17}\neq 0$, and we find that its minimal term is
  \[-a_{11}a_{12}a_{13}a_{14}a_{15}a_{16}a_{17}a_{18}^2a_{19}a_{20}a_{21}a_{22}a_{23}^2a_{24}^2a_{25}^2a_{26}^2a_{27}^2,\]
  which has highest exponent $m=2$. Thus we have
  \begin{align*}
    &L_1=a_{18}a_{23}a_{24}a_{25}a_{26}a_{27}\\
    &L_2=a_{11}a_{12}a_{13}a_{14}a_{15}a_{16}a_{17}a_{18}a_{19}a_{20}a_{21}a_{22}a_{23}a_{24}a_{25}a_{26}a_{27}
  \end{align*}
  which are the minimal terms of $p_6'',p_{17}'$ respectively. $f-p_1p_6''p_{16}+p_1p_5''p_{17}+p_6''p_{17}'\neq 0$, and we find that its minimal term is
  \[a_{10}a_{11}a_{12}a_{13}a_{14}a_{15}a_{16}a_{17}a_{18}a_{19}a_{20}a_{21}a_{22}a_{23}^2a_{24}^2a_{25}^2a_{26}^2a_{27}^2\]
  which has highest exponent $m=2$. Thus we have
  \begin{align*}
    &L_1 = a_{23}a_{24}a_{25}a_{26}a_{27}\\
    &L_2 = a_{10}a_{11}a_{12}a_{13}a_{14}a_{15}a_{16}a_{17}a_{18}a_{19}a_{20}a_{21}a_{22}a_{23}a_{24}a_{25}a_{26}a_{27}
  \end{align*}
  which are the minimal terms of $p_5'',p_{18}$ respectively. Now $f-p_1p_6''p_{16}+p_1p_5''p_{17}+p_6''p_{17}'-p_5''p_{18}=0$, so we obtain the expression
  \[\phiGLS(23,\dfwt[3])=p_5''(p_{18}-p_1p_{17})-p_6''(p_{17}'-p_1p_{16})\]
  exactly as given in Lemma \ref{lem:plucker-expressions-E7}.
\end{ex}

\section{Valuations, Khovanskii bases, and Newton-Okounkov bodies}
\label{sec:khovanskii-basis}
Algorithm \ref{alg:subduction-for-gen-minors} is very reminiscent of the \emph{subduction algorithm} for \emph{Khovanskii bases}. In this section we will show that the Pl\"ucker coordinates form a Khovanskii basis for $\mathbb{C}[\cmX]$, and we then use this to give a \emph{Newton-Okounkov body} for $\cmX$. Throughout this section, we denote by $A$ a domain that is also a finitely generated algebra over an algebraically closed field $k$, and by $\Gamma\subset \mathbb{Q}^r$ a subgroup of the rational lattice with an ordering $\succ$ of its elements that respects the group operation. %%FIX: changed \Gamma to be a subgroup of the rational lattice Q^r

\begin{df}[\cite{Kaveh_Manon_Khovanskii_bases}, Definition 2.1]
  A \emph{valuation} is a function $\nu:A\setminus\{0\}\rightarrow \Gamma$ such that
  \begin{enumerate}
  \item For all $f,g\in A\setminus\{0\}$ with $f+g\neq0$, we have $\nu(f+g)\succeq \min_{\succ}(\nu(f),\nu(g))$. 
  \item For all $f,g\in A\setminus\{0\}$, we have $\nu(fg)=\nu(f)+\nu(g)$.
  \item For all $f\in A\setminus\{0\}$ and $0\neq c\in k$, we have $\nu(cf)=\nu(f)$.
  \end{enumerate}
  For $g\in\Gamma$, let $F_{\succeq g}$ be the vector space $\{f\in A\setminus \{0\}\mid \nu(f)\succeq g\}\cup \{0\}$, and define $F_{\succ g}$ analogously. We then define the \emph{associated graded algebra}
  \[\mathrm{gr}_\nu(A)=\bigoplus_{g\in\Gamma} F_{\succeq g}/F_{\succ g}.\]
  If for all $g\in \Gamma$, the quotient $F_{\succeq g}/F_{\succ g}$ is at most one-dimensional, then we say that $\nu$ has \emph{one-dimensional leaves}. 
\end{df}

The following lemma gives an important family of valuations with one-dimensional leaves.

\begin{lem}
  \label{lem:one-dim-leaves}
  Let $S\subset k[x_1,\dots, x_n]$ be a finitely generated subalgebra of a polynomial ring, and let $\nu$ be the valuation sending $f\in S$ to the exponent vector of its degree-lexicographically minimal term. Then $\nu$ has one-dimensional leaves. 
\end{lem}

\begin{proof}
  Let $g\in \mathbb{Z}^n$, and suppose $a,b\in F_{\succeq g}\setminus F_{\succ g}$. (If $a,b\in F_{\succ g}$, then $a,b=0$ in the quotient $F_{\succeq g}/F_{\succ g}$.) Then $\nu(a)=\nu(b)=g$, so the minimal terms of $a$ and $b$ have the same exponent vector. Thus, there is some constant $c\in k$ such that the minimal terms of the expansions of $a$ and $cb$ are equal, so that $\nu(a-cb)\succ g$, and hence $a-cb\in F_{\succ g}$. Hence, $a=cb$ in $F_{\succeq g}/F_{\succ g}$, and we conclude that $F_{\succeq g}/F_{\succ g}$ has dimension at most $1$.
\end{proof}

We next define \emph{Khovanskii bases} and the \emph{Newton-Okounkov body} associated to a valuation.

\begin{df}[\cite{Kaveh_Manon_Khovanskii_bases}, Definition 2.5]
  A subset $B\subset A$ is called a \emph{Khovanskii basis} for $A$ with respect to the valuation $\nu$ if the image of $B$ in $\mathrm{gr}_\nu(A)$ is a set of algebra generators for $\mathrm{gr}_\nu(A)$. Khovanskii bases give a \emph{subduction algorithm} for writing elements of $A$ as polynomials in the elements of $B$. 
\end{df}

\begin{df}[\cite{Kaveh_Manon_Khovanskii_bases}, Definition 2.21] \label{def:no-body}
  Let $A$ further be positively graded by $\mathbb{Z}_{>0}$. The \emph{Newton-Okounkov body} $\Delta(A,\nu)$ associated to $A$ and a valuation $\nu$ is the closed convex set
  \[
	\Delta(A,\nu)=\overline{\mathrm{conv}\!\left(\bigcup_{i>0}\tfrac1i{\nu(A_i\setminus\{0\})}\right)}
	\]
  where $A_i$ denotes the set of homogeneous elements of $A$ of degree $i$. 
\end{df}

Although the Newton-Okounkov body associated to a valuation $\nu$ is generally very far from being a polytope, the following proposition gives an important family of examples where $\Delta(A,\nu)$ is a polytope.

\begin{prop}[\cite{Kaveh_Manon_Khovanskii_bases}, Corollary 2.24, Remark 2.25]
  \label{prop:no-volume}
  Let $Y$ be a $d$-dimensional projective variety $Y$ embedded in some $\mathbb{P}^N$, and let $A$ denote the homogeneous coordinate ring of $Y$ with respect to this embedding. Let $\nu$ be a valuation on $A$ with one-dimensional leaves. Then $\deg(Y)=\mathrm{Vol}(\Delta(A,\nu))$, where $\deg(Y)$ is computed with respect to the projective embedding of $Y$ and $\mathrm{Vol}$ refers to the normalized lattice volume so that standard lattice $d$-simplex has volume $\frac{1}{d!}$. If furthermore $A$ has a finite Khovanskii basis with respect to $\nu$, then $\Delta(A,\nu)$ is a rational polytope and $\mathrm{Proj}(A)$ has a degeneration to a toric variety whose normalization is the toric variety associated to $\Delta(A,\nu)$.
\end{prop}%%FIX: clarified statement to explicitly state the projective embedding.

%%FIX: fixed to say that the image of A under the inclusion map induced by the inclusion of the torus
Now let $A=\mathbb{C}[\cmX]$ be the coordinate ring of the Cayley plane ($n=6$) or the Freudenthal variety ($n=7$), with $\mathbb{Z}_{>0}$-grading given by the total degree in Pl\"ucker coordinates. We use the degree-lexicographic term order $<$ on $\mathbb{C}[a_1,\dots, a_\ell]$ from Algorithm \ref{alg:subduction-for-gen-minors} to define a valuation $\nu_n:A\rightarrow \mathbb{Z}^\ell$ as follows. Given a polynomial $f$ in the Pl\"ucker coordinates, compute its expansion in the coordinates of $U_+^\circ$ using Algorithm \ref{alg:torus-expansion}, and define $\nu_n(f)$ to be the exponent vector of the degree-lexicographic minimal term $\mathrm{in}_<(f)$ of $f$. It is straightforward to check that $\nu_n$ is a valuation. Furthermore, because the image of the inclusion $A\hookrightarrow \mathbb{C}[a_1,\dots, a_\ell]$ induced by $U_-^\circ\hookrightarrow \cmX$ is a finitely generated subalgebra, Lemma \ref{lem:one-dim-leaves} implies that $\nu_n$ has one-dimensional leaves. Proposition \ref{prop:no-volume} thus shows that the normalized volume of $\Delta(A,\nu_n)$ is equal to the degree of $\cmX$. By Definition \ref{def:no-body}, the suitably scaled valuation of any homogeneous element of $A$ is contained in the Newton-Okounkov body $\Delta(A,\nu_n)$, and thus so is the convex hull of the set of suitably scaled valuations of any finite subset of homogeneous elements of $A$. In particular, if we can find $B\subset A$ such that 
\[
\delta(B,\nu_n)=\mathrm{conv}\!\left(\bigcup_{i>0}\tfrac1i{\nu_n(B\cap A_i\setminus\{0\})}\right)
\]
has normalized volume equal to the degree of $\cmX$, then $\delta(B,\nu_n)=\Delta(A,\nu_n)$.

\begin{prop}
  \label{prop:nobody}
  The Newton-Okounkov body $\Delta(A,\nu_n)$ is equal to the convex hull of the valuations of the Pl\"ucker coordinates on $\cmX$. 
\end{prop}

\begin{proof}
  The proof is by direct computation. We give the valuations of Pl\"ucker coordinates in Appendix \ref{appendix:vals}, and we used Polymake \cite{polymake-2000,polymake-2017} to verify that their convex hulls have the correct volumes.

  For the Cayley plane ($n=6$), we obtain a $16$-dimensional polytope with volume $78$, which is equal to the degree of the Cayley plane. For the Freudenthal variety ($n=7$), we obtain a $27$-dimensional polytope with volume $13110$, which is equal to the degree of the Freudenthal variety. We conclude that in both cases, the respective Newton-Okounkov bodies are given by the convex hulls of the valuations of the respective Pl\"ucker coordinates.
\end{proof}

In particular, $\Delta(A,\nu_n)$ is a rational polytope (in fact a $0/1$-polytope, which follows from the fact that we used a minuscule representation) with vertices given by the valuations of Pl\"ucker coordinates. Moreover, we computed in each case that each polytope contains no lattice points other than its vertices. 

At this point, we would like to conclude that the Pl\"ucker coordinates form Khovanskii bases. However, the converse direction of Proposition \ref{prop:no-volume} requires some further assumptions, e.g.~as in Proposition 17.4 of \cite{Rietsch_Williams_NO_bodies_cluster_duality_and_mirror_symmetry_for_Grassmannians}. These are satisfied in our case because we work with the projectively normal minimal (Pl\"ucker) embeddings of homogeneous spaces $P\backslash G$ into the projectivizations of the corresponding irreducible representations $V(\dfwt[n])$, and because our Newton-Okounkov bodies do not contain any extra lattice points. We summarize this discussion in the following proposition.

\begin{prop}
  \label{prop:khovanskii}
  The Pl\"ucker coordinates give Khovanskii bases for the coordinate rings $\mathbb{C}[\cmX]$ with respect to the valuations $\nu_n$. 
\end{prop}

\begin{rem}
  Even though the Pl\"ucker coordinates give a finite Khovanskii basis for $\mathbb{C}[\cmX]$ with respect to the valuation $\nu_n$, this is not sufficient to guarantee that Algorithm \ref{alg:subduction-for-gen-minors} terminates for any input. If this is the case, then in the language of Definition 3.8 and Section 5 of \cite{Kaveh_Manon_Khovanskii_bases}, $\nu_n$ is called a \emph{subductive valuation} and corresponds to a \emph{prime cone} of the associated tropical variety. While it would be interesting to investigate this further, we do not pursue this here.
\end{rem}

None of the above statements depended on the exceptional types in a crucial way. For any homogeneous space $P\backslash G$, we may choose a reduced expression for $w^P$ which thus gives a parametrization of the corresponding $U_+^\circ$. We then define a valuation $\nu$ on $A=\mathbb{C}[P\backslash G]$ by expanding any element in the coordinates of $U_+^\circ$ and taking its degree-lexicographic minimal term, and this valuation will have one-dimensional leaves by an analogous proof to the proof of Lemma \ref{lem:one-dim-leaves}. We have not verified it ourselves, but we believe the following statement holds \emph{mutatis mutandis}:

\begin{conj}
  \label{conj:nobody}
  Let $X=P\backslash G$ be a minuscule homogeneous space with coordinate ring $A$ and valuation $\nu$ associated to a choice of reduced expression for $w^P$ as defined above. Then the Pl\"ucker coordinates of $X$ form a Khovanskii basis for $A$, and the associated Newton-Okounkov body $\Delta(A,\nu)$ is a rational polytope which is equal to the convex hull of the valuations of Pl\"ucker coordinates. 
\end{conj}

When $X$ is not minuscule, the Pl\"ucker coordinates are more difficult to understand, the valuations of the Pl\"ucker coordinates will no longer be $0/1$-vectors, and the Newton-Okounkov body may be less well-behaved. However, it would still be interesting to investigate to what extent the above conjecture holds. Moreover if $X$ is neither minuscule nor cominuscule, $w^P$ is no longer \emph{fully commutative}, so that different reduced expressions may not be related by swapping adjacent transpositions, and it would be interesting to investigate the relationship between the Newton-Okounkov bodies obtained for different reduced expressions, and whether there is any relation to cluster mutation. 
\bibliographystyle{amsalpha}
\bibliography{Biblio}%\addcontentsline{toc}{chapter}{Bibliography}
%There are \arabic{marginboxes} unresolved issues.
\addresseshere
\newpage
\appendix
\section{Torus expansion example}
\label{sec:expansion-appendix}

We give an example torus expansion computation for the Freudenthal variety, computing the expansion of the generalized minor $\minor_{\dfwt[3],u(\dfwt[3])}$, where $u=(w_0)_{\ge 23}w_0$ and $(w_0)_{\ge 23}=s_{\ellwo}\dots s_{23}$. We represent the $8645$-dimensional representation $V(\dfwt[3])$ as a subrepresentation of the alternating square of the adjoint representation $\bigwedge^{\!2} V(\dfwt[1])\cong V(\dfwt[3])\oplus V(\dfwt[1])$.

The action of $u$ on $\dfwt[3]$ gives the coweight vector $\wtu=(0,1,0,0,0,1,-3)$ (written in the basis of fundamental coweights), and the corresponding action on $\hwt[3]$ gives $\wtv{\wtu}\in V(\dfwt[3])$. In principle, we then compute all sequences $(i_1,\dots, i_k)$ such that $\dChe_{i_1,\dots, i_k}\cdot\wtv{\wtu}$ is a multiple of $\hwt[3]$, check which sequences appear in the expansion of $u_+$, and then collect the coefficients. However, there are too many sequences to compute directly and most of them do not appear in $u_+$. In our computations, we use this to facilitate computations, and we determine that only $43$ sequences have nonzero contributions (we suppress the commas and brackets, i.e.~we write $345\ldots$ instead of $(3,4,5,\ldots)$):
\[
\begin{array}{lll}
34561342574315423456677&
34561327443155423456677&
34561342743155423456677 \\
34561274331544234556677&
34561327431544234556677&
34561327443154234556677 \\
34561342743154234556677&
34561274331542344556677&
34561327431542344556677 \\
34562743115423344556677&
34561274315423344556677&
34561342543154234566777 \\
34561342431554234566777&
34561324431554234566777&
34513424316554234566777 \\
34513244316554234566777&
34132544316554234566777&
34561243315442345566777 \\
34561324315442345566777&
34512433165442345566777&
34125433165442345566777 \\
31425433165442345566777&
34513243165442345566777&
34132543165442345566777 \\
34561342431542345566777&
34561324431542345566777&
34513424316542345566777 \\
34513244316542345566777&
34132544316542345566777&
34561243315423445566777 \\
34561324315423445566777&
34512433165423445566777&
34125433165423445566777 \\
31425433165423445566777&
34513243165423445566777&
34132543165423445566777 \\
34562431154233445566777&
34561243154233445566777&
34524311654233445566777 \\
\!\framebox{\!$34254311654233445566777$\!}\!&
34512431654233445566777&
34125431654233445566777 \\
31425431654233445566777.
\end{array}
\]
As an example, we consider the sequence $(3,4,2,5,4,3,1,1,6,5,4,2,3,3,4,4,5,5,6,6,7,7,7)$, which we boxed above, to determine which monomials this sequence contributes to the torus expansion. Corresponding to our choice of $\wP$, we have $u_+$ of the form:
  \[
	\{\dx_{r_{1}}(a_1)\dx_{r_{2}}(a_2)\cdots\dx_{r_{27}}(a_{27})~|~a_i\in\C^*\},
	\]
  where we fixed $(r_1,r_2,\ldots,r_{27})=(7,6,5,4,3,2,4,5,6,1,3,4,2,5,7,4,3,1,6,5,4,2,3,4,5,6,7)$. From this we find by inspection one term of $u_+$ which sends $\wtv{\wtu}$ to $\hwt[3]$. The coefficient of this term is the following polynomial in torus coordinates:
  \begin{align*}
    &a_{5}a_{7}a_{13}a_{14}a_{16}a_{17}a_{18}^2a_{19}a_{20}a_{21}a_{22}a_{23}^2a_{24}^2a_{25}^2a_{26}^2a_{27}^3\\
    &+a_{5}a_{12}a_{13}a_{14}a_{16}a_{17}a_{18}^2a_{19}a_{20}a_{21}a_{22}a_{23}^2a_{24}^2a_{25}^2a_{26}^2a_{27}^3\\
    &+a_{11}a_{12}a_{13}a_{14}a_{16}a_{17}a_{18}^2a_{19}a_{20}a_{21}a_{22}a_{23}^2a_{24}^2a_{25}^2a_{26}^2a_{27}^3\\
    &=(a_5a_7+a_5a_{12}+a_{11}a_{12})a_{13}a_{14}a_{16}a_{17}a_{18}^2a_{19}a_{20}a_{21}a_{22}a_{23}^2a_{24}^2a_{25}^2a_{26}^2a_{27}^3.
  \end{align*}
  This polynomial is obtained as follows from $u_+$. First, we box those terms in the expansion of each $\dx_{r_i}(a_i)$ which must be taken to obtain the desired $\dChe_{i_1\dots i_k}$. 
  \begin{align*}
    &(\framebox{$1$}+a_{1}\dChe_7+\dots) (\framebox{$1$}+a_{2}\dChe_6+\dots) (\framebox{$1$}+a_{3}\dChe_5+\dots) (\framebox{$1$}+a_{4}\dChe_4+\dots) (1+a_{5}\dChe_3+\dots)\\
    &(\framebox{$1$}+a_{6}\dChe_2+\dots) (1+a_{7}\dChe_4+\dots) (\framebox{$1$}+a_{8}\dChe_5+\dots) (\framebox{$1$}+a_{9}\dChe_6+\dots) (\framebox{$1$}+a_{10}\dChe_1+\dots)\\
    &(1+a_{11}\dChe_3+\dots) (1+a_{12}\dChe_4+\dots)(1+\framebox{$a_{13}\dChe_2$}+\dots) (1+\framebox{$a_{14}\dChe_5$}+\dots) (\framebox{$1$}+a_{15}\dChe_7+\dots)\\
    &(1+\framebox{$a_{16}\dChe_4$}+\dots) (1+\framebox{$a_{17}\dChe_3$}+\dots) (1+a_{18}\dChe_1+\framebox{$\tfrac1{2!}{a_{18}^2(\dChe_1)^2}$}+\dots)\\
    &(1+\framebox{$a_{19}\dChe_6$}+\dots) (1+\framebox{$a_{20}\dChe_5$}+\dots)(1+\framebox{$a_{21}\dChe_4$}+\dots) (1+\framebox{$a_{22}\dChe_2$}+\dots)\\
    &(1+a_{23}\dChe_3+\framebox{$\tfrac1{2!}{a_{23}^2(\dChe_3)^2}$}+\dots)(1+a_{24}\dChe_4+\framebox{$\tfrac1{2!}{a_{24}^2(\dChe_4)^2}$}+\dots)\\
    &(1+a_{25}\dChe_5+\framebox{$\tfrac1{2!}{a_{25}^2(\dChe_5)^2}$}+\dots)(1+a_{26}\dChe_6+\framebox{$\tfrac1{2!}{a_{26}^2(\dChe_6)^2}$}+\dots)\\
    &(1+a_{27}\dChe_7+\tfrac1{2!}{a_{27}^2(\dChe_7)^2}+\framebox{$\tfrac1{3!}{a_{27}^3(\dChe_7)^3}$}+\dots)
  \end{align*}
  Removing the unused terms gives:
  \begin{align*}
    &(1)(1)(1)(1)(1+a_{5}\dChe_3+\dots)(1)(1+a_{7}\dChe_4+\dots)(1)(1)(1)(1+a_{11}\dChe_3+\dots)\\
    &(1+a_{12}\dChe_4+\dots)(a_{13}\dChe_2)(a_{14}\dChe_5)(1)(a_{16}\dChe_4) (a_{17}\dChe_3)\Bigl(\tfrac1{2!}{a_{18}^2(\dChe_1)^2}\Bigr)(a_{19}\dChe_6)(a_{20}\dChe_5)\\
    &(a_{21}\dChe_4)(a_{22}\dChe_2)\Bigl(\tfrac1{2!}{a_{23}^2(\dChe_3)^2}\Bigr)\Bigl(\tfrac1{2!}{a_{24}^2(\dChe_4)^2}\Bigr)\Bigl(\tfrac1{2!}{a_{25}^2(\dChe_5)^2}\Bigr)\Bigl(\tfrac1{2!}{a_{26}^2(\dChe_6)^2}\Bigr)\Bigl(\tfrac1{3!}{a_{27}^3(\dChe_7)^3}\Bigr),
  \end{align*}
  and we simplify to obtain:
  \begin{align*}
    &=(1+a_{5}\dChe_3+\dots)(1+a_{7}\dChe_4+\dots)(1+a_{11}\dChe_3+\dots)(1+a_{12}\dChe_4+\dots)\\
    &\left(\tfrac{1}{96}a_{13}a_{14}a_{16}a_{17}a_{18}^2a_{19}a_{20}a_{21}a_{22}a_{23}^2a_{24}^2a_{25}^2a_{26}^2a_{27}^3(\dChe_{25431^265423^24^25^26^27^3})\right).
  \end{align*}
  Expanding the first four terms above up to second order gives:
  \begin{align*}
    1&+(a_{5}+a_{11})\dChe_3+(a_{7}+a_{12})\dChe_4+(a_5a_{11})\dChe_{3^2}+(a_7a_{12})\dChe_{4^2}\\
     &+\tfrac{1}{4}\framebox{$\!(a_5a_7+a_5a_{12}+a_{11}a_{12})\!$}\dChe_{34}+\tfrac{1}{4}(a_7a_{11})\dChe_{43}+\dots,
  \end{align*}
  and we have boxed the relevant term in this last formula. Altogether, we obtain:
  \[
	\Bigl(\tfrac1{384}{a_5a_7+a_5a_{12}+a_{11}a_{12})a_{13}a_{14}a_{16}a_{17}a_{18}^2a_{19}a_{20}a_{21}a_{22}a_{23}^2a_{24}^2a_{25}^2a_{26}^2a_{27}^3}\Bigr)\dChe_{3425431^265423^24^25^26^27^3},
	\]
  and we compute $\frac{1}{384}\dChe_{3425431^265423^24^25^26^27^3}v_\varpi^+=\hwt[3]$, so that the contribution of this term is
  \begin{align*}
    &\Bigl\lan(a_5a_7+a_5a_{12}+a_{11}a_{12})a_{13}a_{14}a_{16}a_{17}a_{18}^2a_{19}a_{20}a_{21}a_{22}a_{23}^2a_{24}^2a_{25}^2a_{26}^2a_{27}^3 \hwt[3],\hwt[3] \Bigr\ran\\
    &=(a_5a_7+a_5a_{12}+a_{11}a_{12})a_{13}a_{14}a_{16}a_{17}a_{18}^2a_{19}a_{20}a_{21}a_{22}a_{23}^2a_{24}^2a_{25}^2a_{26}^2a_{27}^3.
  \end{align*}
  The remaining $42$ sequences yield analogous computations, though they only contribute one term each to the resulting polynomial which thus has in total 45 terms (the final three terms are the ones we computed explicitly above):

\begingroup\small
\[
  \begin{array}{l}
    a_{5}a_{7}a_{8}a_{9}a_{10}a_{11}a_{12}a_{13}a_{14}a_{15}a_{16}a_{17}a_{18}a_{20}a_{21}a_{22}a_{23}a_{24}a_{25}a_{26}^2a_{27}^2 \\
		{}+a_{5}a_{7}a_{8}a_{9}a_{10}a_{11}a_{12}a_{13}a_{15}a_{16}a_{17}a_{18}a_{20}^2a_{21}a_{22}a_{23}a_{24}a_{25}a_{26}^2a_{27}^2\\
    {}+a_{5}a_{7}a_{8}a_{9}a_{10}a_{11}a_{13}a_{15}a_{16}^2a_{17}a_{18}a_{20}^2a_{21}a_{22}a_{23}a_{24}a_{25}a_{26}^2a_{27}^2 \\
		{}+a_{5}a_{7}a_{8}a_{9}a_{10}a_{11}a_{12}a_{13}a_{15}a_{16}a_{17}a_{18}a_{20}a_{21}a_{22}a_{23}a_{24}a_{25}^2a_{26}^2a_{27}^2\\
    {}+a_{5}a_{7}a_{8}a_{9}a_{10}a_{11}a_{13}a_{15}a_{16}^2a_{17}a_{18}a_{20}a_{21}a_{22}a_{23}a_{24}a_{25}^2a_{26}^2a_{27}^2 \\
		{}+a_{5}a_{7}a_{8}a_{9}a_{10}a_{11}a_{13}a_{15}a_{16}a_{17}a_{18}a_{20}a_{21}^2a_{22}a_{23}a_{24}a_{25}^2a_{26}^2a_{27}^2\\
    {}+a_{5}a_{7}a_{8}a_{9}a_{10}a_{13}a_{15}a_{16}a_{17}^2a_{18}a_{20}a_{21}^2a_{22}a_{23}a_{24}a_{25}^2a_{26}^2a_{27}^2 \\
		{}+a_{5}a_{7}a_{8}a_{9}a_{10}a_{11}a_{13}a_{15}a_{16}a_{17}a_{18}a_{20}a_{21}a_{22}a_{23}a_{24}^2a_{25}^2a_{26}^2a_{27}^2\\
    {}+a_{5}a_{7}a_{8}a_{9}a_{10}a_{13}a_{15}a_{16}a_{17}^2a_{18}a_{20}a_{21}a_{22}a_{23}a_{24}^2a_{25}^2a_{26}^2a_{27}^2 \\
		{}+a_{5}a_{7}a_{8}a_{9}a_{10}a_{13}a_{15}a_{16}a_{17}a_{18}a_{20}a_{21}a_{22}a_{23}^2a_{24}^2a_{25}^2a_{26}^2a_{27}^2\\
    {}+a_{5}a_{7}a_{8}a_{9}a_{13}a_{15}a_{16}a_{17}a_{18}^2a_{20}a_{21}a_{22}a_{23}^2a_{24}^2a_{25}^2a_{26}^2a_{27}^2 \\
		{}+a_{5}a_{7}a_{8}a_{9}a_{10}a_{11}a_{12}a_{13}a_{14}a_{16}a_{17}a_{18}a_{20}a_{21}a_{22}a_{23}a_{24}a_{25}a_{26}^2a_{27}^3\\
    {}+a_{5}a_{7}a_{8}a_{9}a_{10}a_{11}a_{12}a_{13}a_{16}a_{17}a_{18}a_{20}^2a_{21}a_{22}a_{23}a_{24}a_{25}a_{26}^2a_{27}^3 \\
		{}+a_{5}a_{7}a_{8}a_{9}a_{10}a_{11}a_{13}a_{16}^2a_{17}a_{18}a_{20}^2a_{21}a_{22}a_{23}a_{24}a_{25}a_{26}^2a_{27}^3\\
    {}+a_{5}a_{7}a_{8}a_{10}a_{11}a_{12}a_{13}a_{16}a_{17}a_{18}a_{19}a_{20}^2a_{21}a_{22}a_{23}a_{24}a_{25}a_{26}^2a_{27}^3 \\
		{}+a_{5}a_{7}a_{8}a_{10}a_{11}a_{13}a_{16}^2a_{17}a_{18}a_{19}a_{20}^2a_{21}a_{22}a_{23}a_{24}a_{25}a_{26}^2a_{27}^3\\
    {}+a_{5}a_{7}a_{10}a_{11}a_{13}a_{14}a_{16}^2a_{17}a_{18}a_{19}a_{20}^2a_{21}a_{22}a_{23}a_{24}a_{25}a_{26}^2a_{27}^3 \\
		{}+a_{5}a_{7}a_{8}a_{9}a_{10}a_{11}a_{12}a_{13}a_{16}a_{17}a_{18}a_{20}a_{21}a_{22}a_{23}a_{24}a_{25}^2a_{26}^2a_{27}^3\\
    {}+a_{5}a_{7}a_{8}a_{9}a_{10}a_{11}a_{13}a_{16}^2a_{17}a_{18}a_{20}a_{21}a_{22}a_{23}a_{24}a_{25}^2a_{26}^2a_{27}^3 \\
		{}+a_{5}a_{7}a_{8}a_{10}a_{11}a_{12}a_{13}a_{16}a_{17}a_{18}a_{19}a_{20}a_{21}a_{22}a_{23}a_{24}a_{25}^2a_{26}^2a_{27}^3\\
    {}+a_{5}a_{7}a_{8}a_{10}a_{11}a_{13}a_{16}^2a_{17}a_{18}a_{19}a_{20}a_{21}a_{22}a_{23}a_{24}a_{25}^2a_{26}^2a_{27}^3 \\
		{}+a_{5}a_{7}a_{10}a_{11}a_{13}a_{14}a_{16}^2a_{17}a_{18}a_{19}a_{20}a_{21}a_{22}a_{23}a_{24}a_{25}^2a_{26}^2a_{27}^3\\
    {}+a_{5}a_{7}a_{8}a_{9}a_{10}a_{11}a_{13}a_{16}a_{17}a_{18}a_{20}a_{21}^2a_{22}a_{23}a_{24}a_{25}^2a_{26}^2a_{27}^3 \\
		{}+a_{5}a_{7}a_{8}a_{9}a_{10}a_{13}a_{16}a_{17}^2a_{18}a_{20}a_{21}^2a_{22}a_{23}a_{24}a_{25}^2a_{26}^2a_{27}^3\\
    {}+a_{5}a_{7}a_{8}a_{10}a_{11}a_{13}a_{16}a_{17}a_{18}a_{19}a_{20}a_{21}^2a_{22}a_{23}a_{24}a_{25}^2a_{26}^2a_{27}^3 \\
		{}+a_{5}a_{7}a_{10}a_{11}a_{13}a_{14}a_{16}a_{17}a_{18}a_{19}a_{20}a_{21}^2a_{22}a_{23}a_{24}a_{25}^2a_{26}^2a_{27}^3\\
    {}+a_{5}a_{7}a_{8}a_{10}a_{13}a_{16}a_{17}^2a_{18}a_{19}a_{20}a_{21}^2a_{22}a_{23}a_{24}a_{25}^2a_{26}^2a_{27}^3 \\
		{}+a_{5}a_{7}a_{10}a_{13}a_{14}a_{16}a_{17}^2a_{18}a_{19}a_{20}a_{21}^2a_{22}a_{23}a_{24}a_{25}^2a_{26}^2a_{27}^3\\
    {}+a_{5}a_{10}a_{12}a_{13}a_{14}a_{16}a_{17}^2a_{18}a_{19}a_{20}a_{21}^2a_{22}a_{23}a_{24}a_{25}^2a_{26}^2a_{27}^3 \\
		{}+a_{5}a_{7}a_{8}a_{9}a_{10}a_{11}a_{13}a_{16}a_{17}a_{18}a_{20}a_{21}a_{22}a_{23}a_{24}^2a_{25}^2a_{26}^2a_{27}^3\\
    {}+a_{5}a_{7}a_{8}a_{9}a_{10}a_{13}a_{16}a_{17}^2a_{18}a_{20}a_{21}a_{22}a_{23}a_{24}^2a_{25}^2a_{26}^2a_{27}^3 \\
		{}+a_{5}a_{7}a_{8}a_{10}a_{11}a_{13}a_{16}a_{17}a_{18}a_{19}a_{20}a_{21}a_{22}a_{23}a_{24}^2a_{25}^2a_{26}^2a_{27}^3\\
    {}+a_{5}a_{7}a_{10}a_{11}a_{13}a_{14}a_{16}a_{17}a_{18}a_{19}a_{20}a_{21}a_{22}a_{23}a_{24}^2a_{25}^2a_{26}^2a_{27}^3 \\
		{}+a_{5}a_{7}a_{8}a_{10}a_{13}a_{16}a_{17}^2a_{18}a_{19}a_{20}a_{21}a_{22}a_{23}a_{24}^2a_{25}^2a_{26}^2a_{27}^3\\
    {}+a_{5}a_{7}a_{10}a_{13}a_{14}a_{16}a_{17}^2a_{18}a_{19}a_{20}a_{21}a_{22}a_{23}a_{24}^2a_{25}^2a_{26}^2a_{27}^3 \\
		{}+a_{5}a_{10}a_{12}a_{13}a_{14}a_{16}a_{17}^2a_{18}a_{19}a_{20}a_{21}a_{22}a_{23}a_{24}^2a_{25}^2a_{26}^2a_{27}^3\\
    {}+a_{5}a_{7}a_{8}a_{9}a_{10}a_{13}a_{16}a_{17}a_{18}a_{20}a_{21}a_{22}a_{23}^2a_{24}^2a_{25}^2a_{26}^2a_{27}^3 \\
		{}+a_{5}a_{7}a_{8}a_{9}a_{13}a_{16}a_{17}a_{18}^2a_{20}a_{21}a_{22}a_{23}^2a_{24}^2a_{25}^2a_{26}^2a_{27}^3\\
    {}+a_{5}a_{7}a_{8}a_{10}a_{13}a_{16}a_{17}a_{18}a_{19}a_{20}a_{21}a_{22}a_{23}^2a_{24}^2a_{25}^2a_{26}^2a_{27}^3 \\
		{}+a_{5}a_{7}a_{10}a_{13}a_{14}a_{16}a_{17}a_{18}a_{19}a_{20}a_{21}a_{22}a_{23}^2a_{24}^2a_{25}^2a_{26}^2a_{27}^3\\
    {}+a_{5}a_{10}a_{12}a_{13}a_{14}a_{16}a_{17}a_{18}a_{19}a_{20}a_{21}a_{22}a_{23}^2a_{24}^2a_{25}^2a_{26}^2a_{27}^3 \\
		{}+a_{5}a_{7}a_{8}a_{13}a_{16}a_{17}a_{18}^2a_{19}a_{20}a_{21}a_{22}a_{23}^2a_{24}^2a_{25}^2a_{26}^2a_{27}^3\\
    {}+a_{5}a_{7}a_{13}a_{14}a_{16}a_{17}a_{18}^2a_{19}a_{20}a_{21}a_{22}a_{23}^2a_{24}^2a_{25}^2a_{26}^2a_{27}^3 \\
		{}+a_{5}a_{12}a_{13}a_{14}a_{16}a_{17}a_{18}^2a_{19}a_{20}a_{21}a_{22}a_{23}^2a_{24}^2a_{25}^2a_{26}^2a_{27}^3\\
    {}+a_{11}a_{12}a_{13}a_{14}a_{16}a_{17}a_{18}^2a_{19}a_{20}a_{21}a_{22}a_{23}^2a_{24}^2a_{25}^2a_{26}^2a_{27}^3.
  \end{array}
\]\endgroup

\section{Pl\"ucker coordinate valuations}
\label{appendix:vals}
Below are the valuations of the Pl\"ucker coordinates for the Cayley plane (we again write $001\ldots$ for $(0,0,1,\ldots)$):
\begingroup\small
\[
\begin{array}{lll}
\nu_6(p_{0\hphantom{1}}\!\!)\,\, = 0000000000000000, &
\nu_6(p_{1\hphantom{1}}\!\!)\,\, = 0000000000000001, &
\nu_6(p_{2\hphantom{1}}\!\!)\,\, = 0000000000000011, \\
\nu_6(p_{3\hphantom{1}}\!\!)\,\, = 0000000000000111, &
\nu_6(p_{4\hphantom{1}}'\!\!)\,\, = 0000000000001111, &
\nu_6(p_{4\hphantom{1}}''\!\!)\,\, = 0000000000010111, \\
\nu_6(p_{5\hphantom{1}}'\!\!)\,\, = 0000000000011111, &
\nu_6(p_{5\hphantom{1}}''\!\!)\,\, = 0000000100010111, &
\nu_6(p_{6\hphantom{1}}'\!\!)\,\, = 0000000000111111, \\
\nu_6(p_{6\hphantom{1}}''\!\!)\,\, = 0000000100011111, &
\nu_6(p_{7\hphantom{1}}'\!\!)\,\, = 0000000001111111, &
\nu_6(p_{7\hphantom{1}}''\!\!)\,\, = 0000000100111111, \\
\nu_6(p_{8\hphantom{1}}\!\!)\,\, = 0000000011111111, &
\nu_6(p_{8\hphantom{1}}'\!\!)\,\, = 0000000101111111, &
\nu_6(p_{8\hphantom{1}}''\!\!)\,\, = 0000001100111111, \\
\nu_6(p_{9\hphantom{1}}'\!\!)\,\, = 0000000111111111, &
\nu_6(p_{9\hphantom{1}}''\!\!)\,\, = 0000001101111111, &
\nu_6(p_{10}') = 0000001111111111, \\
\nu_6(p_{10}'') = 0000011101111111, &
\nu_6(p_{11}') = 0000011111111111, &
\nu_6(p_{11}'') = 0001011101111111, \\
\nu_6(p_{12}') = 0000111111111111, &
\nu_6(p_{12}'') = 0001011111111111, &
\nu_6(p_{13}) = 0001111111111111, \\
\nu_6(p_{14}) = 0011111111111111, &
\nu_6(p_{15}) = 0111111111111111, &
\nu_6(p_{16}) = 1111111111111111.
\end{array}
\]\endgroup
The convex hull of the above points forms $\Delta(\mathbb{C}[\cmX],\nu_6)$. We compute with Polymake \cite{polymake-2000,polymake-2017} that this is a $16$-dimensional polytope with normalized volume $78$ and f-vector:
\[
(27,297,1858,7598,21884,46415,74521,92095,88372,65979,38160,16900,5612,1349,221,22).
\]

Below are the valuations of the Pl\"ucker coordinates for the Freudenthal variety:
\begingroup\small
\[
\begin{array}{ll}
  \nu_7(p_{0\hphantom{1}}\!\!)\,\,=000000000000000000000000000, &
  \nu_7(p_{1\hphantom{1}}\!\!)\,\,=000000000000000000000000001, \\ 
  \nu_7(p_{2\hphantom{1}}\!\!)\,\,=000000000000000000000000011, & 
  \nu_7(p_{3\hphantom{1}}\!\!)\,\,=000000000000000000000000111, \\ 
  \nu_7(p_{4\hphantom{1}}\!\!)\,\,=000000000000000000000001111, & 
  \nu_7(p_{5\hphantom{1}}'\!\!)\,\,=000000000000000000000101111, \\ 
  \nu_7(p_{5\hphantom{1}}''\!\!)\,\,=000000000000000000000011111, & 
  \nu_7(p_{6\hphantom{1}}'\!\!)\,\,=000000000000000000000111111, \\ 
  \nu_7(p_{6\hphantom{1}}''\!\!)\,\,=000000000000000001000011111, & 
  \nu_7(p_{7\hphantom{1}}''\!\!)\,\,=000000000000000001000111111, \\ 
  \nu_7(p_{7\hphantom{1}}'\!\!)\,\,=000000000000000000001111111, & 
  \nu_7(p_{8\hphantom{1}}''\!\!)\,\,=000000000000000001001111111, \\ 
  \nu_7(p_{8\hphantom{1}}'\!\!)\,\,=000000000000000000011111111, & 
  \nu_7(p_{9\hphantom{1}}''\!\!)\,\,=000000000000000011001111111, \\ 
  \nu_7(p_{9\hphantom{1}}'\!\!)\,\,=000000000000000001011111111, & 
  \nu_7(p_{9\hphantom{1}}\!\!)\,\,=000000000000000000111111111, \\ 
  \nu_7(p_{10}'')=000000000000000011011111111, & 
  \nu_7(p_{10}')=000000000000000001111111111, \\ 
  \nu_7(p_{10})=000000000000100000111111111, & 
  \nu_7(p_{11}'')=000000000000000111011111111, \\ 
  \nu_7(p_{11}')=000000000000000011111111111, & 
  \nu_7(p_{11})=000000000000100001111111111, \\ 
  \nu_7(p_{12}'')=000000000000001111011111111, & 
  \nu_7(p_{12}')=000000000000000111111111111, \\ 
  \nu_7(p_{12})=000000000000100011111111111, & 
  \nu_7(p_{13}'')=000000000000001111111111111, \\ 
  \nu_7(p_{13}')=000000000000010111111111111, & 
  \nu_7(p_{13})=000000000000100111111111111, \\ 
  \nu_7(p_{14})=000000000000011111111111111, & 
  \nu_7(p_{14}')=000000000000101111111111111, \\ 
  \nu_7(p_{14}'')=000000000000110111111111111, & 
  \nu_7(p_{15})=000000000001011111111111111, \\ 
  \nu_7(p_{15}')=000000000000111111111111111, & 
  \nu_7(p_{15}'')=000000001000110111111111111, \\ 
  \nu_7(p_{16})=000000000011011111111111111, & 
  \nu_7(p_{16}')=000000000001111111111111111, \\
  \nu_7(p_{16}'')=000000001000111111111111111, & 
  \nu_7(p_{17})=000000000111011111111111111, \\
  \nu_7(p_{17}')=000000000011111111111111111, & 
  \nu_7(p_{17}'')=000000001001111111111111111, \\ 
  \nu_7(p_{18})=000000000111111111111111111, & 
  \nu_7(p_{18}')=000000001011111111111111111, \\ 
  \nu_7(p_{18}'')=000000011001111111111111111, & 
  \nu_7(p_{19}')=000000001111111111111111111, \\ 
  \nu_7(p_{19}'')=000000011011111111111111111, & 
  \nu_7(p_{20}')=000000011111111111111111111, \\ 
  \nu_7(p_{20}'')=000000111011111111111111111, & 
  \nu_7(p_{21}')=000000111111111111111111111, \\ 
  \nu_7(p_{21}'')=000001111011111111111111111, & 
  \nu_7(p_{22}'')=000001111111111111111111111, \\ 
  \nu_7(p_{22}')=000010111111111111111111111, & 
  \nu_7(p_{23})=000011111111111111111111111, \\ 
  \nu_7(p_{24})=000111111111111111111111111, & 
  \nu_7(p_{25})=001111111111111111111111111, \\ 
  \nu_7(p_{26})=011111111111111111111111111, & 
  \nu_7(p_{27})=111111111111111111111111111.
\end{array}
\]\endgroup

The convex hull of the above points forms $\Delta(\C[\cmX],\nu_7)$. We compute with Polymake (\cite{polymake-2000, polymake-2017}) that this is a $27$-dimensional polytope with volume $13110$. We were unable to compute the f-vector due to computational limitations (our computations exceeded 100GB of RAM).

\end{document}